\documentclass[pdf]{article}

\usepackage{pdfpages}
\usepackage{hyperref}
\usepackage{xcolor}
\usepackage[utf8]{inputenc}
\usepackage{amsmath}
\usepackage{amsfonts}
\usepackage{amssymb}
\usepackage{amsthm}
\usepackage{marginnote}
\usepackage{enumitem}
\usepackage{graphicx}
\usepackage[normalem]{ulem}
\usepackage{fancyhdr}
\usepackage{etoolbox}
\newcommand{\figures}[1]{}
	\usepackage{pst-node}
	\usepackage{tikz-cd}

\newcommand{\Iref}[1]{%
	\ifstrequal{#1}{thm:GHPcompact}{3.5}{%
		\ifstrequal{#1}{lem:infinity}{3.13}{%
			\ifstrequal{#1}{lem:GHP-monotone}{3.8}{%
				\ifstrequal{#1}{lem:GHP-subsets}{5.3}{%
					\ifstrequal{#1}{thm:GHP-Metric}{3.10}{%
						\ifstrequal{#1}{thm:GHP-precompactness}{3.20}{%
							\ifstrequal{#1}{thm:GHPcomplete}{3.19}{%
								\ifstrequal{#1}{thm:convergence}{3.16}{%
									\ifstrequal{#1}{cor:strassen-integer}{2.3}{%
										\ifstrequal{#1}{thm:strassen}{2.1}{%
											\ifstrequal{#1}{subsec:randomMeasure}{4.2}{%
												???}}}}}}}}}}}}
\newcommand{\defstyle}[1]{\textbf{#1}}
\newcommand{\myprob}[1]{\mathbb P \left[ #1 \right]}

\newcommand{\norm}[1]{\left| #1 \right|}

\newcommand{\tv}[2]{||#1-#2||}

\newcommand{\bs}[1]{\boldsymbol{#1}}

\newcommand{\ali}[1]{{#1}}
\newcommand{\new}[1]{{#1}}
\newcommand{\del}[1]{}
\newcommand{\correction}[1]{{\color{black} #1}}
\newcommand{\unwritten}[1]{}

\newcommand{\pmm}{PMM}
\newcommand{\pms}{PM}

\newcommand{\diam}{\mathrm{diam}}
\newcommand{\distortion}{\mathrm{dis}}
\newcommand{\oball}[2]{B_{#1}(#2)}
\newcommand{\cball}[2]{\overline{B}_{#1}(#2)}
\newcommand{\pcball}[2]{\overline{#1}^{(#2)}}

\newcommand{\cnei}[2]{{N}_{#1}(#2)}

\newcommand{\ghp}{d_{GHP}}
\newcommand{\cghp}{d^c_{GHP}}
\newcommand{\gh}{d_{GH}}
\newcommand{\cgh}{d^c_{GH}}
\newcommand{\hausdorff}{d_H}
\newcommand{\prokhorov}{d_P}

\newcommand{\cgf}{d^c_{\tau}}
\newcommand{\gf}{d_{\tau}}
\newcommand{\dsup}{d_{\mathrm{sup}}}
\newcommand{\restrict}[2]{{
		\left.\kern-\nulldelimiterspace 
		#1 
		\vphantom{\big|} 
		\right|_{#2} 
}}

\newcommand{\mc}{\mathfrak M}
\newcommand{\mwstar}[1]{\mathfrak M_{#1}}
\newcommand{\mwcstar}[1]{\mathfrak M^c_{#1}}

\newcommand{\bcm}{\textsf{bcm}}

\theoremstyle{theorem}
\newtheorem{theorem}{Theorem}[section]
\newtheorem{lemma}[theorem]{Lemma}
\newtheorem{proposition}[theorem]{Proposition}
\newtheorem{corollary}[theorem]{Corollary}

\newtheorem{problem}[theorem]{Problem}

\theoremstyle{definition}
\newtheorem{definition}[theorem]{Definition}

\newtheorem{assumption}[theorem]{Assumption}
\newtheorem{example}[theorem]{Example}
\theoremstyle{definition}
\newtheorem{remark}[theorem]{Remark}

\theoremstyle{theorem}

\begin{document}
	
	\title{A Unified Framework for Generalizing the Gromov-Hausdorff Metric}
	\author{Ali Khezeli\footnote{INRIA Paris, ali.khezeli@inria.fr}}
	
	\maketitle
	
	\begin{abstract}
		In this paper, \new{an approach for generalizing the Gromov-Hausdorff metric is presented, which applies to} metric spaces equipped with some additional structure. A special case is the Gromov-Hausdorff-Prokhorov metric between measured metric spaces. This abstract framework unifies several existing Gromov-Hausdorff-type metrics for metric spaces equipped with a measure, a point, a closed subset, a curve, a tuple of such structures, etc. \ali{Along with reviewing these special cases in the literature, several new examples are also presented. Two frameworks are provided, one for compact metric spaces and the other for boundedly-compact pointed metric spaces. In both cases, a Gromov-Hausdorff-type metric is defined and its topological properties are studied.} In particular, completeness and separability is proved under some conditions. This enables one to study random metric spaces equipped with additional structures, which is the main motivation of this work.
		\\
		\textbf{Keywords.} Gromov-Hausdorff metric, Gromov-Hausdorff-Prokhorov metric, 
		boundedly-compact, functor. 
	\end{abstract}

	\tableofcontents

	\section{Introduction}


\subsection{The Gromov-Hausdorff Metric}
\label{intro:gh}


How can one measure the similarity of two metric spaces?
Heuristically, the scaled lattice $\delta\mathbb Z^d=\{\delta x: x\in\mathbb Z^d \}$ {converges} to $\mathbb R^d$ as $\delta\to 0$. 
Gromov defined various notions of convergence of metric spaces in his novel book (\cite{bookGr81structuresmetriqueFrancais} and~\cite{bookGr99}), which was published in French in 1981. He used these notions to study limits of manifolds and Cayley graphs.   
In particular, he proved the celebrated result that 
a finitely-generated group has polynomial growth if and only if it is virtually nilpotent~\cite{Gr81}. Here, we are interested in Gromov's notion of \textit{Hausdorff convergence}, which is now known as \textit{Gromov-Hausdorff convergence}. This is induced by the \defstyle{Gromov-Hausdorff (GH) distance} between two metric spaces $X$ and $Y$ defined by 
\begin{equation}
	\label{eq:GH}
	\cgh(X,Y):=\inf\big\{ \hausdorff(f(X),g(Y))\big\},
\end{equation}
where the infimum is over all metric spaces $Z$ and all pairs of isometric embeddings $f:X\to Z$ and $g:Y\to Z$. Here, $d_H$ denotes the Hausdorff distance between subsets of $Z$. {In words, $X$ and $Y$ are close to each other if one can embed them isometrically in a common larger metric space such that their images are close to each other.} This notion is practically useful if $X$ and $Y$ are compact. Indeed, $\cgh$ is a metric on the space of compact metric spaces (if isometric metric spaces are considered identical). For non-compact metric spaces, Gromov defined another notion of convergence which will be discussed in Subsection~\ref{intro:non-compact}.


In fact, the Gromov-Hausdorff metric~\eqref{eq:GH} was defined before Gromov by Edwards~\cite{Ed75}. He also proved some basic properties of this metric, such as completeness and separability, which are often essential for probabilistic applications. However, the work of Edward was not (and still is not) well known and the works mentioned here do not refer to it. See also~\cite{Tu16WhoInvented}. 

Independently of Gromov's works, Aldous studied scaling limits of random trees in 1991 in his seminal paper~\cite{Al91-crtI}. Among all of the trees with $n$ given vertices, let $T_n$ be a tree chosen randomly and uniformly. Aldous showed that 
by scaling the graph-distance metric on $T_n$ by factor $1/\sqrt{n}$, the resulting tree \textit{converges} to some random fractal called the \textit{(Brownian) continuum random tree}. 
To formalize this convergence, he embedded the trees isometrically in the infinite-dimensional space $l_1$ and used the notion of weak convergence of random compact subsets of a given space~\cite{Al91-crtI}.
\ali{Later, this convergence was expressed equivalently\footnote{We didn't find this equivalence in the literature. However, we have verified that, for real trees, the two notions of convergence are equivalent. This can be seen by approximating every real tree with finite trees. The same holds for convergence of measured trees mentioned in the next paragraph.} using the Gromov-Hausdorff metric as follows
	(\cite{EvPiWi06}, see also \cite{DuLe05} and~\cite{Le05randomtrees}), which is more natural and can be applied to more models: Let $\mc$ be the space of compact metric spaces. The Gromov-Hausdorff metric $\cgh$ makes $\mc$ a complete and separable metric space. So, the convergence of the sequence $\frac 1{\sqrt n}T_n$ can be formalized by the theory of weak convergence of random elements of $\mc$\unwritten{ (one can forget the graph structure on the trees, since it is uniquely determined by the graph-distance metric)}. 
	Defining random real trees and their convergence is the first place where the Gromov-Hausdorff metric has been used in probability theory (\cite{EvPiWi06}, see also~\cite{bookEv08}). Since then, the Gromov-Hausdorff metric and its generalizations have been extensively used in the literature to study the scaling limits of various random discrete objects. 
} 


Aldous also strengthened the above convergence by proving that the uniform measure on the vertices of $\frac 1{\sqrt n}T_n$ converges to a probability measure on $l_1$. This was also expressed later (\cite{EvPiWi06}) in terms of weak convergence under the \defstyle{Gromov-Hausdorff-Prokhorov (GHP) metric}:
If $\mathcal X=(X,\mu)$ and $\mathcal Y=(Y,\nu)$ are two \textit{measured} metric spaces; {i.e., metric spaces $X$ and $Y$ equipped with finite measures $\mu$ and $\nu$,} define
%
\begin{equation}
	\label{eq:GHP}
	\cghp(\mathcal X,\mathcal Y):=\inf \{\hausdorff(f(X),g(Y))\vee \prokhorov(f_*\mu,g_*\nu)\}, 
\end{equation}
where the infimum is over all $Z,f,g$ as in~\eqref{eq:GH}. Here, $\vee$ means `maximum',  $f_*\mu$ and $g_*\nu$ denote the pushforwards of the measures $\mu$ and $\nu$ and $\prokhorov$ denotes the \textit{Prokhorov distance} of two probability measures on $Z$, recalled in~\eqref{eq:prokhorov} below. 

In~\cite{Al93crtIII}, Aldous studied the convergence of trees via the convergence of \textit{finite-dimensional distributions}: The distribution of the random tree spanned by $k$ randomly chosen vertices (and the paths connecting them) is convergent as $k$ is fixed and $n\to\infty$. Later, this 
was extended to (probability-) measured metric spaces and was called \textit{Gromov-weak convergence} in~\cite{GrPfWi09}. Also, it was shown to be equivalent to convergence under the \defstyle{Gromov-Prokhorov (GP) metric}:
\begin{equation}
	\label{eq:GP}
	d^c_{{GP}}(\mathcal X,\mathcal Y):=\inf \{\prokhorov(f_*\mu,g_*\nu)\}.
\end{equation}
This is weaker than GHP-convergence; e.g., it ignores the parts of $X$ and $Y$ that are not in the supports of $\mu$ and $\nu$. 
The GP metric will not be studied in this paper since GP-convergence does not imply the GH-convergence of the underlying metric spaces. 
See~\cite{AtLoWi16} for more discussion on various types of convergence of measured metric spaces.

In fact, the study of convergence of measured metric spaces goes back to Fukaya~\cite{Fu87collapsing} with the aim of using it to study the eigenvalues of the Laplace operator on Riemannian manifolds equipped with the volume measure (it has also been used to study optimal transport and Ricci curvature on measured metric spaces~\cite{LoVi09}). 
Gromov also defined the \textit{box metric} in the new chapter $3\frac 12$ added to the English version~\cite{bookGr99} of his book published in 1999. This metric was shown later to be bi-Lipschitz-equivalent to the GP metric~\cite{Lo13equivalence}. Some of the ideas in this chapter overlap with Aldous's work (e.g., the reconstruction theorem and Section~$3\frac 12.14$ are related to finite-dimensional distributions and their convergence), but it seems that the two works were completed independently.

\subsection{Generalizations in the Literature}

The idea of Gromov-Hausdorff convergence (or the associated metric) has been generalized in various papers to study the convergence (or perturbations) of objects of the type $(X,a)$, where $X$ is a compact metric space and $a$ is some additional structure on $X$ (the non-compact case will be discussed in Subsection~\ref{intro:non-compact}). We have already mentioned the case where the additional structure is a finite measure on $X$. 
\ali{The following are some other instances which will be discussed further in the paper. }

Gromov defined a metric between \textit{pointed} metric spaces, where the additional structure $a$ on $X$ is a selected point of $X$. This was used for the definition of convergence in the non-compact case. In~\cite{AdBrGoMi17}, the additional structure is a tuple of $k$ points of $X$, $l$ measures on $X$ or a combination of them. This is used in the study of the scaling limit of another class of random trees. 
The paper~\cite{Mi09} defines a metric on the set of compact metric spaces equipped with $k$ compact subsets. It is used in the study of the scaling limit of random quadrangulations. In~\cite{GwMi17}, the additional structure is a continuous curve. It is used for the scaling limit of random half-plane quadrangulations to keep track of the boundary curve. The paper \cite{AtLoWi17} uses the case where $a$ is a c\`adl\`ag curve or a c\`adl\`ag process in $X$ to study limits of random walks on graphs and some analogous stochastic processes defined on measured metric spaces. The latter only defines a notion of convergence, and not a metric. Various other generalizations will be discussed throughout the paper, mostly in Section~\ref{sec:special}.

The generalizations in the literature either define a GH-type metric or just a notion of convergence. Defining a metric is usually required in probabilistic applications where the underlying metric spaces are random, as explained in the next paragraph. The methods of defining such metrics will be discussed in Subsection~\ref{intro:framework}. Convergence is defined, for instance, in~\cite{Fu87collapsing} for measured metric spaces and in~\cite{AtLoWi17} for c\`adl\`ag processes. 
In such works, the convergence is defined based on one of the characterizations of the Gromov-Hausdorff convergence; e.g., by the notion of \textit{$\epsilon$-isometries} (skipped in this paper; see e.g., \cite{Fu87collapsing} or Chapter~7 of~\cite{bookBBI}) or by embedding a sequence of metric spaces in a common larger space (see e.g., Lemma~\ref{lem:common2} and also~\cite{AtLoWi17}). Indeed, such convergence can also be expressed by GH-type metrics discussed in the next subsection. We can also mention that Gromov has defined a notion of convergence (and not a metric) in the non-compact case as explained in Subsection~\ref{intro:non-compact}.

In the generalizations of the GH metric, further properties are needed for the desired application.
In the works where only convergence is defined, it is sometimes proved that the convergence gives rise to a notion of topology on the set $\mathcal C$ of (equivalence classes of) compact metric spaces equipped with an additional structure, and the resulting topological space is Hausdorff or separable (e.g., in~\cite{Fu87collapsing}). For probabilistic applications, it is usually shown that $\mathcal C$ is a separable metric space. This is convenient for using weak convergence (e.g., scaling limits) of random objects of the type $(X,a)$. For instance, separability allows one to use Prokhorov's theorem on tightness of probability measures on $\mathcal C$. In addition, proving completeness of $\mathcal C$ is useful if one desires to have a standard probability space. This enables one; e.g., to use the regular conditional distribution. However, completeness is not always necessary. This is similar to the case in the theory of (simple) point processes, where the state space is not complete. Instead, it is a Borel subset of some complete metric space (e.g., the space of locally finite measures on $\mathbb R^d$). Instances where completeness does not hold for GH-type metrics will be discussed throughout the paper. 
In addition, some convergence, pre-compactness or tightness criteria are also needed. Of course, an important task in applications is to actually prove the convergence in specific models, which is 
out of the scope of this paper. Other tasks that are of interest in the literature are studying the properties of the limiting space and using those properties to derive conclusions regarding the sequence of spaces converging to it (see e.g., \cite{objective}).

\subsection{A Unified Framework for Generalizations}
\label{intro:framework}

Many of the generalizations of the Gromov-Hausdorff metric in the literature have similar properties with similar proofs, but none of them are implied by the properties of the Gromov-Hausdorff metric. So the corresponding papers fall in one of these two classes: Some of them dedicate a lot of space to stating and proving the properties, which is technical and is usually a deviation from the main purpose of the paper. Some others only claim the statements without proof. The latter have some errors occasionally since there are many things to check and the proofs are technical. In particular, the claim of completeness is wrong in some papers, which is discussed in Section~\ref{sec:special}. Some other errors are mentioned in Subsection~\ref{subsec:motivation-noncompact}.

This paper provides an abstract framework for generalizing the Gromov-Hausdorff metric that covers all of the current generalizations to the best of the author's knowledge (excluding the GP-type metrics as already mentioned). A framework will also be provided in the non-compact case, which will be introduced in the next subsection.
The framework is based on a small number of assumptions that are usually straightforward to verify (or disprove) in special cases. Several new examples will also be presented.
This generalization is thus useful for studying random metric spaces equipped with new types of additional structures, without needing to replicate the proofs each time. We hope that this makes the use of Gromov-Hausdorff-type metrics easier, quicker and more accessible to researchers and helps them avoid various technicalities and traps.	


The framework applies to the objects of the form $(X,a)$, where $X$ is a compact metric space and $a$ belongs to a given metric space $\tau(X)$, which represents a set of possible additional structures on $X$ (e.g., the set of finite measures on $X$). The distance between two such objects is defined by an equation similarly to that of the GHP metric~\eqref{eq:GHP}. Under some assumptions on $\tau$, which are intended to be as minimalistic as possible (in short, being a \textit{functor} and having some kind of continuity), it is shown that the formula is indeed a metric. Separability (resp. completeness) of the metric is also proved by assuming separability (resp. completeness) of $\tau(X)$ for every $X$ and an additional continuity property of $\tau$. This provides the measure-theoretic requirements for defining \textit{random compact metric spaces equipped with additional structures} and to have a standard probability space. A general pre-compactness result is also provided. 

We will also discuss in detail, mainly in Section~\ref{sec:special}, how this framework extends the existing generalizations of the GH metric. The metrics in the literature are defined in mainly two ways: Some are based on embedding in a larger metric space as in~\eqref{eq:GH} and~\eqref{eq:GHP}. These are directly special cases of the framework, possibly after a minor modification that changes the metric up to a constant factor (e.g., the GHP metric is sometimes defined by using $+$ instead of $\vee$ in~\eqref{eq:GHP}).
Some other papers extend the idea of the equivalent formulation of the GH metric by the notion of \textit{correspondences}, discussed in Subsection~\ref{subsec:strassen}. For instance, an analogous formulation of the GHP metric is available in terms of correspondences and \textit{approximate couplings} based on Strassen's theorem. Here, we call such definitions \textit{Strassen-type metrics}. Although the general framework in this paper is based on the first approach, Subsection~\ref{subsec:strassen} provides a Strassen-type reformulation of the metric as well for some types of additional structures. This extends the existing Strassen-type metrics in the literature with minor modifications that change the metrics up to a constant factor; e.g., using $+$ instead of $\vee$ or having different coefficients in the formulas. We preferred to use $\vee$ in the formula since it is the right choice for having exact Strassen-type formulations.


The main novelty of the paper is the unification of the GH-type metrics and the generality and simplicity of the framework.
The use of functors is new in this context. 
Also, while sometimes more arguments are needed due to few assumptions, the proofs simplify the existing proofs in the literature and show the ideas more clearly
(in fact, we extend the proofs of the previous work~\cite{Kh19ghp} which in turn simplify those in~\cite{AbDeHo13}). For instance, there is no need to use \textit{$\epsilon$-nets} in the proofs.
Additionally, a new result is using the framework to show that in the examples where completeness does not hold, the space is still Polish by showing that it is a $G_{\delta}$ subspace of another Polish space; i.e., a countable intersection of open subsets (Subsections~\ref{subsec:mark}, \ref{subsec:ghpu}, \ref{subsec:spatialTree}, \ref{subsec:spectralGH} and \ref{subsec:cadlagprocess}). We also correct some errors in the literature (where completeness does not hold; e.g., Subsections~\ref{subsec:ghpu} and~\ref{subsec:spatialTree}, or where the metric needs a correction; see Remark~\ref{rem:flaws}).
In addition, GH-type metrics for some new examples of additional structures are also presented; e.g., marked measures, marked closed subsets, continuous functions, ends, processes of closed subsets, and operations on examples (e.g., composition and product). We hope that the framework can be used in the future for studying convergence of metric spaces with new types of additional structures.


%
%
%
%
%
%
%
%

\subsection{The Non-Compact Case}
\label{intro:non-compact}

Gromov defined a notion of convergence for pointed metric spaces $(X_i,o_i)$ ($i=1,2,\ldots$) under the assumption that $X_i$ is \textit{boundedly-compact}; i.e.,  every bounded closed subset of $X_i$ is compact (Section~3.B of~\cite{bookGr99}). 
There are some generalizations in the literature in which an additional structure $a_i$ is assumed on $X_i$.  
For instance, one can mention \textit{Gromov-Hausdorff-Prokhorov convergence} for measured metric spaces (also called \textit{measured Gromov-Hausdorff convergence} \ali{or \textit{Gromov-Hausdorff-vague convergence}}). See e.g., Section~27 of~\cite{bookVi10}.
It is known that these two topologies are metrizable and Polish. This was shown for \textit{length spaces} and discrete metric spaces in~\cite{AbDeHo13} and~\cite{I} respectively and the general case is studied in~\cite{Kh19ghp}.\footnote{In fact, some metrizations were defined by Fukaya~\cite{Fu86convergence}, but they need to be corrected; see Subsection~\ref{subsec:isometries}.} This enables one to study random (measured) non-compact metric spaces; see e.g., \cite{AbDeHo13}, \cite{I} or the references of (and citations to) \cite{AbDeHo13}.

In such generalizations, various formulas have been proposed to defined a notion of metric or convergence, all of which have the following philosophy (which is sometimes called \textit{localization}): 
\begin{equation}
	\label{eq:beingclose}
	\begin{minipage}{.85\textwidth}
		\emph{Two tuples $(X_1, o_1; a_1)$ and $(X_2,o_2;a_2)$  are close to each other when some \textit{large} compact portions of them are close.}
	\end{minipage}
\end{equation}
Here, a \textit{large} portion of $X_i$ means a (closed) ball with large radius centered at $o_i$, or a subset close to (or containing) such a large ball.
Also, the distance between two compact portions is measured with a GH-type metric as in the compact case. 
{This shows why it is important to consider pointed spaces; otherwise, a large part of the space might look like a line and another large part might look like a plane.}

To formalize~\eqref{eq:beingclose}, one should be careful that the distance between the balls of radius $r$ centered at $o_1$ and $o_2$ is not necessarily monotone in $r$. Some papers have fallen into this trap and their definitions should be corrected. Further introduction to the non-compact case is given in Subsection~\ref{subsec:motivation-noncompact} in order to avoid technicalities at this point.

%
%

Among other examples, one can mention the cases where the additional structure on $X_i$ is a continuous curve in $X_i$~\cite{GwMi17}, a c\`adl\`ag curve in $X_i$~\cite{AtLoWi17}, an isometry from $X_i$ to itself (Section~6 of~\cite{bookGr99}) or a function on $X_i$~\cite{DuLe05}. These examples will be studied in detail in Section~\ref{sec:special}. The same philosophy~\eqref{eq:beingclose} can also be used when the underlying spaces $X_1=X_2$ are identical. This has been used in stochastic geometry to defined \textit{random closed subsets},  \textit{random measures}, \textit{point processes}, \textit{random marked measures} and \textit{particle processes} in a fixed metric space (see e.g., \cite{bookScWe08}). 

In this paper, a unified framework is presented as well for the generalizations to the noncompact case. This is based on the framework of compact spaces, mentioned in the previous subsection, under additional assumptions on the map $\tau$. In order to have a more coherent presentation, the introduction to this framework is deferred to Subsection~\ref{subsec:motivation-noncompact} (the reader may jump there right now if he or she wishes).
As in the compact case, the framework unifies several existing generalizations and various new examples will also be provided in Sections~\ref{sec:noncompact} and~\ref{sec:special}.


As a final remark, Benjamini and Schramm's notion of convergence of rooted graphs (also known as \textit{local weak convergence}), defined in \cite{BeSc01} in 2001, is also closely related to Gromov-Hausdorff convergence, 
\ali{but seems to have been introduced independently.} It is in fact stronger than GH-convergence since the edges of a graph are not uniquely determined by the graph-distance metric. It will be shown in Subsection~\ref{subsec:networks} that Benjamini-Schramm convergence is equivalent to the convergence in one of the examples of this paper, namely, \textit{marked metric spaces}. \unwritten{The same holds for local weak convergence of networks~\cite{processes}; i.e., marked graphs.}

\subsection{The Structure of the Paper}



In this introduction, we reviewed convergence of metric spaces and the Gromov-Hausdorff-type metrics in the literature. Further discussion on the literature will be given through the examples. The general framework is presented in Section~\ref{sec:compact} for compact metric spaces and in Section~\ref{sec:noncompact} for non-compact metric spaces. For the sake of mathematical rigor, all proofs are included in these two sections. Some examples are also given in Subsections~\ref{subsec:examples} and~\ref{subsec:examples-noncompact} for illustrating the framework. More involved examples are provided in Section~\ref{sec:special} together with discussing how the framework generalizes the existing examples in the literature. 

\section{The Framework for Compact Metric Spaces}
\label{sec:compact}

This section provides the abstract framework for generalizing the Gromov-Hausdorff metric in the compact case. The proofs of all of the results are postponed to Subsection~\ref{subsec:proofs}.

\subsection{Notation and Basic Definitions}

The minimum and maximum binary operators 
are denoted by $\wedge$ and $\vee$ respectively. 
For all metric spaces $X$ in this paper, the metric on $X$ is usually denoted by $d$ if there is no ambiguity. 
An \textbf{extended metric} is similar to a metric, but might take the value $\infty$.
The complement of a subset $A\subseteq X$ is denoted by $A^c$ or $X\setminus A$.
Also, all measures $\mu$ on $X$ are assumed to be Borel measures. The total mass of $\mu$ is denoted by $||\mu||$.
If in addition, $\rho:X\rightarrow Y$ is measurable, $\rho_*\mu$ denotes the push-forward of $\mu$ under $\rho$; i.e., $\rho_*\mu(\cdot) = \mu(\rho^{-1}(\cdot))$. 
For $x\in X$ and $r\geq 0$, the \defstyle{closed ball} of radius $r$ centered at $x$ is defined by $\cball{r}{x}:= \cball{r}{X,x}:= \{y\in X: d(x,y)\leq r \}$.
The metric space $X$ is \defstyle{boundedly-compact} (also called proper or Heine-Borel in the literature) if every closed ball in $X$ is compact.
Given metric spaces $X$ and $Z$, a function $f:X\to Z$ is an \textbf{isometric embedding} if it preserves the metric; i.e., $d(f(x_1),f(x_2))=d(x_1,x_2)$ for all $x_1,x_2\in X$. It is an \textbf{isometry} if it is a surjective isometric embedding. 
The image of $f:X\to Z$ is denoted by either $f(X)$ or $\mathrm{Im}(f)$.
A \defstyle{measured metric space} is a pair $(X,\mu)$, where $X$ is a metric space and $\mu$ is a Borel measure on $X$. By convention, $(X,\mu)$ is called compact in this paper if $X$ is compact and $\mu$ is a finite measure. Two measured metric spaces $(X,\mu)$ and $(Y,\nu)$ are called equivalent if there exists an isometry $f:X\to Y$ such that $f_*\mu=\nu$.

The \defstyle{Hausdorff distance} of two subsets $A$ and $B$ of a metric space $Z$ is 
\begin{equation}
	\label{eq:hausdorff}
	\hausdorff(A,B):= \inf\{\epsilon\geq 0: A\subseteq \cnei{\epsilon}{B} \text{ and } B\subseteq \cnei{\epsilon}{A} \},
\end{equation}
where $\cnei{\epsilon}{A}:=\{x\in Z:\exists y\in A: d(x,y)\leq \epsilon \}$ is the closed $\epsilon$-neighborhood of $A$ in $Z$.
It is well known that $\hausdorff$ is a metric \ali{when restricted to the set of nonempty compact subsets of $Z$}. In addition, if $Z$ is complete (resp. separable, resp. compact), then the latter is also complete (resp. separable, resp. compact). 
See e.g., Proposition~7.3.7 and Theorem~7.3.8 of~\cite{bookBBI} \new{(for separability, consider finite subsets of a countable dense subset and for completeness, see Lemma~\ref{lem:compact-complete})}.

The \defstyle{Prokhorov distance} of finite measures $\mu$ and $\nu$ on $Z$ is defined by
\begin{eqnarray}
	\label{eq:prokhorov}
	\prokhorov(\mu,\nu):=\min\{\epsilon\geq 0: \forall A: \mu(A)\leq \nu(N_{\epsilon}(A))+\epsilon, \nu(A)\leq \mu(N_{\epsilon}(A))+\epsilon \},
\end{eqnarray}
where $A$ runs over all closed subsets of $Z$. This is a metric on the set of finite Borel measures on $Z$. In addition, if $Z$ is complete and separable, then the latter is also complete and separable. In this case, convergence with respect to the Prokhorov metric is equivalent to weak convergence. 
See Section~6 of~\cite{bookBi99}.

\new{When $\mu(Z)=\nu(Z)$, Strassen's theorem~\cite{St65} reformulates $\prokhorov(\mu,\nu)$ as the minimum $\epsilon\geq 0$ such that there exists a coupling of $\mu$ and $\nu$ supported on $\{(x,y)\in Z\times Z: d(x,y)\leq \epsilon \}$. In the general case, a similar reformulation is provided in~\cite{Kh19ghp} using \textit{approximate couplings}.} 

{To recall, $\mc$ is the set of equivalence classes of compact metric spaces up to isometry. Equipped with the Gromov-Hausdorff metric~\eqref{eq:GH}, $\mc$ is complete and separable. A similar claim holds for the Gromov-Hausdorff-Prokhorov (GHP) metric~\eqref{eq:GHP}. See e.g., Chapter~7 of~\cite{bookBBI} and~\cite{AbDeHo13}.}

\ali{
	Recall that a topological space is Polish when there exists a metrization that turns it into a complete separable metric space. By Alexandrov's theorem, a subset of a Polish space is Polish if and only if it is a $G_{\delta}$ subset. 
}


\subsection{Motivation}
\label{subsec:motivation}

We would like to define the distance between the objects of the form $(X,a)$, where $X$ is a compact metric space and $a$ is an additional structure on $X$. Let $\tau(X)$ be a set which represents the set of additional structures on $X$ under study (e.g., the set of finite measures on $X$).
Motivated by the GHP metric~\eqref{eq:GHP}, the distance between two such objects $(X,a)$ and $(Y,b)$ will be defined by an equation of the form
\begin{equation}
	\label{eq:ghfunctor-motivation}
	\inf_{Z,f,g}\Big\{\hausdorff(f(X),g(Y)) \vee d\big(a', b'\big) \Big\},
\end{equation}
where the infimum is over all compact metric spaces $Z$ and all isometric embeddings $f:X\to Z$ and $g:Y\to Z$. Here, $a'$ and $b'$ are the \textit{images} of $a$ and $b$ under $f$ and $g$ respectively and are considered as structures on $Z$. So a function from $\tau(X)$ to $\tau(Z)$ corresponding to $f$ should be predefined, which is denoted by $\tau_f$ here (e.g., if $a$ is a measure on $X$,  define $\tau_f(a):=f_* a$). Also, in order to make sense of $d(a',b')$, it is required that $\tau(Z)$ is a metric space for all compact metric spaces $Z$. 

The philosophy of the formula is to embed $(X,a)$ and $(Y,b)$ into a common metric space such that their images are close to each other. This philosophy is based on the assumption that such an embedding does not distort the geometry of $X$ and the geometry of the set of additional structures on $X$. So we assume that $\tau_f:\tau(X)\to\tau(Z)$ is always an isometric embedding. In addition, the philosophy is hardly justifiable without the following assumptions: (1) If $X=Z$ and $f$ is the identity function on $X$, then $\tau_f$ is also the identity function on $\tau(X)$ and (2) If $f:X\to Z$ and $h:Z\to Z'$ are isometric embeddings, then $\tau_{h\circ f}=\tau_h \circ \tau_f$. By using the language of category theory, all of these assumptions can be summarized by one simple sentence: $\tau$ should be a \textit{functor} from the category of compact metric spaces to the category of metric spaces. This will be explained in more detail in the next subsection. The notions of categories and functors are useful for simplifying the presentation, but no results or background in category theory are needed here.

The following basic examples will be updated step by step for illustrating the definitions and results (see the conclusion in Example~\ref{ex:basic-compact-result}). Further examples will be discussed in Subsection~\ref{subsec:examples} and also in Section~\ref{sec:special}.

\begin{example}
	\label{ex:basic}
	In the following examples, $X$ and $Z$ represent general compact metric spaces \ali{and $\tau_f$ is an isometric embedding from $\tau(X)$ to $\tau(Z)$, which is defined for an arbitrary isometric embedding $f:X\to Z$.}
	\begin{enumerate}[label=(\roman*)]
		\item \label{ex:basic-points} \textit{Points.} To consider compact metric spaces equipped with a distinguished point, let $\tau(X):=X$ and $\tau_f:=f$.
		\item \label{ex:basic-compact} \textit{Compact subsets.} To consider compact metric spaces equipped with a nonempty compact subset, one can let $\tau(X)$ be the set of nonempty compact subsets of $X$. If one equips $\tau(X)$ with the Hausdorff metric and lets $\tau_f(K):=f(K)$ for all $K\in\tau(X)$, then $\tau_f$ is an isometric embedding. 
		\item \label{ex:basic-measures} \textit{Measures.} To consider compact measured metric spaces, let $\tau(X)$ be the set of all finite measures (or probability measures) on $X$. One can consider the Prokhorov metric on $\tau(X)$ and let $\tau_f(\mu):=f_*\mu$. It is straightforward that $\tau_f$ is an isometric embedding. 
		\item \label{ex:basic-finite} \ali{\textit{Finite subsets.} Let $\tau(X)$ be the set of finite subsets of $X$ and $\tau_f(S):=f(S)$. One can equip $\tau(X)$ either with the Hausdorff metric or with the Prokhorov metric. For the latter, which is more convenient in point process theory, one may regard every finite subset $S$ as the counting measure $\sum_{x\in S}\delta_x$.  
			Then, objects of the form $(X,S)$ are special cases of one of the previous examples depending on the choice of the metric. Note that the two metrics induce different topologies on $\tau(X)$ and have different completions.}
	\end{enumerate}
\end{example}

\begin{remark}
	\label{rem:hausdorff-infty}
	In the above example, one can also let $\tau(X)$ be the set of compact subsets of $X$ including the empty set. For this, it is convenient to extend the Hausdorff metric by letting $\hausdorff(\emptyset,K):=\infty$ for all $K\neq \emptyset$, which leads to an {extended metric} on $\tau(X)$.
\end{remark}

Many other examples can be constructed by combining these simple examples. For instance, for considering compact metric spaces equipped with multiple additional structures (e.g., $k$ points and $l$ measures), it is enough to consider a product space $\tau(X):=\tau_1(X)\times \tau_2(X)\times\ldots$. This will be discussed more in Subsection~\ref{subsec:examples}.

\subsection{The Space $\mathcal C_{\tau}$}
\label{subsec:C}

We now formalize the motivation given in Subsection~\ref{subsec:motivation}. 
Let $\mathfrak{Comp}$ denote the class of compact metric spaces. For two compact metric spaces $X$ and $Y$, let $\mathrm{Hom}(X,Y)$ be the set of isometric embeddings of $X$ into $Y$. 
In the language of category theory, $\mathfrak{Comp}$ is a \defstyle{category} and the elements of $\mathrm{Hom}(X,Y)$ are called \defstyle{morphisms}. An \defstyle{isomorphism} is a morphism which has an inverse. Also, every compact metric space is called an \defstyle{object} of $\mathfrak{Comp}$.
The general definition of categories is omitted for brevity.

Let also $\mathfrak{Met}$ be the category of metric spaces in which the morphisms are isometric embeddings\footnote{This is different from the classical category of metric spaces in the literature in which morphisms are the functions which do not increase the distance between points.}. 

\begin{definition}
	\label{def:functor}
	A \defstyle{functor} $\tau:\mathfrak{Comp}\to\mathfrak{Met}$ is a map that assigns to every compact metric space $X$ a metric space $\tau(X)$, and assigns to every isometric embedding $f:X\to Y$ an isometric embedding $\tau_f:\tau(X)\to \tau(Y)$, such that
	\begin{enumerate}[label=(\roman*)]
		\item For all isometric embeddings $f:X\to Y$ and $g:Y\to Z$, one has $\tau_{g\circ f} = \tau_g\circ \tau_f$.
		\item For every $X$, if $f$ is the identity function on $X$, then $\tau_f$ is the identity function on $\tau(X)$.
	\end{enumerate}
\end{definition}

\ali{Note that $f,g$ are morphisms in $\mathfrak{Comp}$ and $\tau_f,\tau_g$ are morphisms in $\mathfrak{Met}$. Example~\ref{ex:basic} provides basic examples of functors from $\mathfrak{Comp}$ to $\mathfrak{Met}$.} 

\begin{definition}
	\label{def:C_tau}
	\ali{Given a functor $\tau$ as in Definition~\ref{def:functor},}
	let $\mathfrak C_{\tau}$ be the category whose objects are of the form $(X,a)$, where $X$ is an object in $\mathfrak{Comp}$ and $a\in \tau(X)$. Let the set of morphisms between $(X,a)$ and $(Y,b)$ be $\{f\in \mathrm{Hom}(X,Y):\tau_f(a)=b \}$. Let $\mathcal C_{\tau}$ be the set of isomorphism-classes of the objects of $\mathfrak C_{\tau}$ (it can be seen that $\mathcal C_{\tau}$ is indeed a set). 
\end{definition}


\subsection{\ali{The Generalized Gromov-Hausdorff Metric}}
\label{subsec:GHf}


Given a functor $\tau:\mathfrak{Comp}\to\mathfrak{Met}$ as in the previous subsection, one can formalize~\eqref{eq:ghfunctor-motivation} as follows: For two elements $(X,a)$ and $(Y,b)$ of $\mathcal C_{\tau}$, define
\begin{equation}
	\label{eq:ghfunctor}
	\cgf\big((X,a),(Y,b)\big):= \inf_{Z,f,g}\Big\{\hausdorff(f(X),g(Y)) \vee d\big(\tau_f(a),\tau_g(b)\big) \Big\},
\end{equation}
where the infimum is over all compact metric spaces $Z$ and all isometric embeddings $f:X\to Z$ and $g:Y\to Z$. 
It is clear that this definition depends only on the equivalence class of $(X,a)$ and $(Y,b)$, and hence, is indeed well defined on $\mathcal C_{\tau}$. Also, by part~\ref{ex:basic-measures} of Example~\ref{ex:basic}, the function $\cgf$ generalizes the GHP metric~\eqref{eq:GHP}. More examples will be discussed later.

\ali{The following lemma rephrases convergence of a sequence under $d^c_{\tau}$ by embedding them into a common metric space. As mentioned in the introduction, this is used in some works to define a GH-type notion of convergence without defining a metric (see also Lemma~\ref{lem:convergence2} for the non-compact case).}

\begin{lemma}
	\label{lem:common2}
	One has $\cgf\big((X_n,a_n),(X,a) \big)\to 0$ if and only if then there exists a compact metric space $Z$ and isometric embeddings $f:X\to Z$ and $f_{n}:X_{n}\to Z$ such that $f_{n}(X_{n})\to f(X)$ in the Hausdorff metric and $\tau_{f_{n}}(a_{n})\to \tau_f(a)$.
\end{lemma}

The following is the main result of this subsection. The proofs of all of the results are postponed to Subsection~\ref{subsec:proofs}. \new{Recall that a pseudometric allows two distinct points have zero distance.}

\begin{theorem}[Metric]
	\label{thm:functor-metric}
	For any functor $\tau:\mathfrak{Comp}\to\mathfrak{Met}$,
	\begin{enumerate}[label=(\roman*)]
		\item The function $\cgf$ defined in~\eqref{eq:ghfunctor} is a pseudometric on $\mathcal C_{\tau}$.
		\item If $\tau$ is pointwise-continuous (Definition~\ref{def:pointwise} below), then $\cgf$ is a metric and also the infimum in~\eqref{eq:ghfunctor} is attained.
	\end{enumerate}
\end{theorem}


\begin{definition}
	\label{def:pointwise}
	A functor $\tau:\mathfrak{Comp}\to\mathfrak{Met}$ is called \defstyle{pointwise-continuous} when for all 
	compact metric spaces $X$ and $Y$ and all sequences of isometric embeddings $f,f_1,f_2,\ldots$ from $X$ to $Y$, if $f_n\to f$ pointwise, then $\tau_{f_n}\to \tau_f$ pointwise. In addition, it is called \defstyle{pointwise-$M$-Lipschitz} (where $M<\infty$ is given) if for all morphisms $f,g:X\to Y$, one has 
	\begin{equation}
		\label{eq:lipschitz-p}
		\dsup(\tau_f,\tau_g)\leq M\cdot \dsup(f,g),
	\end{equation}
	where $\dsup$ is the sup metric. 
\end{definition}
In many examples, the functor under study has the pointwise-1-Lipschitz property, which is stronger than pointwise-continuity, and even equality holds in~\eqref{eq:lipschitz-p}. The following lemma provides basic examples.

\begin{lemma}
	\label{lem:basic-pointwise}
	In the basic examples of Example~\ref{ex:basic}, 
	$\tau$ is pointwise-continuous and pointwise-1-Lipschitz. In addition, equality holds in~\eqref{eq:lipschitz-p} for $M:=1$.
\end{lemma}

\ali{\begin{proof}[Sketch of the proof]
		Let $f,g:X\to Y$ be isometric embeddings.
		It is an easy exercise to show that for all $a\in \tau(X)$, one has $d(\tau_f(a),\tau_g(a))\leq \dsup(f,g)$ and equality holds for some $a$. \new{For instance, when $a$ is a measure, this is implied by Strassen's theorem mentioned after~\eqref{eq:prokhorov} (or by using \eqref{eq:prokhorov} directly).} This proves the claim.
\end{proof}}

The following lemma is required for proving the next theorems. \new{Recall that a map is \textit{proper} if the inverse image of every compact set is compact.}
\begin{lemma}
	\label{lem:functor-proper}
	It $\tau$ is pointwise-continuous, then for every compact metric space $X$, the map from $\tau(X)$ to $\mathcal C_{\tau}$ defined by $a\mapsto (X,a)$ is \new{1-Lipschitz} and proper \new{(but not necessarily injective). In addition, if $(X,a_n)\to (X,c)$ in $\mathcal C_{\tau}$, then there exists an isometry $h:X\to X$ and a subsequence of $(a_n)_n$ that converges to $h(c)$.} 
\end{lemma}

\begin{remark}[Extended Metrics]
	\label{rem:cgf-infty}
	If $\tau(X)$ is allowed to be an extended metric space, then it can be seen that $\mathcal C_{\tau}$ is an extended metric space. \ali{The key example is the functor of compact subsets (Remark~\ref{rem:hausdorff-infty}), which will be needed in Section~\ref{sec:noncompact}.} In addition, the other results of this section remain valid \ali{(note that the definitions of compactness, completeness and separability can be applied to extended metric spaces as well)}. 
\end{remark}

\subsection{Completeness, Separability and Precompactness}
To have completeness and separability of the metric $\cgf$ defined in~\eqref{eq:ghfunctor}, we assume another continuity property of $\tau$ stated below. 

\begin{definition}
	\label{def:hausdorff-cont}
	A functor $\tau:\mathfrak{Comp}\to\mathfrak{Met}$ is called \defstyle{Hausdorff-continuous} when for every sequence of compact metric spaces $Z,X,X_1,X_2,\ldots$ and isometric embeddings $f:X\to Z$ and $f_n:X_n\to Z$ (for $n=1,2,\ldots$), if $\hausdorff(\mathrm{Im}(f_n),\mathrm{Im}(f))\to 0$, then $\hausdorff(\mathrm{Im}(\tau_{f_n}),\mathrm{Im}(\tau_f))\to 0$. 
	In addition, it is called \defstyle{Hausdorff-$M$-Lipschitz} if for all morphisms $f:X\to Z$ and $g:Y\to Z$, 
	\begin{equation}
		\label{eq:lipschitz-h}
		\hausdorff(\mathrm{Im}(\tau_f), \mathrm{Im}(\tau_g))\leq M \cdot \hausdorff(\mathrm{Im}(f), \mathrm{Im}(g)).
	\end{equation}
\end{definition}

\begin{theorem}[Polishness]
	\label{thm:functor-polish}
	Let $\tau:\mathfrak{Comp}\to\mathfrak{Met}$ be a functor which is both pointwise-continuous and Hausdorff-continuous. 
	Then, $\mathcal C_{\tau}$ is separable (resp. complete) if and only if $\tau(X)$ is separable (resp. complete) for every $X$.
\end{theorem}

As before, many examples have the Hausdorff-1-Lipschitz property, which is stronger than Hausdorff-continuity, and even equality holds in~\eqref{eq:lipschitz-h}. The following lemma provides basic examples. \ali{However, some natural examples fail to be Hausdorff-continuous; e.g., continuous curves (Subsection~\ref{subsec:curves}). Remark~\ref{rem:Hcont-weaker} below provides weaker assumptions that are sufficient for Theorem~\ref{thm:functor-polish}.}

\begin{lemma}
	\label{lem:basic-hausdorff}
	In the basic examples of Example~\ref{ex:basic}, 
	$\tau$ is Hausdorff-continuous and Hausdorff-1-Lipschitz. In addition, equality holds in~\eqref{eq:lipschitz-h} for $M:=1$.
\end{lemma}

\begin{proof}[\new{Sketch of the proof}]
	Let $f:X\to Z$ and $g:Y\to Z$ be isometric embeddings and $\epsilon:=\hausdorff(f(X),g(Y))$. For every $a\in \tau(X)$, one can construct $b\in \tau(Y)$ such that $d(\tau_f(a),\tau_g(b))\leq \epsilon$. For instance, if $a$ is a closed subset, let $b:=g^{-1}(N_{\epsilon}(f(a)))$. Also, if $a$ is a measure, let $b:= (g^{-1}\pi f)_* a$, where $\pi:f(X)\to f(Y)$ assigns to every point of $f(X)$ its closest point in $f(Y)$. \new{Since $d(t,\pi(t))\leq \epsilon, \forall t\in f(X)$, one can directly verify that $\prokhorov(f_*a,g_*b)\leq\epsilon$ using~\eqref{eq:prokhorov} or using Strassen's theorem}. This proves the Hausdorff-1-Lipschitz property. One can also show the equality in~\eqref{eq:lipschitz-h} by considering a single point \new{$x\in X$ such that $d(f(x),g(Y))=\epsilon$}.
\end{proof}

%


\begin{remark}
	\label{rem:Hcont-weaker}
	\unwritten{Note: This is similar to Thm12.2.2 of \cite{bookScWe08} which is the case of closed subsets.}
	\ali{The proof of the above theorem shows that} the assumption of Hausdorff-continuity can be replaced by the following assumptions: For every sequence of compact metric spaces $Z$, $X$, $X_1,X_2,\ldots$ and isometric embeddings $f:X\to Z$ and $f_n:X_n\to Z$ (for $n=1,2,\ldots$) such that $\hausdorff(\mathrm{Im}(f_n),\mathrm{Im}(f))\to 0$,
	\begin{enumerate}[label=(\roman*)]
		\item \label{rem:Hcont-weaker:1} If $b\in \tau(Z)$ and $a_n\in \tau(X_n)$ (for \ali{infinitely many values of $n$}) are such that $\tau_{f_n}(a_n)\to b$, then $b\in \mathrm{Im}(\tau_f)$. 
		\item \label{rem:Hcont-weaker:2} For every $a\in \tau(X)$, there exists a sequence $a_n\in\tau(X_n)$ such that $\tau_{f_n}(a_n)\to \tau_f(a)$.
	\end{enumerate}
\end{remark}

\begin{remark}
	\ali{Condition~\ref{rem:Hcont-weaker:2} in Remark~\ref{rem:Hcont-weaker} is always implied by Hausdorff continuity. So is~\ref{rem:Hcont-weaker:1}, if we assume that $\tau(X)$ is complete for every $X$ (which implies that $\mathrm{Im}(\tau_f)$ is closed in $\tau(Z)$). Conditions~\ref{rem:Hcont-weaker:1} and~\ref{rem:Hcont-weaker:2} together mean that $\mathrm{Im}(\tau_{f_n})\to\mathrm{Im}(\tau_f)$ in the sense of Painleve-Koratowski convergence (Definition B.4 of~\cite{bookMo05}). Assuming completeness and separability in addition, the latter is equivalent to convergence in the Fell topology (Theorem~B.6 of~\cite{bookMo05}).}
\end{remark}

\begin{theorem}[Pre-compactness]
	\label{thm:functor-precompact}
	Let  $\tau:\mathfrak{Comp}\to\mathfrak{Met}$ be a functor which is pointwise-continuous and satisfies Condition~\ref{rem:Hcont-weaker:1} of Remark~\ref{rem:Hcont-weaker}. Then, a set $\mathcal A\subseteq\mathcal C_{\tau}$ is pre-compact if and only if both of the following conditions hold:
	\begin{enumerate}[label=(\roman*)]
		\item \label{thm:functor-precompact:i} The underlying compact sets of the elements of $\mathcal A$ form a pre-compact subset of $\mc$ (equipped with the Gromov-Hausdorff metric~\eqref{eq:GH}). 
		\item \label{thm:functor-precompact:ii} One can select a compact subset $\tau'(X)\subseteq \tau(X)$ for every $X\in\mathfrak{Comp}$ such that $\tau'$ is a functor (by letting $\tau'_f$ be the restriction of $\tau_f$, for every morphism $f$ of $\mathfrak{Comp}$) and $\mathcal C_{\tau'}\supseteq \mathcal A$.
	\end{enumerate}
\end{theorem}


\ali{Condition~\ref{thm:functor-precompact:i} in the above theorem can be characterized by the diameter and the size of $\epsilon$-nets; see e.g., Section~7.4.2 of~\cite{bookBBI}.}
For measured metric spaces, Condition~\ref{thm:functor-precompact:ii} is equivalent to the existence of a uniform upper bound on total masses of the measures (see Theorem~2.6 of~\cite{AbDeHo13}). This fact is generalized in the following proposition \ali{(e.g., let $h_X(\mu)$ be the total mass of $\mu$)}.
This result will also be used for \textit{marked measures} and \textit{marked closed subsets}, which will be defined in Subsection~\ref{subsec:mark}.

\begin{proposition}
	\label{prop:functor-precompact}
	Let $\tau$ and $\mathcal A$ be as in Theorem~\ref{thm:functor-precompact} such that Condition~\ref{thm:functor-precompact:i} of the theorem holds. 
	Assume that there exists a fixed metric\unwritten{(Metric on $E$ is needed for equivalence of compactness with sequential compactness)} space $E$ and a continuous function $h_X:\tau(X)\to E$ for every object $X$ of $\mathfrak{Comp}$ such that $h_X$ is a proper map and is compatible with the morphisms (i.e., $h_Y\circ \tau_f=h_X$ for every morphism $f:X\to Y$). Then, Condition~\ref{thm:functor-precompact:ii} in Theorem~\ref{thm:functor-precompact} is equivalent to the existence of a compact set $E'\subseteq E$ such that all elements $(X,a)\in\mathcal A$ satisfy $h_X(a)\in E'$.
\end{proposition}

\subsection{Examples}
\label{subsec:examples}

The following are some examples for illustrating the framework of this section. Further examples are provided in Section~\ref{sec:special} and their connections to the literature are discussed therein.

\subsubsection{Basic Examples}

\begin{example}[Points, Compact Subsets and Measures]
	\label{ex:basic-compact-result}
	Let $\mathcal C$ be the space of compact metric spaces $X$ equipped with a point $p\in X$, a compact subset $K\subseteq X$ or a finite measure $\mu$ on $X$. In each case, the results of this section imply that $\mathcal C$ is a complete separable metric space. See Lemmas~\ref{lem:basic-pointwise} and~\ref{lem:basic-hausdorff}.
\end{example}

\begin{example}[Constant Functor] 
	\label{ex:constant}
	Let $\Xi$ be a fixed metric space. Assume $\tau(X)=\Xi$ for every $X$ and $\tau_f$ is the identity function on $\Xi$ for every $f$. Then, the metric~\eqref{eq:ghfunctor} becomes $\cgf\big((X,a),(Y,b)\big) = \gh\big(X,Y)\vee d(a,b)$. Hence, $\mathcal C_{\tau}=\mc\times \Xi$ with the max product metric (the case where $\Xi$ is a singleton can be interpreted as \textit{no additional structure}). 
	This trivial example is useful for Subsection~\ref{subsec:mark}.
\end{example}

\begin{example}[Basic Operations]
	Given two functors $\tau_1$ and $\tau_2$, the intersection $\tau_1(X)\cap \tau_2(X)$ and the union $\tau_1(X)\cup \tau_2(X)$ (with the maximal consistent metric) are also functors (if the metrics are compatible on the intersection). Also, the completion of $\tau_1(X)$ is a functor. The reader can verify that pointwise-continuity is inherited, and hence, \eqref{eq:ghfunctor} defines a metric.
\end{example}

\subsubsection{Functorial Subsets}
\label{subsec:subset}
\unwritten{
	\paragraph{Intersection.}\mar{Delete intersection and union. Only keep subsets}
	Let $\tau$ and $\tau'$ be two functors from $\mathfrak{Comp}$ to $\mathfrak{Met}$ such that for every $X$, the metrics on $\tau(X)$ and $\tau'(X)$ agree on $\tau(X)\cap \tau'(X)$. Assume also that for every isometric embedding $f:X\to Y$, the maps $\tau_f$ and $\tau'_f$ agree on $\tau(X)\cap \tau'(X)$. This allows one to define the intersection functor $\tau\cap \tau'$ naturally by $(\tau\cap \tau')(X):=\tau(X)\cap \tau'(X)$. It is immediate that pointwise-continuity and pointwise-Lipschitz properties of $\tau$ and $\tau'$ are inherited to $\tau\cap \tau'$. The same is true for Condition~\ref{rem:Hcont-weaker:1} of Remark~\ref{rem:Hcont-weaker}. But this does not necessarily hold for Hausdorff-continuity nor for Condition~\ref{rem:Hcont-weaker:2} of Remark~\ref{rem:Hcont-weaker}. 
	
	\paragraph{Union.}
	Under the same assumptions as above, one may regard the union $\tau(X)\cup\tau'(X)$ as a functor as well. To do this, one should first define a metric on $\tau(X)\cup\tau'(X)$ in a functorial way. This can be defined naturally if $\tau(X)$ and $\tau'(X)$ are subspaces of a common space $\tau''(X)$, where $\tau''$ is another functor. Otherwise, one may equip $\tau(X)\cup\tau'(X)$ with the maximum metric that is consistent with the two metrics: $d(a,a'):=\inf\{ d(a,y)+d(y,a'): y\in \tau(X)\cap \tau'(X)\}$ for $a\in\tau(X)$ and $a'\in \tau'(X)$ (which is an extended metric if the intersection is empty). It is straightforward to see that all of the continuity and Lipschitz properties of functors are inherited to the union.
}

In some examples where completeness does not hold, one would like that the space $\mathcal C_{\tau}$ is a Borel subset of some Polish space. In such examples, we will prove this property by 
considering a larger functor. 
Let $\tau'$ be another pointwise-continuous functor such that $\tau(X)$ is a metric subspace of $\tau'(X)$ for every $X$ and $\tau_f$ agrees with (the restriction of) $\tau'_f$ for every morphism $f:X\to Y$. It is clear that $\mathcal C_{\tau}$ is also a subspace of $\mathcal C_{\tau'}$, but the topological properties need further assumptions. Call $\tau$ a \defstyle{functorial subset} of $\tau'$ if for every isometric embedding $f:X\to Y$, one has $\mathrm{Im}(\tau_f)=\mathrm{Im}(\tau'_f)\cap \tau(Y)$; or equivalently, $\tau'\setminus\tau$ is a functor \new{(i.e., $\tau'_f$ maps $\tau'(X)\setminus\tau(X)$ into $\tau'(Y)\setminus\tau(Y)$)}.

\begin{lemma}
	\label{lem:subset}
	Let $\tau$ be a functorial subset of $\tau'$ such that $\tau(X)$ is a closed (resp. open) subset of $\tau'(X)$ for every $X$. Then, $\mathcal C_{\tau}$ is also a closed (resp. open) subset of $\mathcal C_{\tau'}$.
\end{lemma}
\begin{proof}
	Since $\tau'\setminus\tau$ is a functor \new{and $\mathcal C_{\tau'}$ is the disjoint union of $\mathcal C_{\tau}$ and $\mathcal C_{\tau'\setminus\tau}$,} it is enough to prove the claim when $\tau(X)$ is closed in $\tau'(X)$.
	Assume $(X_n,a'_n)\to (X,a)$, where $a'_n\in \tau(X_n)$ and $a\in \tau'(X)$. By Lemma~\ref{lem:common2}, there exists a common metric space $Z$ and isometric embeddings $f_n:X_n\to Z$ and $f:X\to Z$ such that $\hausdorff(\mathrm{Im}(f_n), \mathrm{Im}(f))\to 0$ and $\tau'_{f_n}(a'_n)\to \tau'_f(a)$. Since $\tau(Z)$ is closed and $\tau'_{f_n}(a'_n)\in \tau(Z)$, one gets $\tau'_f(a)\in \tau(Z)$. So, being a functorial subset implies that $\tau'_f(a)\in \mathrm{Im}(\tau_f)$, and the claim is proved.
\end{proof}

Now assume that $\tau(X)$ is a Borel subset of $\tau'(X)$ for every $X$. Does it imply that $\mathcal C_{\tau}$ is also a Borel subset of $\mathcal C_{\tau'}$? The answer is not known to the author. It seems that one needs to specify the Borel class in a functorial way as follows.

\begin{lemma}
	Assume $\tau(X)$ is a $G_{\delta}$ subset of $\tau'(X)$ for every $X$. In addition, assume there exists functorial subsets $\tau^{(1)}, \tau^{(2)}, \ldots$ of $\tau'$ such that each $\tau^{(i)}(X)$ is an open subset of $\tau'(X)$ for every $X$ and $\tau(X)=\bigcap_i \tau^{(i)}(X)$. Then, $\mathcal C_{\tau}$ is a $G_{\delta}$ subset of $\mathcal C_{\tau'}$.
\end{lemma}

This claim is immediately implied by the previous lemma \new{by noting that $\mathcal C_{\tau}=\bigcap_i \mathcal C_{\tau^{(i)}}$}. A similar statement can be given for other Borel classes like $F_{\sigma}, G_{\delta\sigma}, \ldots$.

\ali{
	In some examples, different metrics can be considered on $\tau(X)$; e.g., the sup metric between curves and the Hausdorff metric between the graphs of the curves (see Subsection~\ref{subsec:curves} below). The following lemma is useful for comparing the topologies.
	\begin{lemma}
		\label{lem:topology}
		Let $\tau$ and $\tau'$ be pointwise-continuous functors such that $\tau(X)\subseteq \tau'(X)$ for all $X$, $\tau_f$ is the restriction of $\tau'_f$ for every $f$ and the topology of $\tau(X)$ is coarser (resp. finer) than that of $\tau'(X)$ (restricted to $\tau(X)$). Then, $\mathcal C_{\tau}$ is a subset of $\mathcal C_{\tau'}$ and has a coarser (resp. finer) topology. In particular, in the coarser case, if $\mathcal C_{\tau'}$ is separable, then so is $\mathcal C_{\tau}$. 
	\end{lemma}
	
	The proof is similarly to that of Lemma~\ref{lem:subset} and is skipped.
}


\subsubsection{Multiple Additional Structures}
\label{subsec:product}
Let $\tau_1,\tau_2,\ldots,\tau_n$ be functors and $\tau(X):=\prod_i \tau_i(X)$. Then, $\mathcal C_{\tau}$ is the set of compact metric spaces $X$ equipped with a tuple $(a_1,\ldots,a_n)$ of $n$ additional structures such that $a_i\in \tau_i(X)$. Here, we prefer to equip $\tau(X)$ with the max product metric 
\begin{equation}
	\label{eq:maxproduct-finite}
	d\big((a_1,\ldots,a_n),(b_1,\ldots,b_n)\big):=\max_{i} d(a_i,b_i).
\end{equation}

This metric is more convenient for the Lipschitz properties (Lemma~\ref{lem:multiple2}) and for Strassen-type formulations (Subsection~\ref{subsec:strassen}). Similarly, for the product of countably many functors $\tau_1,\tau_2,\ldots$, the following metric is more convenient: 
\begin{equation}
	\label{eq:maxproduct}
	d\big((a_1,a_2,\ldots),(b_1,b_2,\ldots)\big):=\max_{i} \{\ali{2^{-i}\wedge d(a_i,b_i)}\}.
\end{equation}
Also, every isometric embedding $f:X\to Z$ induces a function from $\tau(X)$ to $\tau(Z)$ naturally, which is an isometric embedding. Hence, $\tau$ is a functor.



\begin{lemma}
	\label{lem:multiple2}
	Let\unwritten{\mar{Can we say if each is complete, then the product is complete?}} $\tau$ be the product of functors $\tau_1,\tau_2,\ldots$ defined above.
	\begin{enumerate}[label=(\roman*)]
		\item If $\tau_i$ is pointwise-continuous for every $i$, then so is $\tau$. A similar result holds for Hausdorff-continuity, the pointwise-$M$-Lipschitz property and the Hausdorff-M-Lipschitz property.
		\item If $\tau_i$ satisfies the assumptions of Theorem~\ref{thm:functor-polish} for every $i$, then $\mathcal C_{\tau}$ is complete and separable.
	\end{enumerate}
\end{lemma}

{The proof is straightforward and is left to the reader.}

For instance,~\cite{Mi09} defines a GH-type metric for compact metric spaces equipped with $k$ distinguished closed subsets. According to Lemma~\ref{lem:multiple2}, this is a special case of the framework of this section. Also,~\cite{AdBrGoMi17} considers the set of compact metric spaces equipped with $k$ distinguished points and $l$ finite Borel measures. It is claimed that this space is complete and separable. This is also implied by Lemma~\ref{lem:multiple2}. In fact, the metric in~\cite{AdBrGoMi17} differs from the metric of Lemma~\ref{lem:multiple2} up to a constant factor. This will be explained in Subsection~\ref{subsec:strassen}.


\subsubsection{Continuous Curves and Mappings}
\label{subsec:curves}
Let $\mathcal C$ be the space of compact metric spaces $X$ equipped with a continuous function $\eta:I\to X$, where $I$ is a given compact \ali{metric space}. 
\ali{If $I$ is an interval, then $\eta$ is a curve and the results of~\cite{GwMi17} show that $\mathcal C$ is a Polish space. Here, we represent it as an example of our framework and also generalize it to an arbitrary $I$. 
	See also Subsection~\ref{subsec:curvesnoncompact} (for non-compact $I$) and Subsection~\ref{subsec:ghpu}.
}

Let $\tau(X)$ be the set of continuous functions $\eta:I\to X$ equipped with the sup metric. Also, by letting $\tau_f(\eta):=f\circ \eta$, $\tau_f$ is an isometric embedding. Here, $\tau$ is a functor and $\mathcal C_{\tau}=\mathcal C$.

\ali{
	\begin{proposition}
		\label{prop:curves}
		The functor $\tau$, defined above, is pointwise-1-Lipschitz, satisfies Condition~\ref{rem:Hcont-weaker:1} of Remark~\ref{rem:Hcont-weaker}, but is not Hausdorff-continuous in general. Nevertheless, the space $\mathcal C_{\tau}$ of compact metric spaces $X$ equipped with a continuous function $\eta:I\to X$ is complete and separable.
	\end{proposition}
	
	Here, $\tau$ is not generally Hausdorff-continuous. For instance, if one approximates $X$ by finite metric spaces $X_1,X_2,\ldots$, then a continuous curve in $X$ cannot necessarily be approximated by continuous curves in $X_n$. By the same reason, condition~\ref{rem:Hcont-weaker:2} of Remark~\ref{rem:Hcont-weaker} does not necessarily hold. So separability is not implied by Theorem~\ref{thm:functor-polish}. In the case when $I$ is an interval, one can prove separability directly as in~\cite{GwMi17}, but the method of~\cite{GwMi17} does not work in the general case. Here we will prove it by using Theorem~\ref{thm:functor-polish} for a larger functor as in Lemma~\ref{lem:topology}. 
	
	\begin{proof}
		Assume $f,g:X\to Y$ are isometric embeddings. For any $\eta:I\to X$, one has $\dsup(f\circ \eta, g\circ \eta)\leq \dsup(f,g)$. This means that $\tau$ is pointwise-1-Lipschitz.
		%
		%
		Now, assume $f_n:X_n\to Z$ is an isometric embedding as in Remark~\ref{rem:Hcont-weaker} and $\hausdorff(\mathrm{Im}(f_n),\mathrm{Im}(f))\to 0$. Assume $\eta_n:I\to X_n$ is continuous and $f_n\circ\eta_n\to \eta'$, where $\eta':I\to Z$ is a continuous curve. For all $t\in I$, one obtains $\eta'(t)=\lim f_n(\eta_n(t))\in \mathrm{Im}(f)$. Hence, $\eta'=\tau_f(f^{-1}\circ \eta)\in \mathrm{Im}(\tau_f)$. This proves condition~\ref{rem:Hcont-weaker:1} of Remark~\ref{rem:Hcont-weaker}. Since $\tau(X)$ is complete for every $X$, Remark~\ref{rem:Hcont-weaker} implies that $\mathcal C_{\tau}$ is complete.
		For proving separability of $\mathcal C_{\tau}$, identify every continuous function $\eta:I\to X$ with its graph, which is a closed subset of $I\times X$. So $\tau(X)$ can be regarded as a subset of the set $\tau'(X)$ of closed subsets of $I\times X$ equipped with the Hausdorff metric. In Subsection~\ref{subsec:mark}, it will be shown that $\tau'$ is Hausdorff-continuous and $\mathcal C_{\tau'}$ is separable. 
		So, separability of $\mathcal C_{\tau}$ is implied by Lemma~\ref{lem:topology} and the fact that the topology on $\tau(X)$ is identical to that induced from $\tau'(X)$, which is left to the reader (use modulus of continuity).
	\end{proof}
	
	\unwritten{In the above proof, one can also use Lemma~\ref{lem:subset} to show that $\mathcal C_{\tau}$ is a $G_{\delta}$ subset of $\mathcal C_{\tau'}$ (which is already implied by the fact that $\mathcal C_{\tau}$ is Polish).}
	
}



\subsubsection{C\`adl\`ag Curves}
\label{subsec:cadlag}

	
	For all compact metric spaces $X$, let $\tau_0(X)$ be the set of c\`adl\`ag curves $\eta:I\to X$, where $I$ is a compact or non-compact interval (a c\`adl\`ag curve is a function that is right-continuous and has left-limits at all points). The space $\tau_{0}(X)$ is complete and separable under the \textit{Skorokhod metric} (Sections~12 and~16 of~\cite{bookBi99}).
	One can regard $\tau_0$ as a functor similarly to Subsection~\ref{subsec:curves}.
	\ali{In contrast with continuous functions, this example is Hausdorff continuous. More generally: 
		
		\begin{proposition}
			The functor $\tau_0$ of c\`adl\`ag curves is pointwise-1-Lipschitz and Hausdorff-$(1+\epsilon)$-Lipschitz for every $\epsilon>0$. So the set of compact metric spaces equipped with a c\`adl\`ag curve is complete and separable under the metric~\eqref{eq:ghfunctor}.
		\end{proposition}
		
		To prove this claim, assume $X,Y\subseteq Z$ and $\eta$ is a c\`adl\`ag curve in $X$. The reader can verify that there exists a c\`adl\`ag curve $\eta'$ in $Y$ such that $\dsup(\eta,\eta')\leq (1+\epsilon) \hausdorff(X,Y)$ (take a sufficiently fine net in $Y$ and let $\eta'$ be piece-wise constant). The details are skipped for brevity.
		
		This example will be discussed further in Subsections~\ref{subsec:cadlag2} and~\ref{subsec:cadlagprocess}.
	}
	

	\subsubsection{Composition of Functors}
	\label{subsec:compos}
	
	Assume $\rho$ and $\tau$ are functors from $\mathfrak{Comp}$ to $\mathfrak{Met}$ as in Subsection~\ref{subsec:GHf}. If $\rho(X)$ is compact for every $X$, then one can define the functor $\tau\circ\rho:\mathfrak{Comp}\to\mathfrak{Met}$ by $\tau\circ\rho(X):=\tau(\rho(X))$ and $(\tau\circ\rho)_f:=\tau_{(\rho_f)}$. 
	The following lemma is straightforward to prove.
	\begin{lemma}
		\label{lem:composition}
		If both $\rho$ and $\tau$ are pointwise-continuous (resp., Hausdorff-continuous), then so is $\tau\circ\rho$. A similar statement holds for the Lipschitz properties and the conditions in Remark~\ref{rem:Hcont-weaker}.
	\end{lemma}
	
	Therefore, if both $\rho$ and $\tau$ satisfy the assumptions of the results of this section, then so does $\tau\circ\rho$. In addition, it can be seen that if $\rho$ is Hausdorff-$M$-Lipschitz, then the map from $\mathcal C_{\tau\circ\rho}$ to $\mathcal C_{\tau}$ defined by $(X,a)\mapsto (\rho(X),a)$ \ali{is an $M$-Lipschitz function.}
	
	\ali{For example, composition is useful if one wants to study compact metric spaces $X$ equipped with a family of elements in $\rho(X)$; e.g., $k$ elements of $\rho(X)$, finitely many points in $\rho(X)$, a closed subset of $\rho(X)$ or a probability measure on $\rho(X)$ (i.e., a random element of $\rho(X)$). These are special cases of composition of functors by choosing $\tau$ suitably. 
		Various specific examples will be studied in Subsection~\ref{subsec:particle} and~\ref{subsec:closedProcess} together with their extension to boundedly-compact metric spaces (based on Example~\ref{ex:compose}).
		
		Some examples of $\rho$ in which $\rho(X)$ is always compact are the functors of closed subsets, probability measures and finite measures with total mass at most $M$, where $M$ is a constant. The same holds for \textit{marked measures} and \textit{marked closed subsets} (defined in the next subsection) if the mark space is compact. For boundedly-compact mark spaces, one can also let $\rho$ be the functor of \textit{marked closed subsets}. See the next subsection for more details.
	}
	%
	
	
	\begin{remark}
		\label{rem:composition}
		In some examples of composition, $\rho(X)$ is not compact. This is useful for the study of marks in the next subsection (Proposition~\ref{prop:mark-compact}) and for more general two-level measures (Subsection~\ref{subsec:two-level}). For this, one needs that $\tau$ is a functor $\tau:\mathfrak{Met}\to\mathfrak{Met}$. Such a functor can be defined similarly to Definition~\ref{def:functor}. Here, although~\eqref{eq:ghfunctor} is not well behaved (since the GH metric between noncompact metric spaces is too strict, is an extended metric and the corresponding space is not separable), but the definitions of pointwise-continuity, Hausdorff-continuity, the Lipschitz properties and the conditions of Remark~\ref{rem:Hcont-weaker} are still valid and are useful for composition of functors. These properties are inherited by the composition $\tau\circ\rho$ similarly to Lemma~\ref{lem:composition}.
	\end{remark}

	\subsubsection{Marks}
	\label{subsec:mark}
	
	Fix a {complete separable} metric space $\Xi$ as the \defstyle{mark space}. Heuristically, we consider metric spaces $X$ equipped with marks on points or marks on tuples of points. Due to measurability issues regarding $\Xi^X$, the following definitions are introduced.
	
	\del{In the following, we will consider metric spaces equipped with \textit{marks} (Examples~\ref{ex:marks} and~\ref{ex:marks2}). 
		This will be used in other examples as well, including curves (Subsection~\ref{subsec:curves}), \textit{spatial trees} (Subsection~\ref{subsec:spatialTree}) and in Subsection~\ref{subsec:spectralGH}.
		
		Fix a {complete separable} metric space $\Xi$ as the \defstyle{mark space} {(in some cases in what follows, $\Xi$ is required to be boundedly-compact).} 
		If $X$ is a finite or discrete set, a \textit{marking} of $X$ can be defined as a function from $X$ to $\Xi$ and the space of markings of $X$ is simply $\Xi^X$.
		However, in the general (non-discrete) case, $\Xi^X$ is not suitable to be regarded as a metric space and measure theoretic issues appear. Two candidates are provided here for defining markings: \textit{marked measures} and \textit{marked closed subsets}, defined below, which are inspired by the notion of \textit{marked random measures} in stochastic geometry (see Subsection~\ref{subsec:markedProcess}). 
		
		In some applications, it is convenient to assign marks to the $k$-tuples of points (see e.g., Subsections~\ref{subsec:graphs} and~\ref{subsec:spectralGH}). In the discrete case, the set $\Xi^{X^k}$ can be regarded as space of markings. This is also generalized in the following.
	}
	\begin{definition}
		\label{def:marking}
		Let $X$ be a compact metric space and $k\in \mathbb N$. A \defstyle{$k$-fold marked measure} on $X$ is a Borel measure on $X^k\times \Xi$. 
		Also, a \defstyle{$k$-fold marked closed (resp. compact) subset} of $X$ is a closed (resp. compact) subset of $X^k\times \Xi$. 
		The number $k$ is called the \defstyle{order} of the marked measure/closed subset.
	\end{definition}
	
	\del{In the discrete case, it is straightforward to see that both notions generalize the notion of markings (identify every function $f:X\to\Xi$ with its graph and the counting measure on its graph).}
	
	In the case $k=1$, this definition is inspired by the notion of \textit{marked random measures} in stochastic geometry (see Subsection~\ref{subsec:stochasticGeometry0}). Also, 1-fold marked measures are studied in~\cite{DeGrPf11} and~\cite{KlLo15markfunction}, but with a version of the Gromov-Prokhorov topology, which is weaker than the topology we are considering.
	
	\begin{definition}
		\label{def:mark-functor}
		Let $\tau^f(X)$ be the set of finite $k$-fold marked measures with the Prokhorov metric. Also, let $\tau^c(X)$ be the set of nonempty $k$-fold marked compact subsets of $X$ with the Hausdorff metric.
	\end{definition}
	
	Note that these functors contain the basic examples of Example~\ref{ex:basic} (points, measures and closed subsets) as functorial subsets. \ali{When $\Xi$ is not compact, one can also study marked closed subsets (not necessarily compact) and boundedly-finite marked measures, but the Hausdorff and Prokhorov metrics are no longer suitable. This will be studied in Example~\ref{ex:marks2} with another metric.}
	
	\begin{proposition}
		\label{prop:mark-compact}
		The functors $\tau^f$ and $\tau^c$ are pointwise-1-Lipschitz and also Hausdorff-1-Lipschitz. Hence, the spaces $\mathcal C_{\tau^{(f)}}$ and $\mathcal C_{\tau^{(c)}}$ are complete and separable.
	\end{proposition}
	
	%

	\begin{proof}
		First, assume the mark space $\Xi$ is compact. In this case, the functor $\tau^{(f)}$ is the composition of the functor $X\mapsto X^k\times \Xi$ (which is itself a product of the identity functor $X\mapsto X$ and the constant functor $X\mapsto\Xi$ of Subsection~\ref{subsec:subset}) with the functor of finite measures (Example~\ref{ex:basic}). So the claim is implied by Lemmas~\ref{lem:composition}, \ref{lem:basic-pointwise} and~\ref{lem:basic-hausdorff}. The proof for the functor $\tau^{(c)}$ is similar. In the case $\Xi$ is not compact, the above argument is still valid by Remark~\ref{rem:composition} (one can also prove the claim directly similarly to Lemmas~\ref{lem:basic-pointwise} and~\ref{lem:basic-hausdorff}).
	\end{proof}

	The following example motivates the names `marked measure' and `marked closed subset'. 
	
	\begin{example}[\ali{Simple Marks}]
		\label{ex:marking-function}
		If $\mu$ is a finite measure on $X$ and $f:X\to\Xi$ is a measurable function, then the push-forward of $\mu$ under the map $x\mapsto (x,f(x))$ is a finite marked measure on $X$. Similarly, if $C$ is a compact subset of $X$ and $f:X\to\Xi$ is a continuous map, then $\{(x,f(x)):x\in C\}$ is a marked compact subset of $X$. 
		This can be extended to $k$-fold marks similarly.
	\end{example}
	
	A compact space with a simple 1-fold marked measure is called a \textit{functionally marked metric measure space} in~\cite{KlLo15markfunction}. It is proved in~\cite{KlLo15markfunction} that the set of such spaces  is  Polish (with the GP-type topology). 
	By the tools presented here, we modify and shorten the proof of~\cite{KlLo15markfunction} to show the same with a stronger topology. 
	
	\begin{proposition}
		\label{prop:simplemarks}
		Let $\tau(X)$ be the set of simple $k$-fold marked closed subsets of $X$ (resp. simple $k$-fold marked measures on $X$). Then, $\tau$ is pointwise-continuous and $\mathcal C_{\tau}$ is a $G_{\delta}$ subspace of some Polish space, and hence, $\mathcal C_{\tau}$ is Polish itself.
	\end{proposition}

	\begin{proof}
		We only prove the case $k=1$ for simplicity. Pointwise-continuity is straightforward. 
		For marked closed subsets, we use modulus of continuity. Given $m,n\in\mathbb N$, let $\tau^{(m,n)}(X)$ be the set of marked closed subsets $K\subseteq X\times \Xi$ such for every $(x,a),(y,b)\in K$, if $d(x,y)\leq 1/m$, then $d(a,b)<1/n$. It can be seen that $\tau^{(m,n)}(X)$ is open in $\tau^{(c)}(X)$\unwritten{ ()See the unwritten notes in OneNote)} and is a functorial subset. Also, the set of simple marked closed subsets is exactly $\cap_n\cup_m \tau^{(m,n)}(X)$. So, Lemma~\ref{lem:subset} implies that $\mathcal C_{\tau}$ 
		is a $G_{\delta}$ subset of the Polish space $\mathcal C_{\tau^{(c)}}$.
		
		For marked measures, we use the characterization of simple marked measures in~\cite{KlLo15markfunction} as follows. 
		In~\cite{KlLo15markfunction}, an explicit function $\beta:\tau^{(f)}(X)\to\mathbb R^{\geq 0}$ is constructed for every $X$ such that simple marked measures on $X$ are exactly $\beta^{-1}(0)=\tau^{(f)}(X)\setminus \bigcup_m \overline{\beta^{-1}(1/m,\infty)}$. See~(2.2) and~(2.5) of~\cite{KlLo15markfunction}. So, Lemma~\ref{lem:subset} can be used as above to deduce that $\mathcal C_{\tau}$ 
		is a $G_{\delta}$ subset of the Polish space $\mathcal C_{\tau^{f}}$, and hence, is Polish by itself. This completes the proof.
	\end{proof}
	

	
	
	Proposition~\ref{prop:functor-precompact} gives an explicit pre-compactness result for marks as follows.
	
	\begin{proposition}
		For $\tau=\tau^{(c)}$ or $\tau=\tau^{(f)}$, a subset $\mathcal A\subseteq \mathcal C_{\tau}$ is pre-compact if and only if Condition~\ref{thm:functor-precompact:i} of Theorem~\ref{thm:functor-precompact} holds and 
		\begin{enumerate}[label=(\roman*)]
			\item if $\tau=\tau^{(c)}$, there exists a compact subset $E\subseteq\Xi$ such that all elements $(X,K)\in\mathcal A$ satisfy $K\subseteq X^k\times E$.
			\item if $\tau=\tau^{(f)}$, there exists $M<\infty$ and compact subsets $E_1,E_2,\ldots\subseteq \Xi$ such that all elements $(X,\mu)\in A$ satisfy $||\mu||\leq M$ and $\forall i:\mu(X^k\times E_i^c)\leq 1/i$.
		\end{enumerate} 
	\end{proposition}
	\begin{proof}
		For the first part, let $h_X(K):=\pi(K)$, where $\pi:X^k\times \Xi\to \Xi$ is the projection map. Then, the claim is implied by Proposition~\ref{prop:functor-precompact}. For the second part, let $h_X(\mu):=(||\mu||, \del{\frac{1}{||\mu||}}\pi_*\mu)$ and the claim is implied by Proposition~\ref{prop:functor-precompact} and Prokhorov's theorem on tightness of measures.
	\end{proof}

	\subsection{Strassen-Type Formulations}
	\label{subsec:strassen}
	
	A \defstyle{correspondence} between metric spaces $X$ and $Y$ is a closed subset of $X\times Y$ such that, regarded as a relation between $X$ and $Y$, every point of $X$ corresponds to at least one point of $Y$ and vice versa. The \defstyle{distortion} of $R$ is defined by $\sup\{|d(y,y')-d(x,x')|: (x,y)\in R, (x',y')\in\mathbb R \}$. Let $C_{\epsilon}(X,Y)$ be the set of correspondences with distortion at most $2\epsilon$.
	
	Some generalization of the Gromov-Hausdorff metric in the literature are defined in terms of correspondences (instead of isometric embeddings) with a formula of the following type:
	\begin{equation}
		\label{eq:strassen-ghf}
		d\left((X,a),(Y,b) \right) := \min\{\epsilon\geq 0: \exists R\in \mathrm{C}_{\epsilon}(X,Y): p(X,Y,a,b,\epsilon,R)=1\},
	\end{equation}
	where $p$ is some criterion in terms of $X,Y,a,b,\epsilon,R$. For instance, the GH metric is obtained by $p\equiv 1$. Also, if $a$ and $b$ are probability measures, the GHP metric is obtained by the criterion of the existence of a coupling $\alpha$ of $a$ and $b$ such that $\alpha(R)\geq 1-\epsilon$. This is based on Strassen's theorem~\cite{St65} and is generalized to finite measures in~\cite{Kh19ghp} using approximate couplings. Formulas like~\eqref{eq:strassen-ghf} are sometimes easier to use than~\eqref{eq:ghfunctor} and are called \defstyle{Strassen-type formulations} here. 
	Although there is no isometric embedding in~\eqref{eq:strassen-ghf}, it seems that the existence of a Strassen-type formulation is a property of a functor which is hidden beyond:
	
	\begin{definition}
		A pointwise-continuous functor $\tau:\mathfrak{Comp}\to\mathfrak{Met}$ is called a \defstyle{Strassen-type functor} if for $a\in\tau(X)$, $b\in\tau(Y)$ and isometric embeddings $f:X\to Z$ and $g:Y\to Z$, whether or not $d(\tau_f(a),\tau_g(b))\leq\epsilon$ depends only on $X,Y,a,b,\epsilon$ and the set $R:=\{(x,y)\in X\times Y:d(f(x),g(y))\leq \epsilon\}$. In this case, let $p(X,Y,a,b,\epsilon,R)$ be a formula \new{which represents the indicator of $d(\tau_f(a),\tau_g(b))\leq\epsilon$ in terms of $(X,Y,a,b,\epsilon,R)$.}
	\end{definition}
	
	\unwritten{All of the examples given so far have Strassen-type formulations. Also, the tools below can be used to obtain Strassen-type formulations for various other instances of functors.}
	
	\begin{proposition}
		If $\tau$ is a Strassen-type functor, then the GHP-type metric $\cgf$ is equal to~\eqref{eq:strassen-ghf} and the minimum in~\eqref{eq:strassen-ghf} is attained.	
	\end{proposition}
	
	This result is immediate from the definition. \new{Indeed, the existence of a correspondence $R\in C_{\epsilon}(X,Y)$ is equivalent to $\gh(X,Y)\leq\epsilon$ and the other condition in~\eqref{eq:strassen-ghf} is analogous to the condition~$d(\tau_f(a),\tau_g(b))\leq\epsilon$ in~\eqref{eq:ghfunctor}.} This justifies the use of $\vee$ in the definition of the metric~\eqref{eq:ghfunctor} (one could also use $+$ and it would give an equivalent metric).
	
	Some trivial examples are points and closed subsets (Example~\ref{ex:basic}), continuous curves (Subsection~\ref{subsec:curves}), c\`adl\`ag curves (Subsection~\ref{subsec:cadlag}) and a constant functor (Subsection~\ref{subsec:subset}). For finite measures, two formulas $p$ can be given, one in terms of Strassen's theorem, mentioned above, and one by the original definition~\eqref{eq:prokhorov} of the Prokhorov metric. In contrast, Wasserstein metrics for finite measures (see e.g., Section~6 of~\cite{bookVi10}) give functors which are not Strassen-type (but are still pointwise-1-Lipschitz and Hausdorff-1-Lipschitz). 
	
	As another example, marked compact subsets and marked finite measures (Definition~\ref{def:mark-functor}) are also Strassen-type. This is implied by the following lemma and the fact that these functors are obtained by product and composition of simpler functors (see the proof of Proposition~\ref{prop:mark-compact}). See also Example~\ref{ex:marks2}.
	
	\begin{lemma}
		\label{lem:strassen-closed}
		Strassen-type functors are closed under product (with the max product metrics~\eqref{eq:maxproduct-finite} or~\eqref{eq:maxproduct}) and composition (Subsections~\ref{subsec:product} and~\ref{subsec:compos}).
	\end{lemma}
	The proof is straightforward and is skipped. This lemma justifies the use of the max product metric for products.
	
	\begin{example}
		Proposition~9 of~\cite{Mi09} provides a Strassen-type formulation for compact metric spaces equipped with $k$ closed subsets. This is a special case of products in Lemma~\ref{lem:strassen-closed}. Also, when the additional structure is a tuple of $k$ distinguished points and $l$ finite measures, \cite{AdBrGoMi17} defines a metric using a Strassen-type formula directly. After a minor modification that changes the metric up to a constant factor (changing a $\vee$ to $+$; see~\cite{Kh19ghp}), this is also a special case of products in Lemma~\ref{lem:strassen-closed}. Another example is given in Subsection~\ref{subsec:spatialTree}.
	\end{example}

	\subsection{Proofs and Auxiliary Lemmas}
	\label{subsec:proofs}
	
	In order to prove the main results of this section, we start by some lemmas.	
	
	\begin{lemma}
		\label{lem:embedding}
		For compact metric spaces $X$ and $Z$, the set of isometric embeddings $f:X\to Z$, equipped with the sup metric, is compact. In addition, the topology of this set is identical to the topology of pointwise convergence.
	\end{lemma}
	\new{This lemma is quickly implied by the Arzel\`a-Ascoli theorem for compact metric spaces (see e.g., \cite{bookKe75generaltopology}). 
	}
	\unwritten{
		\begin{proof}
			{Assume $f_i:X\to Z$, $i\in\mathbb N$ are isometric embeddings. Let $I:=\{x_1,x_2,\ldots\}$ be a countable dense subset of $X$. By passing to subsequences inductively and a diagonal argument, one may assume that for each $j$, the limit $f(x_j):=\lim_i(f_i(x_j))\in Z$ exists. Since each $f_i$ is an isometric embedding, $f:I\to Z$ is also an isometric embedding. Therefore, since $I$ is dense, $f$ can be extended to an isometric embedding of $X$ into $Z$. It is now straightforward to check that $f_i$ converges to $f$ pointwise. To prove that $f_i\to f$ uniformly, let $\epsilon>0$ be arbitrary and $\{y_1,\ldots,y_k\}\subseteq X$ be an $(\epsilon/3)$-net in $X$ (i.e., its $(\epsilon/3)$-neighborhood covers the whole $X$). For large enough $n$ and all $1\leq j\leq k$, $d(f_n(y_j),f(y_j))<\epsilon/3$. It follows that if $d(y_j,x)<\epsilon/3$, then $d(f_n(x),f(x))<\epsilon$. This proves uniform convergence.}
		\end{proof}
	}
	\begin{lemma}
		\label{lem:compact-complete}
		\unwritten{Blaschke's theorem is the case where $Z$ is compact. Also, the lemma is known for closed subsets of $Z$ under $\hausdorff$ (Prop 7.3.7 of \cite{bookBBI})}
		If $Z$ is a complete metric space, then the set of compact subsets of $Z$, equipped with the Hausdorff metric, is complete.
	\end{lemma}
	\new{\begin{proof}[Sketch of the proof]
			By Proposition~7.3.7 of~\cite{bookBBI}, the set of closed subsets of $X$ is complete. It remains to show that if a sequence of compact subsets $C_n\subseteq X$ converges to a closed subset $C\subseteq X$, then $C$ is compact. Let $(x_i)_i$ be a sequence in $C$ and $\epsilon_n:=\hausdorff(C_n,C)$. By compactness of $C_1$ and passing to a subsequence, one may assume that $(x_i)_i$ eventually lie in a ball of radius $2\epsilon_1$. By doing this for $C_2,C_3,\ldots$ recursively and a diagonal argument, one finds a Cauchy subsequence, which is convergent in $C$. This finishes the proof.
		\end{proof}
	}
	\unwritten{
		\begin{proof}
			{It is known that the set of closed subsets of $Z$ is complete under the Hausdorff metric (see e.g., Proposition~7.3.7 of~\cite{bookBBI}). It is left to the reader to show that if a sequence of compact subsets converge to a closed subset, then the limit is also compact (see e.g., the arXiv version). This proves the claim.}
		\end{proof}
	}
	
	\begin{lemma}
		\label{lem:functor-union}
		In taking infimum in~\eqref{eq:ghfunctor}, one can add the condition $Z=f(X)\cup g(Y)$ and the value of the infimum is not changed.
	\end{lemma}
	
	\begin{proof}
		Let $Z,f,g$ be as in~\eqref{eq:ghfunctor}. Let $Z':=f(X)\cup g(Y)$ and $\iota:Z'\hookrightarrow Z$ be the inclusion map. Let $f'\in \mathrm{Hom}(X,Z')$ and $g'\in \mathrm{Hom}(Y,Z')$ be obtained by restrictions of $f$ and $g$ respectively. One has $\iota \circ f' = f$ and $\iota \circ g' = g$ as morphisms in $\mathfrak{Comp}$. 
		\ali{By the definition of functors (Definition~\ref{def:functor}),} $\tau_{\iota} \circ \tau_{f'} = \tau_{f}$ and $\tau_{\iota}\circ\tau_{g'}=\tau_g$. 
		Since $\tau_{\iota}$ is an isometric embedding (by \ali{Definition~\ref{def:functor}}), one gets that $d(\tau_f(a),\tau_g(b))= d(\tau_{f'}(a),\tau_{g'}(b))$. This proves the claim.
	\end{proof}

	\begin{lemma}
		\label{lem:common}
		Let $\correction{\mathcal X_n}=(X_n,a_n)$ be an object in $\mathfrak C_{\tau}$ and $\epsilon_n>0$ for $n=1,2,\ldots$ such that $\cgf\left(\mathcal X_n, \mathcal X_{n+1}\right)<\epsilon_n$ for each $n$. If $\sum \epsilon_n<\infty$, then there exists a compact set $Z$ and isometric embeddings $f_n:X_n\to Z$ such that for all $n$, 
		\begin{eqnarray}
			\label{eq:lem:common:1}
			\hausdorff\big(f_n(X_n),f_{n+1}(X_{n+1})\big) &\leq& \epsilon_n,\\ \label{eq:lem:common:2}
			d(\tau_{f_n}(a_n),\tau_{f_{n+1}}(a_{n+1}))&\leq& \epsilon_n.
		\end{eqnarray}
	\end{lemma}
	
	\begin{proof}
		By~\eqref{eq:ghfunctor}, there exists a compact metric space $Z_n$ for every $n$ and isometric embeddings $g_n:X_n\to Z_n$ and $h_n:X_{n+1}\to Z_n$ such that 
		\begin{eqnarray}
			\label{eq:lem:common:3} \hausdorff\big(g_n(X_n),h_n(X_{n+1}) \big)&\leq &  \epsilon_n,\\ 
			\label{eq:lem:common:4} d(\tau_{g_n}(a_n),\tau_{h_n}(a_{n+1})) &\leq& \epsilon_n.
		\end{eqnarray}
		By Lemma~\ref{lem:functor-union}, one can assume $Z_n=g_n(X_n)\cup h_n(X_{n+1})$ without loss of generality.
		
		
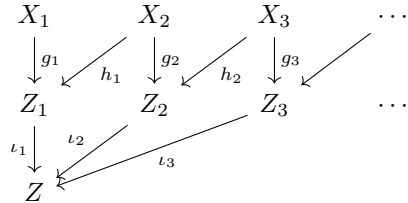
{
			\[ 
			\begin{tikzcd}
				X_1 \arrow{d}{g_1} & X_2\arrow{dl}{h_1} \arrow{d}{g_2} & X_3\arrow{dl}{h_2} \arrow{d}{g_3}&\cdots \arrow{dl}{}\\
				Z_1\arrow[d,"\iota_1"'] & Z_2\arrow[dl,"\iota_2"'] & Z_3\arrow{dll}{\iota_3} & \cdots\\
				Z&&&
				%
			\end{tikzcd}
			\]
		}
		Let $Z_{\infty}$ be the gluing of the spaces $Z_n$; i.e., the quotient of the disjoint union $\sqcup_n Z_n$ by identifying $h_n(x)\in Z_n$ with $g_{n+1}(x)\in Z_{n+1}$ for every $n$ and every $x\in X_{n+1}$. 
		By considering the quotient metric on $Z_{\infty}$, the natural map from $Z_n$ to $Z_{\infty}$ is an isometric embedding for each $n$.%
		\del{The quotient metric on $Z_{\infty}$ can be described as follows:
			for $z_i\in Z_i$ and $z_j\in Z_j$, if $i\leq j$, the distance between $z_i$ and $z_j$ is the \textit{length} of the shortest path of the form $(z_i,h_i(x_{i+1}), h_{i+1}(x_{i+2}),\ldots,h_{j-1}(x_j), z_j)$, where $\forall k: x_k\in Z_k$ and the distance between consecutive pairs in this path are considered under the metrics of $Z_i,Z_{i+1},\ldots,Z_j$ respectively (note that $h_k(x_{k+1})$ is an element of $Z_k$ and is identified with the element $g_{k+1}(x_{k+1})$ of $Z_{k+1}$). It is straightforward that this gives a metric on $Z_{\infty}$ and the natural map from $Z_n$ to $Z_{\infty}$ is an isometric embedding for each $n$ (see Lemma~5.7 in~\cite{GrPfWi09}). 
		}
		%
		Let ${Z}$ be the metric completion of $Z_{\infty}$ and $\tilde{Z}_n$ be the quotient of $Z_1\sqcup\cdots\sqcup Z_n$ defined similarly. We may regard $\tilde{Z}_n$ as a subset of $Z$, which gives $Z_{\infty}=\cup_n\tilde{Z}_n$. Inequality~\eqref{eq:lem:common:3} implies that $\hausdorff(\tilde{Z}_n,\tilde{Z}_{n+1})\leq\epsilon_n$. So the assumption $\sum \epsilon_n<\infty$ implies that $\hausdorff(\tilde{Z}_n,{Z})$ is finite and tends to zero as $n\to \infty$. Since $Z$ is complete and each $\tilde{Z}_n$ is compact, Lemma~\ref{lem:compact-complete} implies that $Z$ is compact. 
		
		Let $\iota_n:Z_n\to {Z}$ be the natural isometric embedding and $f_n:=\iota_n\circ g_n$.
		%
		Since $\iota_n$ is an isometric embedding and $\iota_n\left(h_{n}(X_{n+1})\right)= \iota_{n+1}\left(g_{n+1}(X_{n+1})\right)$, \eqref{eq:lem:common:3} implies that $\hausdorff\big(f_n(X_n),f_{n+1}(X_{n+1})\big) \leq \epsilon_n$, which proves  \eqref{eq:lem:common:1}. Similarly, since $\tau_{\iota_n}$ is an isometric embedding, \eqref{eq:lem:common:4} implies \eqref{eq:lem:common:2}. So the claim is proved.
	\end{proof}
	
	\begin{remark}
		Lemma~\ref{lem:common} is similar to	 Lemma~5.7 in~\cite{GrPfWi09}. The latter is for metric measure spaces and does not assume $\sum \epsilon_n<\infty$. So the metric space $Z$ is not necessarily compact therein.
		\del{
			Also, a similar statement holds for every sequence of compact metric spaces $(X_n)_n$ without any additional structure (by deleting $a_n$ and $a$ in Lemma~\ref{lem:common} and by replacing $\cgf$ with $\cgh$). This claim is implied by Lemma~\ref{lem:common} by considering the functor $\tau(X)=\{0\}$ for all $X$.}
	\end{remark}
	
	\begin{proof}[Proof of Lemma~\ref{lem:common2}]
		The proof is similar to that of Lemma~\ref{lem:common} and is only sketched here. 
		Let $\epsilon_n>\cgf\big((X_n,a_n),(X,a) \big)$ be such that $\epsilon_n\to 0$.
		Embed $X_n$ and $X$ in a common space $Z_n$ as in~\eqref{eq:ghfunctor}. Then, let $Z$ be the gluing all of $Z_1,Z_2,\ldots$ along the copies of $X$ in all of the sets $Z_n$. It can be proved similarly to Lemma~\ref{lem:common} that $Z$ is compact and can be used as the desired space.
	\end{proof}

	We are now ready to prove the theorems.
	
	\begin{proof}[Proof of Theorem~\ref{thm:functor-metric}]
		It is clear that $\cgf$ is symmetric. For the triangle inequality, let $\mathcal X=(X,a)$, $\mathcal Y=(Y,a)$ and $\mathcal Z=(Z,c)$ be elements of $\mathfrak{C}_{\tau}$. Assume $\cgf(\mathcal X,\mathcal Y)<\epsilon$ and $\cgf(\mathcal Y,\mathcal Z)<\delta$. It is enough to prove that $\cgf(\mathcal X,\mathcal Z)\leq \epsilon+\delta$.
		By Lemma~\ref{lem:common}, there exists a compact metric space $H$ and isometric embeddings $f_X:X\to H$, $f_Y:Y\to H$ and $f_Z:Z\to H$ such that
		\begin{eqnarray*}
			\hausdorff(f_X(X),f_Y(Y))\leq \epsilon, &&
			d(\tau_{f_X}(a), \tau_{f_Y}(b)) \leq  \epsilon,\\
			\hausdorff(f_Y(Y),f_Z(Z))\leq \delta, &&
			d(\tau_{f_Y}(b), \tau_{f_Z}(c)) \leq  \delta.
		\end{eqnarray*}
		These inequalities imply that $\hausdorff(f_X(X),f_Z(Z))\leq \epsilon+\delta$ and $d(\tau_{f_X}(a), \tau_{f_Z}(c))\leq \epsilon+\delta$. So the definition~\eqref{eq:ghfunctor} implies that $\cgf(\mathcal X,\mathcal Z)\leq \epsilon+\delta$ and the triangle inequality is proved.
		
		\ali{Now, let $\mathcal X=(X,a)$ and $\mathcal Y=(Y,b)$ be such that $\cgf(\mathcal X,\mathcal Y)=0$. Here, we use the fact that the infimum in~\eqref{eq:ghfunctor} is attained (this will be proved below and there is no circular argument). So, there exists a compact metric space $Z$ and isometric embeddings $f:X\to Z$ and $g:Y\to Z$ such that 
			$f(X)=g(Y)$ and $\tau_{f}(a)=\tau_{g}(b)$. Therefore, the function $h:=g^{-1}\circ f$ is an isometry between $X$ and $Y$ and $\tau_{h}(a)=b$. Hence, $(X,a)$ is isomorphic to $(Y,b)$. So, it is proved that $\cgf$ is a metric.}
		
		\ali{It only remains to show that the infimum in~\eqref{eq:ghfunctor} is attained.}
		Let $\epsilon:=\cgf(\mathcal X, \mathcal Y)$. By~\eqref{eq:ghfunctor}, for every $n>0$, there exists a compact metric space $Z_n$ and isometric embeddings $f_n:X\to Z_n$ and $g_n:Y\to Z_n$ such that
		\[
		\hausdorff\big(f_n(X),g_n(Y) \big)\vee d(\tau_{f_n}(a),\tau_{g_n}(b)) \leq \epsilon + \frac 1 n.
		\]
		Also, by Lemma~\ref{lem:functor-union}, one can assume $f_n(X)\cup g_n(Y)=Z_n$. This implies that $\diam(Z_n)\leq \diam(X)+2\epsilon+2/n$, which is uniformly bounded. So, by Theorem~7.4.15 of~\cite{bookBBI}, one can show that the sequence $(Z_n)_n$ is pre-compact under the metric $\cgh$ \new{(note that $Z_n$ is a union of two copies of $X$ and $Y$, and hence, can be covered by a uniformly bounded number of balls of a given radius)}. Therefore, by taking a subsequence if necessary, we can assume $\cgh(Z_n,K)\leq 2^{-n}$ for some $K$ without loss of generality. By Lemma~\ref{lem:common}, there exists a compact metric space $Z$ and isometric embeddings $h_n:Z_n\to Z$ (this is all we need from Lemma~\ref{lem:common}). 
		By Lemma~\ref{lem:embedding} \ali{and passing to a subsequence},
		one can assume $h_n\circ f_n\to f$ and $h_n\circ g_n\to g$ uniformly for some isometric embeddings $f:X\to Z$ and $g:Y\to Z$ without loss of generality. Since $\hausdorff(\mathrm{Im}(h_n\circ f_n), \mathrm{Im}(h_n\circ g_n))\leq \epsilon+\frac 1 n$, one can show that $\hausdorff(\mathrm{Im}(f), \mathrm{Im}(g))\leq \epsilon$. Moreover, by pointwise-continuity of $\tau$, one gets that $\tau_{h_n\circ f_n}(a) \to \tau_f(a)$ and $\tau_{h_n\circ g_n}(b)\to \tau_g(b)$. It follows similarly that $d(\tau_f(a),\tau_g(b))\leq \epsilon$. Now $(Z,f,g)$ satisfy the claim and the claim is proved.
		%
	\end{proof}

	\begin{proof}[Proof of Lemma~\ref{lem:functor-proper}]
		By the definition of $\cgf$ in~\eqref{eq:ghfunctor}, one directly obtains that
		$
		\cgf\big((X,a_1),(X,a_2) \big)\leq d(a_1,a_2).
		$
		This implies that the map is continuous (and also 1-Lipschitz). 
		To prove properness of the map, 
		let $K\subseteq\mathcal C_{\tau}$ be a compact set and $a_1,a_2,\ldots \in \tau(X)$ be such that $(X,a_n)\in K$. To show the compactness of the inverse image of $K$, it is enough to show that the sequence $(a_n)_n$ has a convergent subsequence (note that by continuity, the inverse image of $K$ is closed). 
		Since $K$ is compact, by taking a subsequence, we may assume $(X,a_n)\to (Y,b)$, where $(Y,b)\in K$. It follows that $Y$ is isometric to $X$. So there exists $c\in X$ such that $(X,c)$ is equivalent to $(Y,b)$ as elements of $\mathcal C_{\tau}$. So $(X,a_n)\to (X,c)\in K$. By Lemma~\ref{lem:common2}, 
		there exists a compact metric space $Z$ and isometric embeddings $f_n:X\to Z$ and $f:X\to Z$ such that $f_n(X)\to f(X)$ and $\tau_{f_n}(a_n)\to \tau_f(c)$. 
		By Lemma~\ref{lem:embedding} and
		passing to a subsequence, we may assume there exists $g:X\to Z$ such that $\dsup(f_n,g)\to 0$. This implies that $f_n(X)\to g(X)$, and hence, $f(X)=g(X)$. So there exists an isometry $h:X\to X$ such that $f=g	 \circ h$. Let $a:=\tau_h(c)$.
		Pointwise-continuity and $f_n\to g$ implies that $\tau_{f_n}(a)\to \tau_g(a) = \tau_f(c) = \lim_n \tau_{f_n}(a_n)$. 
		So $d(\tau_{f_n}(a_n),\tau_{f_n}(a))\to 0$. Since $\tau_{f_n}$ is an isometry, one gets that $d(a_n,a)\to 0$; i.e., $a_n\to a$. This completes the proof.
		%
	\end{proof}

	\begin{proof}[Proof of Theorem~\ref{thm:functor-polish}]
		\new{Assume $\mathcal C_{\tau}$ is complete. Let $(a_n)_n$ be a Cauchy sequence in $\tau(X)$. By Lemma~\ref{lem:functor-proper}, $(X,a_n)$ is also a Cauchy sequence in $\mathcal C_{\tau}$, and hence, is convergent. We just proved in the proof of Lemma~\ref{lem:functor-proper} that $(a_n)_n$ has a convergent subsequence, which implies that the whole sequence is convergent. So $\tau(X)$ is complete.
			
			Assume $\mathcal C_{\tau}$ is separable. So the image of $\tau(X)$ in $\mathcal C_{\tau}$ is also separable. Let $\{(X,a_n):n\in\mathbb N\}$ be a countable dense set in this image. Let $G$ be a countable dense subset of the isometry group of $X$ and let $S:=\{g(a_n): g\in G, n\in\mathbb N\}$. We claim that $S$ is dense in $\tau(X)$. Indeed, for every $c\in\tau(X)$, there exists a subsequence $(X,a_{n_i})\to (X,c)$. By the second assertion of Lemma~\ref{lem:functor-proper}, there exists an isometry $h:X\to X$ such that $h^{-1}(a_{n_i})\to c$. Now the claim follows by approximating $h^{-1}$ by the elements of $G$.}
		
		Now, we prove the `if' parts.
		Let $(X_n,a_n)$ be a sequence of elements of $\mathfrak C_{\tau}$ such that the corresponding elements in $\mathcal C_{\tau}$ form a Cauchy sequence. By taking a subsequence (if necessary) and using Lemma~\ref{lem:common}, one can assume there exists a compact metric space $Z$ and isometric embeddings $f_n:X_n\to Z$ such that
		\begin{eqnarray}
			\label{eq:lem:funcomplete:1}
			\hausdorff\big(f_n(X_n),f_{n+1}(X_{n+1})\big) &\leq& 2^{-n},\\ \label{eq:lem:funcomplete:2}
			d(\tau_{f_n}(a_n),\tau_{f_{n+1}}(a_{n+1}))&\leq& 2^{-n}.
		\end{eqnarray}
		So the sequences $(f_n(X_n))_n$ and $(\tau_{f_n}(a_n))_n$ are Cauchy. Lemma~\ref{lem:compact-complete} and the assumption of completeness of $\tau(Z)$ imply that there exists a compact subset $X\subseteq Z$ and $b\in \tau(Z)$ such that $f_n(X_n)\to X$ and $\tau_{f_n}(a_n)\to b$. Let $\iota: X\hookrightarrow Z$ be the inclusion map. The definition of Hausdorff-continuity implies that $\hausdorff(\mathrm{Im}(\tau_{f_n}), \mathrm{Im}(\tau_{\iota}))\to 0$. This implies that $b$ is in the closure of $\mathrm{Im}(\tau_{\iota})$. On the other hand, since $\tau(X)$ is complete (by assumption) and $\tau_{\iota}$ is an isometric embedding, one gets that $\mathrm{Im}(\tau_{\iota})$ is also complete, and hence, closed in $\tau(Z)$. So $b\in \mathrm{Im}(\tau_{\iota})$; i.e., there exists $a\in \tau(X)$ such that $\tau_{\iota}(a)=b$. Now, one can obtain that $\cgf\left((X_n,a_n),(X,a) \right)\to 0$. This proves that $\mathcal C_{\tau}$ is complete.
		
		Now assume that $\tau(X)$ is separable for every $X$.
		Let $A$ be a sequence of compact metric spaces which is dense in $\mc$. By assumption, for every $X\in A$, there exists a countable dense subset $C(X)$ of $\tau(X)$. It is enough to prove that the set $E:=\{(X,a): X\in A, a\in C(X) \}$ is dense in $\mathcal C_{\tau}$.
		Let $(X,a)\in\mathfrak{C}_{\tau}$ be arbitrary. For every $n>0$, there exists $X_n\in A$ such that $\cgh(X_n,X)\leq 2^{-n}$. By Lemma~\ref{lem:common2},  there exists a compact metric space $Z$ and isometric embeddings $f:X\to Z$ and $f_n:X_n\to Z$ such that $\hausdorff(f(X),f_n(X_n))\to 0$. The assumption of Hausdorff-continuity of $\tau$ implies that $\hausdorff\big(\mathrm{Im}(\tau_{f_n}),\mathrm{Im}(\tau_f)\big)\to 0$. So one can select an element $a_n\in \tau(X_n)$ for each $n$ such that $d(\tau_{f_n}(a_n),\tau_f(a))\to 0$. Since $C(X_n)$ is dense in $\tau(X_n)$, one can choose $a_n$ such that $a_n\in C(X_n)$. This implies that $\cgf((X_n,a_n),(X,a))\to 0$ and the claim is proved.
	\end{proof}

	\begin{proof}[Proof of Theorem~\ref{thm:functor-precompact}]
		($\Leftarrow$). Assume~\ref{thm:functor-precompact:i} and~\ref{thm:functor-precompact:ii} hold. Let $(X_n,a_n)$ be a sequence in $A$. We should prove that it has a convergent subsequence in $\mathcal C_{\tau}$. By~\ref{thm:functor-precompact:i}, one can assume  $X_n\to X$ for some compact metric space $X$ without loss of generality. 
		So Lemma~\ref{lem:common2} implies that there exists a compact metric space $Z$ and isometric embeddings $f_n:X_n\to Z$ and $f:X\to Z$ such that $f_n(X_n)\to f(X)$. By assumption, $a_n\in \tau'(X_n)$. Since $\tau'$ is a functor, one gets $\tau_{f_n}(a_n)\in \tau'(Z)$. Since the latter is compact, by passing to a subsequence, one can assume that there exists $b\in \tau'(Z)$ such that $\tau_{f_n}(a_n)\to b$. Similarly to the proof of completeness in Theorem~\ref{thm:functor-polish}, one can show that there exists $a\in \tau'(X)$ such that $\tau_f(a)=b$\del{ (repetition of the arguments is skipped for brevity)}. Now, it can be seen that $(X_n,a_n)\to (X,b)$ and the claim is proved. 
		
		($\Rightarrow$). Assume $\mathcal A$ is pre-compact. The claim of~\ref{thm:functor-precompact:i} is straightforward. To prove~\ref{thm:functor-precompact:ii}, for every compact metric space $X$, define
		\[
		\tau_0(X):= \{\tau_f(b): (Y,b)\in\mathcal A \text{ and } f\in \mathrm{Hom}(Y,X) \}\subseteq \tau(X).
		\]
		Let $\tau'(X)$ be the closure of $\tau_0(X)$ in $\tau(X)$. It is straightforward that $\tau_0$ and $\tau'$ are functors and $\mathcal C_{\tau'}\supseteq\mathcal A$. So it is enough to show that $\tau'(X)$ is compact for  every $X$; i.e., $\tau_0(X)$ is pre-compact in $\tau(X)$. Let $a_1,a_2,\ldots\in \tau_0(X)$. We should prove that it has a convergent subsequence. By the definition of $\tau_0$, for every $n$ there exists $(Y_n,b_n)\in \mathcal A$ and $f_n:Y_n\to X$ such that $\tau_{f_n}(b_n)=a_n$. Since $\mathcal A$ is pre-compact, we can assume that $(Y,b):=\lim_n (Y_n,b_n)$ exists. 
		By Lemma~\ref{lem:common2}, there exists a compact metric space $Z$, $g_n:Y_n\to Z$ and $g:Y\to Z$ such that 
		\begin{eqnarray}
			\label{eq:thm:functor-precompact:1}
			g_n(Y_n)&\to& g(Y),\\
			\label{eq:thm:functor-precompact:2}
			\tau_{g_n}(b_n)&\to& \tau_g(b).
		\end{eqnarray}
		For each $n$, let $K_n$ be the gluing of $X$ and $Z$ along the two copies of $Y_n$. 
		Note that $\diam(K_n)\leq \diam(X)+\diam(Z)$. So the diameters of $K_n$ are uniformly bounded. By using Theorem~7.4.15 of~\cite{bookBBI}, one can show that the sequence $(K_n)_{n}$ is pre-compact under the metric $\cgh$. So, by taking a subsequence, we may assume that $(K_n)_n$ is convergent under $\cgh$. 
		So, by Lemma~\ref{lem:common2}, 
		all of $K_1,K_2,\ldots$ are isometrically embeddable into a common compact metric space $H$ such that their images in $H$ are convergent. 
		{
			\[ 
			\begin{tikzcd}
				Y_n \arrow{r}{f_n}  \arrow{rd}{g_n} & X \arrow{rd}{} & &\\
				Y\arrow{r}{g}&Z \arrow{r} & K_n\arrow{r}{} & H
			\end{tikzcd}
			\]
		}
		By composing the isometric embeddings{ (see the above diagram)}, 
		one finds isometric embeddings $h_n:X\to H$ and $\iota_n:Z\to H$ (for each $n$) such that $h_n\circ f_n = \iota_n\circ g_n$ for each $n$ (in fact, one may do this such that the maps $\iota_n$ are equal). {See the diagram below.} 
		By Lemma~\ref{lem:embedding} \ali{and passing to a subsequence},
		we may assume there exist isometric embeddings $h:X\to H$ and $\iota:Z\to H$ such that $\dsup(h_n,h)\to 0$ and $\dsup(\iota_n,\iota)\to 0$. Equation~\eqref{eq:thm:functor-precompact:1} implies that ${\iota_n}({g_n}(Y_n))\to {\iota}(g(Y))$. On the other hand, ${\iota_n}({g_n}(Y_n)) = {h_n}({f_n}(Y_n))\subseteq h_n(X)$. These facts imply that ${\iota}(g(Y))\subseteq h(X)$. It follows that there exists an isometric embedding $f:Y\to X$ such that $h\circ f = \iota \circ g$. 
		{
			\[ 
			\begin{tikzcd}
				Y_n \arrow{r}{f_n}  \arrow{rd}{g_n} & X \arrow{rd}{h_n} & \\
				&Z \arrow{r}{\iota_n} & H
			\end{tikzcd}	
			\begin{tikzcd}
				Y \arrow[dashed]{r}{f}  \arrow{rd}{g} & X \arrow{rd}{h} & \\
				&Z \arrow{r}{\iota} & H
			\end{tikzcd}
			\]
		}
		Let $a:=\tau_f(b)\in \tau(X)$. 
		It follows that
		\begin{eqnarray*}
			\tau_{h_n}(a_n) = &\tau_{h_n}\left(\tau_{f_n}(b_n)\right)& = \tau_{\iota_n}\left(\tau_{g_n}(b_n) \right),\\
			\tau_{h}(a) = &\tau_{h}\left(\tau_{f}(b)\right)& = \tau_{\iota}\left(\tau_{g}(b) \right).
		\end{eqnarray*}
		Therefore, by letting $c_n:=\tau_{g_n}(b_n)$ and $c:=\tau_g(b)$, one has 
		\begin{eqnarray*}
			d(\tau_{h_n}(a_n),\tau_h(a)) &=& d\big(\tau_{\iota_n}\left(c_n \right), \tau_{\iota}\left(c \right) \big)\\
			&\leq & 
			d\big(\tau_{\iota_n}\left(c_n \right), \tau_{\iota_n}\left(c \right) \big) + d\big(\tau_{\iota_n}\left(c \right), \tau_{\iota}\left(c \right) \big)\\
			&=& d\big(c_n, c\big) + d\big(\tau_{\iota_n}\left(c \right), \tau_{\iota}\left(c \right) \big).
		\end{eqnarray*}
		So, \eqref{eq:thm:functor-precompact:2} and pointwise-continuity imply that $\tau_{h_n}(a_n)\to \tau_h(a)$. Since we had $h_n(X)\to h(X)$, one gets that $(X,a_n)\to (X,a)$ under the metric $\cgf$. In particular, the sequence $(X,a_n)$ is pre-compact in $\mathcal C_{\tau}$. So Lemma~\ref{lem:functor-proper} implies that the sequence $a_1,a_2,\ldots$ has a convergent subsequence in $\tau(X)$. This completes the proof.
	\end{proof}

	\begin{proof}[Proof of Proposition~\ref{prop:functor-precompact}]
		First, assume that such $E'$ exists. For every $X$, let $\tau'(X):=h_X^{-1}(E')$. Since $h_X$ is a proper function, $\tau'(X)$ is a compact subset of $\tau(X)$. It is straightforward that $\tau'$ is a functor from $\mathfrak{Comp}$ to $\mathfrak{Met}$. So Condition~\ref{thm:functor-precompact:ii} of Theorem~\ref{thm:functor-precompact} holds.
		
		Conversely, assume that Condition~\ref{thm:functor-precompact:ii} of Theorem~\ref{thm:functor-precompact} holds but $E'$ does not exist with the desired conditions. The latter implies that there exists a sequence $(X_n,a_n)\in A$ such that $(h_{X_n}(a_n))_n$ does not have any convergent subsequence. By Theorem~\ref{thm:functor-precompact}, we may assume that $(X_n,a_n)\to (X,a)$ for some $(X,a)\in\mathcal C_{\tau}$. By Lemma~\ref{lem:common2}, 
		there exists a compact metric space $Z$ and isometric embeddings $f:X\to Z$ and $f_n:X_n\to Z$ such that $\tau_{f_n}(a_n)\to \tau_f(a)$ in $\tau(Z)$. Continuity of $h_Z$ implies that $h_{X_n}(a_n)=h_Z(\tau_{f_n}(a_n)) \to h_Z(\tau_f(a))=h_X(a)$, which is a contradiction. So the claim is proved.
	\end{proof}

	\section{The Framework for Boundedly-Compact Metric Spaces}
	\label{sec:noncompact}
	
	In this section, we extend the framework of Section~\ref{sec:compact} to non-compact metric spaces equipped with an additional structure. First, an introduction to this matter is provided in Subsection~\ref{subsec:motivation-noncompact}. 
	
	\subsection{Motivation}
	\label{subsec:motivation-noncompact}
	
	In this section, the term \bcm{} abbreviates \textit{boundedly-compact metric space}.
	Let $X_i$ ($i=1,2$) be metric spaces and $o_i\in X$. Such a pair $(X_i,o_i)$ is called a \defstyle{pointed metric space} and we call $o_i$ the \defstyle{root} here. According to Subsection~\ref{intro:non-compact} of the introduction, we will use the idea~\eqref{eq:beingclose} to extend the framework of Section~\ref{sec:compact} to pointed \bcm s. 
	Let $a_i\in\varphi(X_i)$, where for every \bcm{} $X$, $\varphi(X)$ is a set of possible additional structures (under some conditions specified later).
	For example, one can let $a_i$ be a closed subset, a measure or a discrete subset. 

	For defining a metric with the intuition~\eqref{eq:beingclose}, a formula that has always been successful in the literature is the following integral formula (e.g., in~\cite{AbDeHo13} for measured length spaces): Define the distance between $\mathcal X_i=(X_i,o_i; a_i)$, $i=1,2$, by
	\begin{equation}
		\label{eq:integral}
		\int_0^{\infty} e^{-r}\left(1\wedge d\left(\pcball{\mathcal X}{r}_1,\pcball{\mathcal X}{r}_2 \right) \right)dr,
	\end{equation}
	where $\pcball{\mathcal X}{r}_i$ is the closed ball of radius $r$ centered at $o_i$. The additional structure on $\pcball{\mathcal X}{r}_i$ is the \textit{truncation} of $a_i$ to the ball, the definition of which should be presumed. Also, the distance between $\pcball{\mathcal X}{r}_1$ and $\pcball{\mathcal X}{r}_2$ is measured with a \textit{rooted} GH-type metric. In the latter, the root is taken into account as an additional structure on $\pcball{\mathcal X}{r}_i$. 
	For instance, \del{the rooted GH metric is a special case of~\eqref{eq:ghfunctor} and Example~\ref{ex:basic}:
		\begin{eqnarray}
			\nonumber
			\cgh(\mathcal X_1,\mathcal X_2):= \inf_{Z,f,g}\Big\{\hausdorff(f(X_1),g(X_2)) \vee d(o_1,o_2) \Big\},
		\end{eqnarray}
		when $\mathcal X_i=(X_i,o_i)$ is compact. Also, }the rooted GHP metric (when $\mathcal X_i=(X_i,o_i;\mu_i)$ is compact) is a special case of~\eqref{eq:ghfunctor} and Example~\ref{ex:basic}:
	\begin{equation}
		\nonumber
		\cghp(\mathcal X_1,\mathcal X_2):= \inf_{Z,f,g}\Big\{\hausdorff(f(X_1),g(X_2)) \vee d(o_1,o_2) \vee \prokhorov(f_*\mu_1,g_*\mu_2) \Big\}.
	\end{equation}
	
	\begin{remark}
		\label{rem:flaws}
		Note that $d(\pcball{\mathcal X}{r}_1,\pcball{\mathcal X}{r}_2)$ is neither monotone nor continuous in $r$ in general. In addition, in order to have $\mathcal X_n\to \mathcal X$, the balls $\pcball{\mathcal X}{r}_n$ with a given radius $r$ need not converge to $\pcball{\mathcal X}{r}$ \ali{(one should perturb the ball a little)}. For example, 
		$\mathcal X_n:=(\{0,1+1/n\},0)$ converges to $\mathcal X:=(\{0,1\},0)$, while $\pcball{\mathcal X}{1}_n=(\{0\},0)$ does not converge to $\pcball{\mathcal X}{1}=\mathcal X$. 
		Also, formulas like $\inf\{\epsilon: d(\pcball{\mathcal X}{1/\epsilon}_1,\pcball{\mathcal X}{1/\epsilon}_2)\leq\epsilon\}$ do not satisfy the triangle inequality ({a counter example is constructed in Example~\ref{ex:flaws} below}) and also do not treat the exceptional radii shown in the above example (equation~\eqref{eq:beingclose-khodam} below is a correction of this formula). 
		Some earlier works in the literature have made these mistakes and their definitions or proofs need to be corrected; e.g., Appendix A.2.6 of~\cite{bookDaVe03I} has the monotonicity issue and Subsection~\ref{subsec:isometries} below mentions references having the other issues.
		\unwritten{
			Finally, for the counter example mentioned above, let $N$ be large and consider the following subsets of $l_1$:
			\begin{eqnarray*} 
				\mathcal X_1&:=&\{0\}\cup \{\pm \frac{j}{N} e_k: j\geq N, 3^{j-1}\leq k<3^j \},\\
				\mathcal X_2&:=&\{0\}\cup \{\pm \frac{j+1}{N} e_k: j\geq N, 3^{j-1}\leq k<3^j \},\\
				\mathcal X_3&:=&\{0\}\cup \{\pm \frac{j+1/N}{N} e_k: j\geq N, 3^{j-1}\leq k<3^j \},\\
			\end{eqnarray*}
			rooted at $0$.
			For every $r\geq 1$, one has $\gh(\pcball{\mathcal X}{r}_1,\pcball{\mathcal X}{r}_2)\geq 1$. However, for $r=N+1/N$, one has $\gh(\pcball{\mathcal X}{r}_1,\pcball{\mathcal X}{r}_3)=1/N^2$ and for $r=N$, one has $\gh(\pcball{\mathcal X}{r}_2,\pcball{\mathcal X}{r}_3)=1/N$.
		}
	\end{remark}
	
	Returning to~\eqref{eq:integral} in the general case, the integral is well defined if $d(\pcball{\mathcal X}{r}_1,\pcball{\mathcal X}{r}_2)$ is a measurable function of $r$. In fact, it is usually a c\`ad\`ag function. If so, it is immediate that~\eqref{eq:integral} is a pseudo-metric. Extra effort has been made to show that it is indeed a metric in various examples. Here, we extend the simple proof of~\cite{Kh19ghp} which is based on a generalization of K\"onig's infinity lemma (under the conditions specified later). 
	However, separability and completeness (or being a Borel subset of a complete space) seem to require further assumptions which are discussed below. \unwritten{Without a partial order, the metric is not practical. separability is also an issue, since convergence in $\mathcal C_{\tau}$ should imply convergence in $\mathcal D$.}
	
	As mentioned earlier, the idea of~\eqref{eq:beingclose} is formulated in various ways in the literature. 
	Some papers use~\eqref{eq:integral}, approximate isometries or embedding into a common metric space (for the latter, see Lemma~\ref{lem:convergence2} below).
	Here, we prefer to use the formulation in the previous work~\cite{Kh19ghp}, which we think is easier to axiomatize and more convenient to use: For $\epsilon<1$, 
	\begin{equation}
		\label{eq:beingclose-khodam}
		\begin{minipage}{.85\textwidth}
			\emph{$d(\mathcal X, \mathcal Y)<\epsilon$ when there exists a subspace $\mathcal Y'$ of $\mathcal Y$ \emph{between} $\pcball{\mathcal Y}{1/\epsilon-\epsilon}$ and $\pcball{\mathcal Y}{1/\epsilon+\epsilon}$ such that $d(\pcball{\mathcal X}{1/\epsilon},\mathcal Y')< \epsilon/2$ and vice versa.}
		\end{minipage}
	\end{equation}
	See~\eqref{eq:ghf-noncompact} for a more formal definition.	
	This metric generates the same topology as~\eqref{eq:integral} for (measured) metric spaces and has a Strassen-type property (similarly to Subsection~\ref{subsec:strassen}). The main task in the present section is to extend this definition to metric spaces equipped with additional structures, together with extending the results and proofs of~\cite{Kh19ghp} (being a metric, completeness, separability, etc). This framework will be applied to various types of additional structures in Subsection~\ref{subsec:examples-noncompact} and Section~\ref{sec:special}.
	
	To formalize the above definitions, we require more than a functor $\tau:\mathfrak{Comp}\to\mathfrak{Met}$. The philosophy of~\eqref{eq:beingclose} is based on the assumption that an additional structure $a$ on $X$ can be \textit{truncated} to a compact subset $K\subseteq X$. For instance, if $a$ is a measure on $X$ (resp. a closed subset of $X$), one can let the truncation be $a|_K$ (resp. $a\cap K$). In general, for any isometric embedding $f:Y\to X$, we assume a \textit{truncation map} $\tau^t_f:\tau(X)\to \tau(Y)$ is given. We require $\tau^t$ be a functor, as explained later. In addition, the term \textit{`between'} used in~\eqref{eq:beingclose-khodam} assumes that one can compare two additional structures on the same space. So we will assume that $\tau(X)$ is a metric space equipped with a partial order. 
	\ali{It should be noted that, while a partial order is not explicitly assumed in~\eqref{eq:integral}, it is still used for proving the properties of~\eqref{eq:integral} in the specific examples in the literature.}
	\unwritten{It should be noted that the metric~\eqref{eq:integral} can be generalized without a partial order as well (only with a truncation functor and a c\`adl\`ag assumption), but extending the topological properties of the metric seems to require a partial order.}
	
	An obvious limitation of the philosophy of~\eqref{eq:beingclose} is that one can only consider the additional structures that are uniquely determined by their truncation to compact subsets (if not, one might try to enlarge the set of additional structures under study to ensure that this property is satisfied; e.g., the example of \textit{ends} in Subsection~\ref{subsec:ends}). Therefore, we start by only having a functor on $\mathfrak{Comp}$. Then, when $X$ is non-compact, we \textit{define} an additional structure on $X$ as an abstract object that its truncations to compact subsets of $X$ are known (Definition~\ref{def:phi}). This is an instance of the notion of \textit{inverse limits}. 
	
	{
		\begin{example}
			\label{ex:flaws}
			Finally, here is the counter example promised in Remark~\ref{rem:flaws}. Let $N$ be large, $\delta_1:=0$, $\delta_2:=1$, $\delta_3:=1/N$  and consider the following subsets of $l_1$ for $i=1,2,3$:
			$\mathcal X_i:=\{0\}\cup \frac 1 N\{\pm (j+\delta_i) e_k: j\geq N, 3^{j-1}\leq k<3^j \}$.
			For every $r\geq 1$, one has $\gh(\pcball{\mathcal X}{r}_1,\pcball{\mathcal X}{r}_2)\geq 1$. However, for $r=N+1/N$, one has $\gh(\pcball{\mathcal X}{r}_1,\pcball{\mathcal X}{r}_3)=1/N^2$ and for $r=N$, one has $\gh(\pcball{\mathcal X}{r}_2,\pcball{\mathcal X}{r}_3)=1/N$.
		\end{example}
	}

	\subsection{The Space $\mathcal D$}
	\label{subsec:C'}
	
	Now, the precise definitions regarding the framework are presented.	
	Let $\mathfrak{Pos}$ be the category of partially ordered sets (abbreviated by \textit{posets}). The symbol $\leq$ is used to denote the order on any poset. A morphism between objects $A$ and $A'$ of $\mathfrak{Pos}$ is an order-preserving function $f:A\to A'$; i.e., if $a_1\leq a_2$, then $f(a_1)\leq f(a_2)$. 
	Let $\mathfrak{Pos_m}$ be the category of \textit{metric posets} defined as follows. Every object of $\mathfrak{Pos_m}$ is an extended metric space $A$ equipped with a partial order such that for all $a\in A$, the \defstyle{lower cone} $\{a'\in A: a'\leq a \}$ \ali{and the \defstyle{upper cone} $\{a'\in A: a\leq a' \}$ are closed subsets of $A$ (in most of the examples, the lower cones are compact, but this will be assumed only in some results in Subsection~\ref{subsec:topology})}. A morphism between objects $A$ and $A'$ of $\mathfrak{Pos_m}$ is a function $f:A\to A'$ which is both an isometric embedding and is order-preserving. \del{The reader can verify that the sets in Examples~\ref{ex:measures-noncompact}, \ref{ex:closedSubsets-noncompact} and~\ref{ex:multiple-noncompact} are elements of $\mathfrak{Pos_m}$ in the case where the underlying metric space $X$ is compact. }
	
	From now on, we assume a functor $\tau:\mathfrak{Comp}\to \mathfrak{Pos_m}$ is given.
	The examples of compact subsets and finite measures will be updated step by step to illustrate the definitions (see the conclusion in Example~\ref{ex:basicnoncompact-conclusion}).
	
	\begin{example}
		\label{ex:basic-noncompact}
		Let $\tau(X)$ be the set of compact subsets of $X$, including the empty set (equipped with the Hausdorff metric and the inclusion order), or the set of finite measures on $X$ (equipped with the Prokhorov metric and the natural partial order $\mu_1\leq \mu_2$). It is easy to check that all of the above assumptions are satisfied. Note that in the first example, $\tau(X)$ is an extended metric space. It is necessary to include the empty set for the assumptions given later.
	\end{example}
	
	\unwritten{Note that forgetting the metric gives a natural functor from $\mathfrak{Pos_m}$ to $\mathfrak{Pos}$. Here, for every morphism $f:A\to A'$ in $\mathfrak{Pos_m}$, its corresponding morphism in $\mathfrak{Pos}$ is also denoted by $f:A\to A'$. This abuse of notation makes no confusion according to the context.} 
	
	
	\unwritten{
		\begin{lemma}
			\label{lem:f(a)<a}
			For\mar{It seems that we don't need this anymore!} every isometry $f:X\to X$ and every $a\in \tau(X)$, if $\tau_f(a)\leq a$, then $\tau_f(a)=a$.
		\end{lemma}
		\begin{proof}
			Since $\tau_f$ is an isometric embedding, is order preserving and $\tau_f(a)\leq a$, the cone $C:=\{a'\in\tau(X): a'\leq a \}$ is mapped isometrically into itself under $\tau_f$. Since $C$ is compact, $\tau_f$ should be surjective on $C$. But if $\tau_f(a)<a$, then $a\not\in \tau_f(C)$, which is a contradiction.
			%
		\end{proof}
	}
	
	%
		%
	
	\ali{Assume that for every isometric embedding $f:X\to Y$, a \textbf{truncation map} $\tau^t_f:\tau(Y)\to \tau(X)$ is given. By letting $\tau^t(X)$ be the underlying poset of $\tau(X)$, we assume $\tau^t:\mathfrak{Comp}\to\mathfrak{Pos}$ is a \textit{contra-variant functor} (defined similarly to Definition~\ref{def:functor}, but with the morphisms in the reverse direction). It is called \defstyle{the truncation functor} here.}
	Also, assume that for every morphism $f:X\to Y$ of $\mathfrak{Comp}$, $a\in \tau(X)$ and $b\in \tau(Y)$, 
	\begin{equation}
		\label{eq:truncation}
		\begin{array}{rcl}
			\tau^t_f \circ \tau_f(a)&=& a,\\
			\ali{\restrict{b}{f(Y)}:=}\tau_f \circ \tau^t_f(b)&\leq & b.
		\end{array}
	\end{equation}
	
	\begin{example}
		\label{ex:basicnoncompact-trunc}
		The reader can verify that these assumptions hold for the functors of compact subsets and finite measures, where the truncation maps are defined by $\tau^t_f(K):=f^{-1}(K)$ and $\tau^t_f(\mu)(A):=\mu(f(A))$ respectively. Note that $f^{-1}(K)$ might be the empty set, which justifies why the empty set was included in Example~\ref{ex:basic-noncompact}.
	\end{example}

	
	Now, $\tau$ is extended to \bcm s as follows. Here, $\varphi(X)$ represents the set of additional structures on $X$ when $X$ is not compact.
	
	\begin{definition}[The functor $\varphi$]
		\label{def:phi}
		For a \bcm{} $X$, let $I_X$ be the set of compact subsets of $X$. In the language of category theory, let $\varphi(X)$ be the \textit{inverse limit} in $\mathfrak{Pos}$ of the diagram consisting of the objects $\tau(K)$ for $K\in I_X$, where the arrows are the truncation maps \new{(an explicit construction of the inverse limit is given in Example~\ref{ex:basicnoncompact-phi} below). The abstract definition of the inverse limit $\varphi(X)$} is a poset equipped with order-preserving functions $\varphi(X)\to\tau(K)$ for every $K\in I_X$ (which are called \defstyle{truncation maps} again) such that
		\begin{enumerate}[label=(\roman*)]
			\item \label{def:phi-1} For every $K_1\subseteq K_2\in I_X$ and the inclusion map $\iota:K_1\to K_2$, the following diagram is commutative (i.e. the composition of two of the maps is identical to the third map):
			\[ 
			\begin{tikzcd}
				&\arrow[ld] \varphi(X) \arrow{rd}{}&\\
				\tau(K_1)   && \arrow{ll}{\tau^t_{\iota}}\tau(K_2)
			\end{tikzcd}
			\]
			\item \label{def:phi-2} For every other poset $\varphi'(X)$ equipped with truncation maps $\varphi'(X)\to \tau(K)$ as above, there is a unique order-preserving map $\varphi'(X)\to \varphi (X)$ such that the following diagram is commutative for every $K\in I_X$:
			\[ 
			\begin{tikzcd}
				\varphi'(X)\arrow[rd] \arrow{rr} && \arrow{ld}{}\varphi(X)\\
				& \tau(K) &
			\end{tikzcd}
			\]
		\end{enumerate}
	\end{definition}
	The second condition is called  \textit{the universal property}. Such a $\varphi(X)$ (together with the truncation maps) is defined uniquely up to isomorphism (in $\mathfrak{Pos}$). 
	
	\unwritten{Remark: If the upper cones are compact, one can defined $\varphi(X)$ with a forward limit without using truncations as well.}
	
	\begin{example}[\new{Construction of $\varphi(X)$}]
		\label{ex:basicnoncompact-phi}
		In explicit examples, it is usually clear what $\varphi(X)$ is. For examples, for the example of compact subsets (resp. finite measures) discussed above, $\varphi(X)$ can be the set of closed subsets of $X$ (resp. boundedly-finite measures). As another example, if $X$ is compact, one can let $\varphi(X)$ be simply the underlying poset of $\tau(X)$. In the general case, $\varphi(X)$ can be constructed explicitly as follows: 
		An element of $\varphi(X)$ is a collection of elements $a:=\{a_K: K\in I_X\}$ that are compatible with truncation maps; i.e., for every inclusion map $\iota:K_1\to K_2$, one has $\tau^t_{\iota}(a_{K_2})=a_{K_1}$. The partial order can be defined by $a\leq a'$ when $a_K\leq a'_K$ for all $K\in I_X$. Note that $\tau(K)$ is naturally embedded in $\varphi(X)$ for every compact $K\subseteq X$ (see Definition~\ref{def:phi2} below).
	\end{example}
	
	\ali{Note that no metric is assumed on $\varphi(X)$. In fact, one cannot define such a metric in a functorial way for non-rooted \bcm s (see Definition~\ref{def:metricOnPhi}).}
	
	%
	%
	
	\begin{definition}
		\label{def:phi2}
		Let $\varphi$ be defined as above.
		For every isometric embedding $f:X\to X'$ between \bcm s $X$ and $X'$, define the order-preserving maps $\varphi_f:\varphi(X)\to \varphi(X')$ and $\varphi^t_f:\varphi(X')\to\varphi(X)$ as follows (e.g., for the example of measures, these are the push-forward map and the natural truncation map respectively). For every $K'\in I_{X'}$, compose the truncation map $\varphi(X)\to \tau(f^{-1}(K'))$ with $\tau_g:\tau(f^{-1}(K'))\to \tau(K')$, where $g:=f|_{f^{-1}(K')}$. The universal property of $\varphi(X')$ (condition~\ref{def:phi-2} of Definition~\ref{def:phi}) gives a map $\varphi(X)\to \varphi(X')$, which we call $\varphi_f$. Also, for every $K\in I_X$, compose the truncation maps $\varphi(X')\to \tau(f(K))$ and $\tau^t_h:\tau(f(K))\to \tau(K)$, where $h:=f|_K$. The universal property of $\varphi(X)$ gives a map $\varphi(X')\to \varphi(X)$, which we call $\varphi^t_f$.
		%
	\end{definition}
	
	It can be seen that $\varphi$ is a functor and $\varphi^t$ is a contravariant functor from the category of \bcm s to $\mathfrak{Pos}$. In addition, \eqref{eq:truncation} holds for $\varphi$ and $\varphi^t$.

	
	\begin{definition}
		\label{def:C'}
		Let $\mathcal D:=\mathcal D_{\tau,\tau^t}$ be the set of isomorphism classes of tuples $\mathcal X=(X,o; a)$, where $X$ is a \bcm, $o\in X$ and $a\in \varphi(X)$ (it can be seen that $\mathcal D$ is indeed a set). 
		Also, let $\mathcal C$ be the subset of $\mathcal D$ comprising the tuples in which $X$ is compact. Equivalently, $\mathcal C=\mathcal C_{\tau'}$ in Definition~\ref{def:C_tau}, where $\tau':\mathfrak{Comp}\to\mathfrak{Met}$ is the functor defined by $\tau'(X):=X\times \tau(X)$ equipped with the max product metric (taking the product of $X$ and $\tau(X)$ is due to considering rooted metric spaces). 
		Note that~\eqref{eq:ghfunctor} defines a metric $d^c_{\tau'}$ on $\mathcal C$.
	\end{definition}

	\subsection{The Metric on $\mathcal D$}
	\label{subsec:metricOnC'}
	
	Now, we formalize~\eqref{eq:beingclose-khodam} to define a metric on $\mathcal D$.	
	For all tuples $\mathcal X=(X,o;a)$ as above and $r\geq 0$, let $\pcball{\mathcal X}{r}:=(\cball{r}{o},o;\varphi^t_{\iota_r}(a))$, where $\iota_r:\cball{r}{o}\hookrightarrow X$ is the inclusion map. \ali{Here, $\varphi^t_{\iota_r}(a)$ is the truncation of $a$ to $\cball{r}{o}$.}
	For all tuples $\mathcal X:=(X,o;a)$ and $\mathcal X':=(X',o';a')$, 
	define $\mathcal X'\preceq\mathcal X$ if $X'$ is a subspace of $X$, $o'=o$ and $a'\leq \varphi^t_{\iota}(a)$, where $\iota$ is the inclusion map (the latter is equivalent to $\varphi_{\iota}(a')\leq a$). 
	For tuples $\mathcal X=(X,o_X;a_X)$ and $\mathcal Y=(Y,o_Y; a_Y)$ and $0\leq\epsilon\leq 1$, let 
	\begin{equation}
		\label{eq:a_r}
		a_{\epsilon}(\mathcal X,\mathcal Y):=\inf\{d^c_{\tau'}(\pcball{\mathcal X}{1/\epsilon}, \mathcal Y')\},
	\end{equation}
	where \unwritten{The condition is needed s.th. the topology extends the GH (I think)}
	the infimum is over all $\mathcal Y'$ such that $\pcball{\mathcal Y}{1/\epsilon-\epsilon}\preceq \mathcal Y'\preceq \mathcal Y$\del{ (one can also remove the condition $\pcball{\mathcal Y}{r-\epsilon}\preceq\mathcal Y'$ and Theorems~\ref{thm:metric-noncompact}, \ref{thm:functor-precompact-noncompact} and~\ref{thm:functor-polish-noncompact} will remain valid)}. Define \del{the distance between $\mathcal X$ and $\mathcal Y$ by }
	\begin{equation}
		\label{eq:ghf-noncompact}
		d(\mathcal X,\mathcal Y):= \inf\{ \epsilon\in (0,1]: 
		a_{\epsilon}(\mathcal X,\mathcal Y) \vee a_{\epsilon}(\mathcal Y,\mathcal X)
		< \frac{\epsilon}{2}\},
	\end{equation}
	with the convention that $\inf \emptyset:=1$. \ali{Note that because of the bound $\epsilon/2$ in the above definition, one could assume $\mathcal Y'\preceq \mathcal Y^{(1/\epsilon+\epsilon)}$ and~\eqref{eq:ghf-noncompact} would not change. This is due to the fact that an $\epsilon/2$-perturbation in the rooted GH metric results in an at most $\epsilon$-perturbation of the radius.}
	
	To ensure that this equation defines a metric on $\mathcal D$, we assume that the following further assumptions hold. 
	
	\begin{definition}
		\label{def:dZfg}
		Let $\mathcal X=(X,o_X;a_X)$ and $\mathcal Y=(Y,o_Y;a_Y)$ be compact. For isometric embeddings $f:X\to Z$ and $g:Y\to Z$, define
		\[
		d^{Z,f,g}_{\tau}(\mathcal X,\mathcal Y):=d_H(f(X),g(Y))\vee d(f(o_X),g(o_Y))\vee d(\tau_{f}(a_X),\tau_{g}(a_Y)).
		\]	
	\end{definition}
	Note that by taking infimum over all $Z,f,g$, one obtains $d^c_{\tau'}(\mathcal X, \mathcal Y)$.
	\begin{assumption}
		\label{assump:subsetLemma1}
		In the setting of Definition~\ref{def:dZfg}, assume that for all $\mathcal X'\preceq\mathcal X$, there exists $\mathcal Y'\preceq\mathcal Y$ such that 
		$
		d^{Z,f,g}_{\tau}(\mathcal X', \mathcal Y')\leq d^{Z,f,g}_{\tau}(\mathcal X, \mathcal Y) 
		$.
		%
	\end{assumption}
	
	
	
	In this assumption, if $\mathcal X'=(X',o_X; a'_X)$ and $\mathcal Y'=(Y', o_Y; a'_Y)$, one can always let $Y'$ be the largest possible choice $Y':=\{y\in Y: d(g(y),f(X'))\leq \epsilon\}$, where $\epsilon:=d^{Z,f,g}_{\tau}(\mathcal X, \mathcal Y)$, and it remains to find $a'_Y\in \tau(Y')$. This implies that if $\mathcal X'\preceq \pcball{\mathcal X}{r}$, then $\mathcal Y'\preceq \pcball{\mathcal Y}{r+2\epsilon}$. For the other way, we add the following assumption.

	\unwritten{Note also that $\mathcal Y'$ in the above assumption can be chosen such that $d^c_{\tau'}(\mathcal X', \mathcal Y')\leq d^c_{\tau'}(\mathcal X, \mathcal Y)$.}
	
	
	\begin{assumption}
		\label{assump:subsetLemma2}
		In the previous assumption, assume that if $\pcball{\mathcal X}{r}\preceq \mathcal X'\preceq \mathcal X$, then $\mathcal Y'$ can be chosen such that $\pcball{\mathcal Y}{r-2\epsilon}\preceq \mathcal Y'\preceq \mathcal Y$, where $\epsilon:= d^{Z,f,g}_{\tau}(\mathcal X,\mathcal Y)$ (assuming $r\geq 2\epsilon$). 
	\end{assumption}
	
	\unwritten{\begin{remark}
			\ali{In fact, it is enough that Assumptions~\ref{assump:subsetLemma1} and~\ref{assump:subsetLemma2} hold whenever $d^{Z,f,g}_{\tau}(\mathcal X, \mathcal Y)\leq \frac 12$. This is due to the term $\epsilon\in (0,1]$ in~\eqref{eq:ghf-noncompact} and might simplify the arguments in some examples; e.g., in Lemma~\ref{lem:closedProcess}.}
	\end{remark}}
	
	\begin{example}
		\label{ex:subsetLemma}
		\unwritten{This holds even for continuous and cadlag curves, although are not 1-Lipschitz (?).}
		\ali{These assumptions are straightforward for the functor of compact subsets \new{(see the construction in the proof of Lemma~\ref{lem:basic-hausdorff})}. For finite measures, the claim can be proved by the same idea and is already proved in Lemma~3.8 of~\cite{Kh19ghp} by using Strassen's theorem (see also Lemma~1.4 of~\cite{KlLo15markfunction} for Assumption~\ref{assump:subsetLemma1}).}
		\ali{The conditions are also straightforward for the functors of points, curves, marked measures and marked compact subsets, discussed in Subsection~\ref{subsec:examples-noncompact} below.}
	\end{example}
	
	It should be noted that for Assumptions~\ref{assump:subsetLemma1} and~\ref{assump:subsetLemma2} to hold, the metric on $\tau(X)$ should be carefully chosen. 
	Subsection~\ref{subsec:examples-noncompact} discusses some new examples where this assumption does not hold and provides other suitable metrics. 
	It is not known by the author whether or not the next theorems hold for other metrics on $\mathcal D$ (e.g., \eqref{eq:integral}) under weaker assumptions.

	
	
	Assumptions~\eqref{assump:subsetLemma1} and~\eqref{assump:subsetLemma2} straightforwardly lead to the following result, the proof of which is left to the reader (see Lemmas~3.11 and~3.12 of~\cite{Kh19ghp} for the case of measures).
	
	\begin{lemma}
		\label{lem:monotone}
		The function $a_{\epsilon}(\mathcal X,\mathcal Y)$ is non-decreasing in $\epsilon$ in the interval $[0,d(\mathcal X,\mathcal Y))$. Also, if $d(\mathcal X,\mathcal Y)<\epsilon<1$, then $a_{\epsilon}(\mathcal X,\mathcal Y)<\epsilon/2$.
	\end{lemma}

	\begin{theorem}[Metric on $\mathcal D$]
		\label{thm:metric-noncompact}
		Let $\tau, \tau^t$ and $\mathcal D$ be as above. Assume that $\tau$ is pointwise-continuous. 
		Then,~\eqref{eq:ghf-noncompact} defines a metric on $\mathcal D$. 
	\end{theorem}
	\ali{Note that the metric~\eqref{eq:ghf-noncompact} is always bounded by 1 even if $\tau(X)$ is an extended metric space.}
	
	\begin{proof}
		It is clear that~\eqref{eq:ghf-noncompact} depends only on equivalence classes of $\mathcal X$ and $\mathcal Y$, is symmetric and $d(\mathcal X, \mathcal X)=0$. For the triangle inequality, assume $d(\mathcal X, \mathcal Y)<\epsilon$ and $d(\mathcal Y,\mathcal Z)<\delta$. By Lemma~\ref{lem:monotone}, $\pcball{\mathcal X}{1/\epsilon}$ is $\frac{\epsilon}{2}$-close to some subspace of $\mathcal{Y}$ and $\pcball{\mathcal Y}{1/\delta}$ is $\frac{\delta}{2}$-close to some subspace of $\mathcal Z$. By using Assumption~\ref{assump:subsetLemma1} two times, one obtains that $\pcball{\mathcal X}{1/(\epsilon+\delta)}$ is $\frac{\epsilon}{2}$-close to some $\mathcal Y'\preceq\mathcal Y$ and $\mathcal Y'$ is $\frac{\delta}{2}$-close to some $\mathcal Z'\preceq\mathcal Z$. In addition, Assumption~\ref{assump:subsetLemma2} implies that $\pcball{\mathcal Z}{1/(\epsilon+\delta)-(\epsilon+\delta)}\preceq\mathcal Z'$ for a suitable choice of $\mathcal Z'$. So, one obtains $a_{\epsilon+\delta}(\mathcal X,\mathcal Z)<\frac 12(\epsilon+\delta)$ and the triangle inequality is proved.
		
		Finally, assume $d(\mathcal X,\mathcal Y)=0$. Lemma~\ref{lem:monotone} implies that for every $\epsilon>0$, $\pcball{\mathcal X}{1/\epsilon}$ is $\frac{\epsilon}{2}$-close to some subspace of $\mathcal Y$. By letting $\epsilon=\frac 1 n \to 0$ and keeping $r$ fixed, Assumption~\ref{assump:subsetLemma1} implies that $\pcball{\mathcal X}{r}$ is $\frac{1}{2n}$-close to some subspace $\mathcal Y_n=(Y_n,o_Y;b_n)$ of $\mathcal Y$ and $\mathcal Y_n\preceq\pcball{\mathcal Y}{r+1}$. 
		By passing to a subsequence, assume that $Y_n\subseteq\cball{r+1}{o_Y}$ converges, namely to $Y'$. So $(Y',o_Y)$ is isometric to $(\cball{r}{o_X},o_X)$, and hence, there exists $b'\in \tau(Y')$ such that $\pcball{\mathcal X}{r}$ is isometric to $\mathcal Y':=(Y',o_Y; b')$. A challenge is to show that $b'$ can be chosen such that a subsequence of $(b_n)_n$ converges to $b'$ after embedding them in $\tau(\cball{r+1}{o_Y})$ (this is immediate if the lower cones are compact, see Assumption~\ref{assump:compactcones}, but we have not assumed it yet). 
		This will be proved in the next paragraph. Assuming this, one can use closedness of cones to deduce that $\mathcal Y'\preceq Y$. In addition, if $\mathcal Y_n$ is chosen according to Assumption~\ref{assump:subsetLemma2}, then closedness of upper cones implies that $\pcball{\mathcal Y}{r-\epsilon}\preceq\mathcal Y'$ for every $\epsilon$. This implies that $\pcball{\mathcal X}{r/2}$ is isometric to $\pcball{\mathcal Y}{r/2}$ for every $r$. Note that the set of isomorphisms between $\pcball{\mathcal X}{r/2}$ and $\pcball{\mathcal Y}{r/2}$ is compact under the sup metric (which is implied by pointwise-continuity). Now, by a version of K\"onig's infinity lemma for compact sets (Lemma~3.13 in~\cite{Kh19ghp}), one can choose an isometry between $\pcball{\mathcal X}{r/2}$ and $\pcball{\mathcal Y}{r/2}$ for every integer $r$ and glue them to obtain an isometry between $\mathcal X$ and $\mathcal Y$. So, $\mathcal X$ is isometric to $\mathcal Y$.
		
		It remains to show that $b'$ can be chosen such that a subsequence of $(b_n)_n$ converges to $b'$ (after embedding them in $\tau(\cball{r+1}{o_Y})$). {This is proved below similarly to the proof of Lemma~\ref{lem:functor-proper}.}
		First, one can safely replace $Y_n$ (resp. $b_n$) by $Y'_n:=Y_n\cup Y'$ (resp. the image of $b_n$ in $\tau(Y'_n)$). So we might assume $Y'\subseteq Y_n$ from the beginning. Let $\iota_n:Y'\to Y_n$ be the inclusion map. Since $\mathcal Y_n\to \mathcal Y'$, Lemma~\ref{lem:common2} gives a compact metric space $Z$ and isometric embeddings $f:Y'\to Z$ and $f_n:Y_n\to Z$ such that $f_n(Y_n)\to f(Y')$, $f_n(o_Y)\to f(o_Y)$ and $\tau_{f_n}(b_n)\to \tau_{f}(b')$. By Lemma~\ref{lem:embedding} and taking a subsequence, one might assume that $f_n\circ \iota_n:Y'\to Z$ is convergent pointwise, namely converging to $g:Y'\to Z$. Note that 
		\[\mathrm{Im}(g)=\lim_n f_n(Y')=\lim_n {f_n(Y_n)}=\mathrm{Im}(f),\]
		where the limits are under the Hausdorff metric. So there exists an isometry $h:Y'\to Y'$ such that $f=g\circ h$. Since $g(o_Y)=f(o_Y)$, one has $h(o_Y)=o_Y$. Let $b'':=\tau_h(b')$. Pointwise-continuity implies that $\tau_{f_n}(\tau_{\iota_n}(b''))\to \tau_g(b'')=\tau_f(b')$. Since $\tau_{f_n}(b_n)\to \tau_f(b')$ as well, it follows that $d\big(\tau_{f_n}(\tau_{\iota_n}(b'')), \tau_{f_n}(b_n)\big)\to 0$; i.e., $d(\tau_{\iota_n}(b''), b_n)\to 0$. 
		Now, since $(Y',o_Y; b'')$ is isomorphic to $(Y', o_Y; b')$ (under the isometry $h$), one could choose $b''$ instead of $b'$ from the beginning. This completes the proof.
	\end{proof}
	
	
	\subsection{The Topology of $\mathcal D$}
	\label{subsec:topology}
	
	From now on, we assume that $\tau$ is pointwise-continuous, and hence, \eqref{eq:ghf-noncompact} is a metric by Theorem~\ref{thm:metric-noncompact}.
	
	\begin{lemma}
		\label{lem:convergence}
		Let $\mathcal X,\mathcal X_1,\mathcal X_2,\ldots\in\mathcal D$. Then $\mathcal X_n\to \mathcal X$ if and only if for every $r>0$ and $\epsilon>0$, there exists a sequence $\pcball{\mathcal X}{r-\epsilon} \preceq\mathcal X'_n\preceq \pcball{\mathcal X}{r+\epsilon}$  such that $d^c_{\tau'}(\pcball{\mathcal X}{r}_n, \mathcal X'_n)\to 0$.
	\end{lemma}
	Further characterizations of convergence will be provided in Theorems~\ref{thm:convergence} and~\ref{thm:integral} and Lemma~\ref{lem:convergence2}.
	\new{
		\begin{proof}
			It is straightforward to see that $d(\mathcal X_n,\mathcal X)\to 0$ if and only if for every $\epsilon>0$, one has $a_{\epsilon}(\mathcal X_n,\mathcal X)\to 0$. If so, and if $r=1/\epsilon$, then the definition of $a_{\epsilon}$ in~\eqref{eq:a_r} gives the desired sequence $\mathcal X'_n$. This implies the same claim for smaller values of $r$ as well by Assumption~\ref{assump:subsetLemma2}. Conversely, one can use the condition for $r:=1/\epsilon$ to deduce that $a_{\epsilon}(\mathcal X_n,\mathcal X)\to 0$.
		\end{proof}
	}
	

	\begin{theorem}[Completeness and Separability]
		\label{thm:functor-polish-noncompact}
		In the setting of Theorem~\ref{thm:metric-noncompact}, assume that $\tau$ is Hausdorff-continuous. If $\tau(X)$ is complete (resp. separable) for every compact metric space $X$, then $\mathcal D$ is also complete (resp. separable).
	\end{theorem}
	
	One can also replace the assumption of Hausdorff-continuity by the assumptions in Remark~\ref{rem:Hcont-weaker}.
	
	\begin{proof}
		
		The definition of the metric~\eqref{eq:ghf-noncompact} directly implies that $d(\mathcal X, \pcball{\mathcal X}{r})\leq 1/r$ for every $\mathcal X\in\mathcal D$. So $\mathcal C_{\tau'}$ is dense as a subset of $\mathcal D$. By Assumptions~\ref{assump:subsetLemma1} and~\ref{assump:subsetLemma2}, one can show that the induced topology on $\mathcal C_{\tau'}$ is coarser than that of the metric $d^c_{\tau'}$. So, separability of $\mathcal D$ is implied by Theorem~\ref{thm:functor-polish}. For completeness, assume $(\mathcal X_n)_n$ is a Cauchy sequence in $\mathcal D$. By taking a subsequence, it is enough to assume $d(\mathcal X_n,\mathcal X_m)\leq 2^{-n}$ for every $m\geq n$. Let $r_n:=2^n$ and $\epsilon_{n,m}:=2^{-n-1}+\cdots + 2^{-m-1}$.
		Let $n$ be fixed. 
		Inductively, construct subspaces $\mathcal X_{n,m}\preceq\mathcal X_m$ for all $m\geq n$ such that $\mathcal X_{n,n}=\pcball{\mathcal X}{r_n}_n$, $d^c_{\tau'}(\mathcal X_{n,m},\mathcal X_{n,m+1})\leq 2^{-m}$ and $\pcball{\mathcal X}{r_n-\epsilon_{n,m}}_m\preceq\mathcal X_{n,m}\preceq \pcball{\mathcal X}{r_n+\epsilon_{n,m}}_m$. This is possible because of
		Assumptions~\ref{assump:subsetLemma1} and~\ref{assump:subsetLemma2}. Therefore, $(\mathcal X_{n,m})_m$ is a Cauchy sequence under $d^c_{\tau'}$, and hence is convergent by Theorem~\ref{thm:functor-polish}. Let $\mathcal Y_n:=\lim_m \mathcal X_{n,m}$.
		
		Note that $\mathcal X_{n,m}$ is a subspace of $\mathcal X_{n+1,m}$ and contains the large ball $\pcball{\mathcal X}{r_n-\epsilon_{n,m}}_{n+1,m}$. So, Assumptions~\ref{assump:subsetLemma1} and~\ref{assump:subsetLemma2} imply that $\mathcal X_{n,m}$ is close to some subspace of $\mathcal Y_{n+1}$. This way, one finds a sequence of subspaces of $\mathcal Y_{n+1}$ that converge to $\mathcal Y_n$ under $d^c_{\tau'}$. So, an argument similar to the proof of Theorem~\ref{thm:metric-noncompact} shows that $\mathcal Y_n$ is isomorphic to a subspace of $\mathcal Y_{n+1}$ which contains the large ball $\pcball{\mathcal Y}{r_n-1}_{n+1}$. In particular, by letting $\mathcal Z_n:=\pcball{\mathcal Y}{r_n-1}_n$, then $\mathcal Z_n$ is isomorphic to a large ball in $\mathcal Z_{n+1}$. So, the spaces $\mathcal Z_n$ can be pasted together to construct an element $\mathcal Z\in\mathcal D$ (see the explicit construction after Definition~\ref{def:phi}). It is straightforward to show that $\mathcal X_n$ converges to $\mathcal Z$. This proves the claim.
	\end{proof}
	
	\unwritten{the balls do not converge, but under stronger assumptions, have a convergent subsequence. Also, without stronger assumptions, the latter can be false.}
	
	
	\ali{For further studying convergence in $\mathcal D$, we assume the following. Most of the examples satisfy this assumption (an exception is c\`adl\`ag curves in Subsection~\ref{subsec:cadlag2}). In the next subsection, we study what happens without this assumption.}
	
	\begin{assumption}
		\label{assump:compactcones}
		For every compact $X$ and $a\in \tau(X)$, the lower cone $\{a'\in \tau(X): a'\leq a \}$ is compact.
	\end{assumption}
	
	


\begin{lemma}
	\label{lem:leftrightlimit}
	For every compact $(X,o;a)\in\mathcal C_{\tau'}$, the monotone curve $h(r):=\restrict{a}{\cball{r}{o}}$ in $\tau(X)$ (defined in~\eqref{eq:truncation}) has both right-limits and left-limits at all points. In addition, if $r<s$, then $h(r)\leq h(r^+)\leq h(s^-)\leq h(s)$. 
\end{lemma}

This curve is usually a c\`adl\`ag curve in the examples (the only exception here is Example~\ref{ex:cadlag}), but it is not required in what follows.

\begin{proof}
	Let $r<s$. For every $r<t<s$, one has $h(r)\leq h(t)\leq h(s)$. By compactness of the cone below $h(s)$, one can find $t_n\downarrow r$ such that 
	$a:=\lim_n h(t_n)$ exists and $a\leq h(s)$. By the same argument, $a\leq h(t)$ for every $t>r$. Also, since upper cones are closed, one has $a\geq h(r)$. Assume $t'_n\downarrow r$ and $a':=\lim_n h(t'_n)$ exists. Since $a\leq h(t'_n)$ for every $n$ and upper cones are closed, one gets $a\leq a'$. Similarly, $a'\leq a$, and hence, $a'=a$. This implies that $a$ is the right limit at $r$. The claim for left limits is proved similarly.
\end{proof}

For instance, for the functor of compact subsets discussed earlier, the curve $h$ is c\`adl\`ag. In particular, this implies that the curve $r\mapsto \cball{r}{o}$ is c\`adl\`ag. In the general case, a radius $r_0\geq 0$ is called a \defstyle{continuity radius of} $\mathcal X$ if both curves $r\mapsto \cball{r}{o}$ and $r\mapsto a|_{\cball{r}{o}}$ are continuous at $r_0$. Therefore, $\mathcal X$ has at most countably many discontinuity radii.

Note that for every continuity radius $r_0$ of $\mathcal X$, the curve $r\mapsto\pcball{\mathcal X}{r}$ is also continuous at $r_0$. This implies that the integral~\eqref{eq:integral} is well defined. 

\begin{theorem}[Convergence]
	\label{thm:convergence}
	Let $\mathcal X,\mathcal X_1,\mathcal X_2,\ldots\in\mathcal D$. The following are equivalent.
	\begin{enumerate}[label=(\roman*)]
		\item \label{thm:convergence:1} $\mathcal X_n\to \mathcal X$.
		\item \label{thm:convergence:3} For every continuity radius $r$ of $\mathcal X$, $\pcball{\mathcal X}{r}_n\to\pcball{\mathcal X}{r}$ under the metric $d^c_{\tau'}$.
		\item \label{thm:convergence:4} For some unbounded set $I\subseteq R^{\geq 0}$, $\pcball{\mathcal X}{r}_n\to\pcball{\mathcal X}{r}$ for every $r\in I$.
		\item \label{thm:convergence:5} $\mathcal X_n$ converges to $\mathcal X$ under~\eqref{eq:integral} (which is a metric by the next theorem).
	\end{enumerate}
\end{theorem}
\new{In addition, Lemma~\ref{lem:convergence2} below provides another convergence criterion based on embedding into a common space. The proof of the above theorem uses the following lemma.}
\begin{lemma}
	\label{lem:continuityRadius}
	If $r>0$ is a continuity radius of $\mathcal X=(X,o;a)$, then 
	\[\bigcap_{\epsilon>0} \{a'\in \tau(X): \restrict{a}{\cball{r-\epsilon}{o}}\leq a'\leq \restrict{a}{\cball{r+\epsilon}{o}} \}=\{\restrict{a}{\cball{r}{o}} \}.\]
	Hence, under Assumption~\ref{assump:compactcones}, the nested sets under intersection \new{are compact and} converge to $\{\restrict{a}{\cball{r}{o}} \}$ in the Hausdorff metric.
\end{lemma}
\begin{proof}
	Let $a'$ be in the intersection. Since the upper (resp. lower) cone at $a'$ is closed, one gets $a'\leq \lim_{\epsilon} \restrict{a}{\cball{r+\epsilon}{o}} = \restrict{a}{\cball{r}{o}}$ (resp. $a'\geq \restrict{a}{\cball{r}{o}}$). This proves the claim.
\end{proof}

\begin{proof}[Proof of Theorem~\ref{thm:convergence}]
	The implications
	\eqref{thm:convergence:3}$\Rightarrow$\eqref{thm:convergence:4}, \eqref{thm:convergence:4}$\Rightarrow$\eqref{thm:convergence:1}, \eqref{thm:convergence:3}$\Rightarrow$\eqref{thm:convergence:5} and also \eqref{thm:convergence:5}$\Rightarrow$\eqref{thm:convergence:3} are straightforward and are left to the reader (see Theorem~3.16 of~\cite{Kh19ghp} and its proof in the case of measures). Part \eqref{thm:convergence:1}$\Rightarrow$\eqref{thm:convergence:3} is implied by Lemmas~\ref{lem:convergence} and~\ref{lem:continuityRadius}
	\new{as follows: Assuming $\mathcal X_n\to \mathcal X$ and $r$ is a continuity radius, Lemma~\ref{lem:convergence} gives a sequence $\epsilon_n\to 0$ and a sequence $\mathcal X'_n$ such that $\pcball{\mathcal X}{r-\epsilon_n} \preceq\mathcal X'_n\preceq \pcball{\mathcal X}{r+\epsilon_n}$ and $d^c_{\tau'}(\pcball{\mathcal X}{r}_n, \mathcal X'_n)\to 0$. The second assertion of Lemma~\ref{lem:continuityRadius} implies that $\mathcal X'_n$ converges to $\pcball{\mathcal X}{r}$, and hence,~\ref{thm:convergence:3} is implied.}
\end{proof}

\begin{theorem}
	\label{thm:integral}
	The integral~\eqref{eq:integral} is well defined, defines a metric on $\mathcal D$ and generates the same topology as the metric~\eqref{eq:ghf-noncompact}. In addition, under the assumptions of Theorem~\ref{thm:functor-polish-noncompact}, $\mathcal D$ is complete and separable under this metric.
\end{theorem}	

\begin{proof}
	It is immediate that~\eqref{eq:integral} is a pseudo-metric. If $\mathcal X$ and $\mathcal Y$ have zero distance under~\eqref{eq:integral}, then $\pcball{\mathcal X}{r}$ is isomorphic to $\pcball{\mathcal Y}{r}$ for almost every $r$. As in the proof of Theorem~\ref{thm:metric-noncompact}, one can deduce that $\mathcal X$ is isomorphic to $\mathcal Y$ by the extension of K\"onig's infinity lemma to compact sets. Hence,~\eqref{eq:integral} is a metric. Theorem~\ref{thm:convergence} proves that the two metrics generate the same topology.
	
	Let $\epsilon>0$ be given and assume that $\mathcal X$ and $\mathcal Y$ have distance less than $\frac 12\epsilon e^{-1/\epsilon}$ under~\eqref{eq:integral}. This implies that $d^c_{\tau'}(\pcball{\mathcal X}{r}, \pcball{\mathcal Y}{r})<\frac{\epsilon}{2}$ for some $r\geq \frac{1}{\epsilon}$. So, Assumptions~\ref{assump:subsetLemma1} and~\ref{assump:subsetLemma2} imply that $a_{\epsilon}(\mathcal X, \mathcal Y)<\frac{\epsilon}{2}$; i.e., $d(\mathcal X,\mathcal Y)\leq\epsilon$. This implies that every Cauchy sequence under~\eqref{eq:integral} is also a Cauchy sequence under~\eqref{eq:ghf-noncompact}. This implies the claim.
\end{proof}

For studying pre-compactness, we add the following assumption.

\begin{assumption}
	\label{assump:Hcont-weaker:1}
	Assume also that condition~\ref{rem:Hcont-weaker:1} of Remark~\ref{rem:Hcont-weaker} holds.
\end{assumption}

\begin{lemma}
	\label{lem:conecompact}
	For every compact $\mathcal X=(X,o; a)$, the set of $\mathcal X'$ such that $\mathcal X'\preceq\mathcal X$ is compact under the metric $d^c_{\tau'}$. Hence, the infimum in~\eqref{eq:a_r} is attained.
\end{lemma}
\begin{proof}
	Let $\mathcal X_n:=(X_n,o; a_n)\preceq \mathcal X$. By the assumption of compactness of cones and taking a subsequence, one may assume $\lim X_n=:X'$ (under the Hausdorff metric) and $\lim \tau_{\iota_n}(a_n)=:b'\leq a$ exist, where $\iota_n:X_n\to X$ is the inclusion map. Assumption~\ref{assump:Hcont-weaker:1} implies that $b'=\tau_{\iota'}(a')$ for some $a'\in \tau(X')$. This proves the claim.
\end{proof}


\begin{theorem}[Pre-compactness]
	\label{thm:functor-precompact-noncompact}
	\unwritten{A modification of the proof is required. This is written in the 1-volume version.}
	Under the assumptions of Theorem~\ref{thm:metric-noncompact}, a subset $\mathcal A \subseteq \mathcal D$ is relatively compact if and only if for every $r\geq 0$, the set of (equivalence classes of the)
	balls $\mathcal A_r:=\{\pcball{\mathcal X}{r}: \mathcal X\in \mathcal A \}$ is relatively compact under the metric $d^c_{\tau'}$.
\end{theorem}

\begin{proof}
	The forward direction of the claim is a simple corollary of Lemmas~\ref{lem:convergence} and~\ref{lem:conecompact}. For the other direction, 
	assume $\mathcal A_r$ is relatively compact for every $r$. Let $(\mathcal X_n)_n$ be an arbitrary sequence in $\mathcal A$. By passing to a subsequence and a diagonal argument, one can assume that for every $m\in\mathbb N$, the sequence $(\pcball{\mathcal X}{m}_n)_n$ is convergent under the metric $d^c_{\tau'}$, namely to $\mathcal Y_m$. 
	\ali{The rest of the proof is similar to that of Theorem~\ref{thm:functor-polish-noncompact} and is sketched here.}
	By using Assumptions~\ref{assump:subsetLemma1} and~\ref{assump:subsetLemma2}, one can show that $\mathcal Y_{m-1}$ is isomorphic to a pointed subspace of $\mathcal Y_m$ which contains a large ball; e.g., $\pcball{\mathcal Y}{m-2}_m\preceq \mathcal Y_{m-1}$. It follows that one can paste the spaces $\pcball{\mathcal Y}{m-1}_m$ for all $m$ to obtain a $\mathcal Y\in\mathcal D$. In addition, Lemma~\ref{lem:convergence} implies that $\mathcal X_n\to \mathcal Y$ and the claim is proved.
\end{proof}

\subsection{Weaker Assumptions}
\label{subsec:weaker}

There exist examples in which the cones are not compact; i.e., Assumption~\ref{assump:compactcones} does not hold (e.g., c\`adl\`ag curves in Subsection~\ref{subsec:cadlag2}). Nevertheless, $\mathcal D$ is still a metric space and completeness and separability hold under the assumptions of Theorems~\ref{thm:metric-noncompact} and \ref{thm:functor-polish-noncompact}. However, the forward implication in the pre-compactness theorem (Theorem~\ref{thm:functor-precompact-noncompact}) may fail without Assumption~\ref{assump:compactcones}. 

Also, Lemma~\ref{lem:leftrightlimit} may fail\unwritten{ Note: Without compactness, there exists at most one subsequential limit from right (or left).}. To ensure that the integral~\eqref{eq:integral} is well defined, one may add the assumption that the curve $r\mapsto \restrict{a}{\cball{r}{o}}$ (defined in Lemma~\ref{lem:leftrightlimit}) is continuous at almost every point (this property holds for the example of c\`adl\`ag curves). However, it is still not clear whether this metric generates the same topology as the metric~\eqref{eq:ghf-noncompact} (the proof of Theorem~\ref{thm:convergence} shows that the topology of the integral metric~\eqref{eq:integral} is finer than that of~\eqref{eq:ghf-noncompact}). The reason is that the nested closed sets in Lemma~\ref{lem:continuityRadius} are not necessarily compact, and hence, are not known to converge under the Hausdorff metric. If the latter is added as an assumption, then the proof of Theorem~\ref{thm:convergence} is valid and the two metrics generate the same topology.


\subsection{Fixed Underlying Space}
\label{subsec:fixed}
Given a \bcm{} $S$, Definition~\ref{def:phi} defines a set $\varphi(S)$ of additional structures on $S$. In various cases, one might be interested in considering a random element in $\varphi(S)$ while $S$ is fixed. Some examples are the notions of \textit{random closed sets}, \textit{random measures} and \textit{point processes} in stochastic geometry (see also Subsection~\ref{subsec:stochasticGeometry0}). 
The following topology on $\varphi(S)$ can be used for this purpose.

\begin{definition}
	\label{def:topology}
	Let\unwritten{\mar{Topology doesn't depend on root, but metric does.}} $S$ be a \bcm{} and $a,a_1,a_2,\ldots\in\varphi(S)$. Define $a_n\to a$ if for every $o\in S$ and $0<\epsilon\leq r$, there exists a sequence $(b_n)_n$ in $\varphi(S)$ such that $b_n\to \restrict{a}{\cball{r}{o}}$ and $\restrict{a_n}{\cball{r-\epsilon}{o}}\leq b_n\leq \restrict{a_n}{\cball{r+\epsilon}{o}}$ (the last convergence is regarded in $\tau(\cball{r+\epsilon}{o})$ by an abuse of notation). 
	Note that $\restrict{a_n}{\cball{r}{o}}$ does not need to converge to $\restrict{a}{\cball{r}{o}}$. 
\end{definition}


By fixing an arbitrary root for $S$, the following defines a metrization of this topology (one could also define a metric by the integral formula~\eqref{eq:integral}). Similarly to the previous subsections, it can be seen that this is indeed a metric and $\varphi(S)$ is complete and separable under the assumptions of Theorem~\ref{thm:functor-polish-noncompact}.
This allows one to define a random element in $\varphi(S)$. 
{One could also regard the latter as a random element of $\mathcal D$
	by considering the map $a\mapsto (S,o;a)$ from $\varphi(S)$ to $\mathcal D$, where $o\in S$ is arbitrary (it can be seen that this map is continuous). This is at the cost of considering the elements of $\varphi(S)$ up to equivalence under automorphisms of $(S,o)$.} 


\begin{definition}
	\label{def:metricOnPhi}
	Let $o\in S$ be arbitrary\unwritten{ (it seems that, without choosing a root, no functorial metrization of the above topology exists in the general case)}. For $b,b'\in \varphi(S)$ and $0<\epsilon\leq 1$, define
	\[
	a_{\epsilon}(b,b'):=\inf\{d(\restrict{b}{\cball{1/\epsilon}{o}}, b''): \restrict{b'}{\cball{1/\epsilon-\epsilon}{o}}\leq b''\leq \restrict{b'}{\cball{1/\epsilon+\epsilon}{o}} \}.
	\]
	Then, define $d(b,b'):=\inf\{ \epsilon\in (0,1]: a_{\epsilon}(b, b') \vee a_{\epsilon}(b', b) < \frac{\epsilon}{2}\}$ similarly to~\eqref{eq:ghf-noncompact}.
\end{definition} 

By Assumptions~\ref{assump:subsetLemma1} and~\ref{assump:subsetLemma2}, it can be seen that the topology of this metric coincides with Definition~\ref{def:topology}.

\begin{example}[\ali{Fell Topology and Vague Convergence}]
	\label{ex:Fell-metrization}
	For the functor of closed subsets (resp. boundedly-finite measures), it can be seen that Definition~\ref{def:topology} coincides with the Fell topology (resp. vague convergence) and Definition~\ref{def:metricOnPhi} gives a metrization of this topology. \new{Recall that $C_n\to C\subseteq S$ in the Fell topology when for every open set $U\subseteq S$ intersecting $C$ and for every compact set $K\subseteq S$ avoiding $C$, one has $U\cap C_n\neq \emptyset$ and $K\cap C_n=\emptyset$ for large enough $n$. The proof of the equivalence is left to the reader (use a covering of $\cball{r+\epsilon}{o}$ by finitely many open balls of radius at most $\delta$ for an arbitrary $\delta>0$).}
\end{example}

The topology on $\varphi(S)$ allows one to characterize convergence in $\mathcal D$ in terms of embedding into a common space as follows. This lemma does not require Assumption~\ref{assump:compactcones} of compactness of the cones.

\begin{lemma}
	\label{lem:convergence2}
	A sequence $\mathcal X_n=(X_n,o_n; a_n)$ in $\mathcal D$ converges to $\mathcal X=(X,o; a)$ if and only if there exists a \bcm{} $S$ and isometric embeddings $f_n:X_n\to S$ and $f:X\to S$ such that $f_n(o_n)\to f(o)$, $f_n(X_n)\to f(X)$ in the Fell topology and $\varphi_{f_n}(a_n)\to \varphi_f(a)$ in $\varphi(S)$.
\end{lemma}

\begin{proof}
	Assume $d(\mathcal X_n,\mathcal X)\to 0$. For each $n$, one can embed $\mathcal X_n$ and $\mathcal X$ in a common \bcm{} $S_n$ such that the images of $o_n, X_n$ and $a_n$ are close to those of $o, X$ and $a$ respectively (for this, it is enough to embed two large compact portions of them and then attach the remaining parts). Then, glue $S_1,S_2,\ldots$ by identifying the copies of $X$ in them. It can be seen that the resulting metric space is boundedly-compact and satisfies the claim.
\end{proof}

\begin{example}[Non-Compact/Infinite Marks]
	\label{ex:marks2}
	Fix $k\in\mathbb N$. For compact $X$, 
	let $\tau^{(s)}(X)$ be the set of $k$-fold marked closed subsets of $X$ (not necessarily compact) and $\tau^{(m)}(X)$ be the set of boundedly-finite $k$-fold marked measures on $X$ defined in Definition~\ref{def:marking}. 
	Assuming that the mark space $\Xi$ is boundedly-compact, then
	one can equip $\tau^{(s)}(X)$ and $\tau^{(m)}(X)$ with the Fell topology and the vague topology respectively.  
	Note that one needs a metrization of these topologies in a functorial way (Definition~\ref{def:functor}). In particular, the metrics in Definition~\ref{def:metricOnPhi} or the integral metric~\eqref{eq:integral} do not work since they need a root. To correct this, consider truncating marked measures/subsets to sets of the form $X^k\times \cball{r}{\xi_0}$, where $\xi_0\in\Xi$ is fixed arbitrarily (instead of truncating to balls in $X^k \times \Xi$). Using this truncation, modify Definition~\ref{def:metricOnPhi} accordingly. It can be seen that these are Strassen-type metrics and make $\tau^{(s)}(X)$ and $\tau^{(m)}(X)$ complete and separable. In addition, $\tau^{(s)}$ and $\tau^{(m)}$ are functors which are pointwise-continuous and Hausdorff-continuous. Therefore, the corresponding spaces $\mathcal C_{\tau^{(s)}}$ and $\mathcal C_{\tau^{(m)}}$ are also complete separable metric spaces. This example will be extended to \bcm s in Example~\ref{ex:measures3} (one can also modify the integral metric~\eqref{eq:integral} similarly, but it is no longer Strassen-type and Assumption~\ref{assump:subsetLemma1} does not hold\unwritten{A counterexample for curves is in OneNote.}). 
	%
\end{example}

\subsection{Examples}
\label{subsec:examples-noncompact}

Here, some examples are provided for illustrating how the framework of this section can be used. \new{The verification of the assumptions is mostly left to the reader in order to keep focused on the main thread}. Further examples will be presented in Section~\ref{sec:special}. 

\begin{example}[No Additional Structure]
	When there is no additional structure (see Example~\ref{ex:constant}), Lemma~\ref{lem:convergence} implies that \eqref{eq:ghf-noncompact} is a metrization of Gromov's notion of convergence of boundedly-compact pointed metric spaces (Section~3.B of~\cite{bookGr99}). This metrization has been already provided in~\cite{Kh19ghp}.
\end{example}




\subsubsection{Marked Measures and Marked Closed Subsets}
\label{subsec:markedmeasures}

The following example is the conclusion of the running examples~\ref{ex:basic-noncompact}, \ref{ex:basicnoncompact-trunc}, \ref{ex:basicnoncompact-phi} and \ref{ex:subsetLemma}. It is generalized to marked subsets/measures in the next examples.

\begin{example}[Measures and Closed Subsets]
	\label{ex:basicnoncompact-conclusion}
	For the space $\mathcal D$ of pointed \bcm s equipped with a closed subset (resp. a boundedly-finite measure), already discussed in this section, all of the the assumptions are satisfied. In particular, $\mathcal D$ is a complete separable metric space.
\end{example}

\begin{example}[\ali{Finite Marked Measures}]
	\label{ex:measures4}
	Let $\tau(X)$ be the set of $k$-fold marked finite measures on $X$ (Definition~\ref{def:marking}); i.e., finite measures on $X^k\times \Xi$. Equip $\tau(X)$	 with the Prokhorov metric and the natural partial order on measures. For $\mu\in \tau(X)$ and $f:Y\to X$, let the truncation $\tau^t_f(\mu)$ be the inverse image (under $\tau_f$) of $\restrict{\mu}{f(Y)^k\times \Xi}$.
	Then, $\mathcal D$ is the set of pointed \bcm s $(X,o)$ equipped with a $k$-fold marked measure $\nu$ on $X$ such that the \textit{ground measure} of $\nu$ (i.e., the projection of $\nu$ on $X^k$) is boundedly-finite.
	The results of this section imply that~\eqref{eq:ghf-noncompact} is a metric and $\mathcal D$ is complete and separable.
	%
	%
	%
\end{example}

\begin{example}[Marked Compact Subsets]
	\label{ex:markedcompact}
	Let $\tau(X)$ be the set of $k$-fold marked compact subsets $K$ of $X$, including the empty set. Equip $\tau(X)$ with the Hausdorff extended metric and the inclusion partial order. Define the truncation by $\tau^t_f(K):=(\tau_f)^{-1}(K\cap f(X)^k\times \Xi)$ for $f:X\to Y$ and $K\subseteq\tau(Y)$. 
	Then, $\mathcal D$ is the set of pointed \bcm s $(X,o)$ equipped with a closed subset $C$ of $X^k\times \Xi$ such that for every compact subset $K\subseteq X^k$, $C\cap (K^k\times \Xi)$ is compact. 
	Again, the results of this section imply that $\mathcal D$ is complete and separable.
\end{example}

\del{
	\begin{example}
		\label{ex:measures5}
		\unwritten{Another method: Let $\varphi(X):=\{(\mu,a): a\geq \mu(X) \}$ and similarly for compact sets. This gives a different topology since it prevents escaping to infinity.}
		Let $\mathcal C''$ be the set of boundedly-compact pointed metric spaces equipped with a finite measure (resp. a compact subset). 
		This example can not be obtained by the inverse limit in Definition~\ref{def:phi}. However, it can be seen that $\mathcal C''$ is a Borel subset (in fact, a $F_{\sigma}$ subset) of the Polish space $\mathcal C'$ of Example~\ref{ex:measures4} (resp. Example~\ref{ex:markedcompact}). This is enough for probability-theoretic applications\del{ (see also Proposition~\ref{prop:subset-noncompact})}.
	\end{example}
}

\begin{example}
	\label{ex:measures3}
	Let $\mathcal D$ be the space of pointed \bcm s equipped with a $k$-fold marked closed subset (resp. a boundedly-finite $k$-fold marked measure) with no restrictions. Similarly to the previous examples, if the mark space is boundedly-compact, then $\mathcal D$ can be obtained by considering the functor $\tau^{(s)}$ (resp. $\tau^{(m)}$) of non-compact marked closed subsets (resp. non-finite marked measures) defined in Example~\ref{ex:marks2}. So the results of this section imply that $\mathcal D$ is complete and separable.
	%
\end{example}

\ali{
	\begin{example}[Vague vs Weak convergence]
		Let $\mathcal D'$ be the space of pointed \bcm s equipped with a finite measure (resp. a compact subset). 
		This is a subset of the set in Example~\ref{ex:basicnoncompact-conclusion}, 
		but here we consider the following finer topology which prevents \textit{escaping to infinity}: $(X_n,o_n;a_n)\to (X,o; a)$ if they can be embedded in a common \bcm{} such that, after the embedding, $a_n$ converges to $a$ weakly (resp. under the Hausdorff metric), $X_n$ converges to $X$ in the Fell topology and $o_n$ converges to $o$ (the topology of Example~\ref{ex:basicnoncompact-conclusion} uses vague convergence instead of weak convergence, see~\cite{AtLoWi16}). To use the framework of this section, let $\tau(X)$ be the set of pairs $(\mu,s)$, where $\mu$ is a finite measure on $X$ and $\mu(X)\leq s\in\mathbb R$ (resp. $(K,s)$ where $K$ is a compact subset and $\mathrm{diam}(K)\leq s$). Define the truncation by truncating only the measure/subset while keeping $s$ unchanged. The reader can verify that all of the assumptions of this section are satisfied and the corresponding space $\mathcal D$ is a Polish space which contains $\mathcal D'$ as a closed subspace.
	\end{example}
}

\subsubsection{Additional Point and Discrete Subset}

\unwritten{It is also $G_{\delta}$. To show this, consider the set of elements of $\mathcal D$ such that the intersection with $\oball{r}{o}$ is not a single point!}
Let $\mathcal D'$ be the set of \textit{doubly-pointed} \bcm s; i.e., equipped with one additional point other than the origin. 
One can regard the additional point as a closed subset (or as a Dirac measure) and show that $\mathcal D'$ is a $G_{\delta}$ subset of the set $\mathcal D$ of Example~\ref{ex:basicnoncompact-conclusion} (in fact, the difference of two closed subsets), and hence is Polish.
The following example gives a \ali{more} direct method as a basic illustration of the method of this section. It can be seen that it produces the same metric on $\mathcal D$.
\begin{example}[Additional Point]
	\label{ex:double}
	Let $\varphi(X):=X\cup\{\Delta\}$, where $\Delta$ is an arbitrary element not contained in $X$ called \textit{the grave} (that might depend on $X$). For $a\in \varphi(X)$ and an isometric embedding $f:Y\to X$, define the truncation of $a$ by $\varphi^t_f(a):=f^{-1}(a)$ if $a\in f(Y)$ and $\varphi^t_f(a):=\Delta$ if $a\not\in f(Y)$. The partial order on $\varphi(X)$ is $\forall a:\Delta\leq a$. If $X$ is compact, then consider the extended metric on $\tau(X):=\varphi(X)$ defined by $d(\Delta,a):=\infty$ for all $a\in X$. The reader can verify that this fits into the framework of this section and the assumptions are satisfied. So the corresponding set $\mathcal D$ of Definition~\ref{def:C'} is complete and separable. Here, $\mathcal D'$ is an open subset of $\mathcal D$ (and hence, $\mathcal D'$ is Polish by itself). \unwritten{$\mathcal D'$ is not closed. The additional may escape to the infinity; i.e., escape to the grave.}
\end{example}

{One can also regard $\mathcal D'$ as a subset of the space in the following example. This generates the same topology on $\mathcal D'$.}

\begin{example}[Discrete Subsets]
	\label{ex:finite}
	Let $\mathcal D''$ the space of pointed \bcm s equipped with a discrete subset. This is a subspace of Example~\ref{ex:basicnoncompact-conclusion}, but can also be obtained directly by the functor of finite subsets (Example~\ref{ex:basic}) and all of the assumptions are satisfied except completeness. If one uses the Prokhorov metric (resp. the Hausdorff metric) between finite subsets, the completion of $\mathcal D''$ is the space of pointed \bcm s equipped with a discrete multi-set (resp. a closed subset). 
	%
	%
\end{example}

\subsubsection{Product and Composition of Functors}


\begin{example}[Product]
	\label{ex:multiple-noncompact}
	Let $\mathcal D$ be the space of pointed \bcm s $(X,o)$ equipped with a tuple $(a_1,\ldots,a_n)$, where each $a_i$ belongs to a set $\varphi_i(X)$. Assume each $\varphi_i(X)$ is obtained by a functor $\tau^{(i)}$ and a truncation functor $\tau^{(i),t}$ that satisfy the assumptions of this section.
	Define $\tau(X):=\prod_i \tau^{(i)}(X)$ with the max product metric (if $n=\infty$, use the metric~\eqref{eq:maxproduct}). 
	Define the truncation functor $\tau^t$ and the partial order on $\tau(X)$ element by element. 
	It can be seen that $\tau$ and $\tau^t$ satisfy all of the assumptions as well, and hence, $\mathcal D$ is complete and separable with the resulting metric.
\end{example}

\begin{example}[Composition]
	\label{ex:compose}
	Assume $\tau$ and $\tau^t$ satisfy the assumptions of this section. Let $\rho:\mathfrak{Comp}\to\mathfrak{Met}$ be a functor such that $\rho(X)$ is compact for every $X$. Then, one can define the functor $\tau\circ\rho:\mathfrak{Comp}\rightarrow\mathfrak{Pos_m}$ similarly to Subsection~\ref{subsec:compos}: Let $\tau\circ\rho(X):=\tau(\rho(X))\in\mathfrak{Pos_m}$ and $(\tau\circ\rho)_f:=\tau_{(\rho_f)}$. Define the truncation functor $(\tau\circ\rho)^t$ as follows: If $f:X\to Y$ is an isometric embedding, \ali{which implies that $\rho_f:\rho(X)\to\rho(Y)$ is also an isometric embedding,} for $b\in \tau(\rho(Y))$ define $(\tau\circ\rho)^t_f(b):=\tau^t_{\rho_f}(b)\in\tau(\rho(X))$.
	\ali{According to Lemma~\ref{lem:composition},} if both $\tau$ and $\rho$ are pointwise-continuous (resp. Hausdorff-continuous, resp. 1-Lipschitz), then so is $\tau\circ\rho$. In addition, it can be seen that $\tau\circ\rho$ and $(\tau\circ\rho)^t$ satisfy all of the assumptions of this section except maybe Assumptions~\ref{assump:subsetLemma1} and~\ref{assump:subsetLemma2}. These assumptions should be manually checked in explicit examples. {See e.g., Subsections~\ref{subsec:particle} and~\ref{subsec:closedProcess}.} 
\end{example}

\subsubsection{Continuous Curves}
\label{subsec:curvesnoncompact}
\ali{Let $I$ be a closed interval in $\mathbb R$ containing zero. Here, the set $\mathcal D'$ of pointed \bcm s $X$ equipped with a continuous curve $\eta:I\to X$ is studied. By splitting $\eta$ into two curves $\restrict{\eta}{I\cap [0,\infty)}$ and $\restrict{\eta}{I\cap (-\infty,0]}$ and using Example~\ref{ex:multiple-noncompact}, it is enough to study the cases $I=[0,T]$ and $I=[0,\infty)$.
	
	To define a metric on $\mathcal D'$, we define the truncation of curves as follows. This is in fact a modification of the method of~\cite{GwMi17} (see Subsection~\ref{subsec:ghpu}). A simpler method will also be discussed after the example, which can be generalized to the case where $I$ is an arbitrary compact (or boundedly-compact) metric space.	
}

\begin{example}
	\label{ex:curves2}
	Let $I:=[0,T]$. 
	For compact $X$, let $\tau(X)$ be the set of continuous curves $\eta:I\to X\cup\{\Delta\}$ equipped with the sup (extended) metric, where $\Delta$ is a grave as in Example~\ref{ex:double}. Define $\eta'\leq \eta$ when either $\eta'(\cdot)=\Delta$ or $\eta'$ is obtained by \textit{stopping} $\eta$ at some time $t_0\in I$; i.e., $\eta'(t)=\eta(t\wedge t_0)$.  For all $f:X\to Y$ and $\eta\in \tau(X)$, let $\tau_f(\eta):=f\circ \eta$. Also, for $\eta'\in \tau(Y)$, let $\tau^t_f(\eta'):=f^{-1}\circ \eta'$ stopped at the first exit time of $\eta'$ from $f(X)$ (if $\eta'(0)\not\in f(X)$, let $\tau^t_f(\eta')(\cdot):=\Delta$). It can be seen that these definitions satisfy the assumptions of Subsections~\ref{subsec:C'} and~\ref{subsec:metricOnC'} except Hausdorff-continuity (see Subsection~\ref{subsec:curves}). 
	So the corresponding set $\mathcal D$ is a complete metric space. \ali{Since compact spaces with an additional structure are dense in $\mathcal D$,} separability of $\mathcal D$ is implied by separability in the compact case (Subsection~\ref{subsec:curves}). 
	By Definition~\ref{def:C'}, it can be seen that $\mathcal D$ is the set of pointed \bcm s $(X,o)$ with a continuous curve $\eta$ in $X\cup\{\Delta\}$ such that either (1) $\eta$ is defined on the entire of $[0,T]$ or (2) $\eta$ is defined on some interval $[0,T')\subseteq [0,T]$ and exits any compact subset of $X$ \ali{before time $T'$}. 
	One can see that $\mathcal D'$ is an open subset of the Polish space $\mathcal D$, and hence, is $\mathcal D'$ is Polish. 

\end{example}

\del{\begin{remark}
		\label{rem:curvesSplitting}
		In the setting of the above example, similar results hold for the case $I=\mathbb R$. These results are obtained by noting that every continuous curve $\eta:\mathbb R\to X$ is the joining of the two curves $\restrict{\eta}{(-\infty,0]}$ and $\restrict{\eta}{[0,\infty)}$ and by using Examples~\ref{ex:multiple} and~\ref{lem:multiple2}.
	\end{remark}
}

\begin{example}
	\label{ex:curves3}
	Let $I:=[0,\infty)$ and define $\mathcal D'$, the set $\tau(X)$, the partial order and the truncations similarly to the previous example.
	The metric on $\tau(X)$ should be carefully chosen\unwritten{A counterexample for curves is in OneNote.} to ensure that Assumption~\ref{assump:subsetLemma2} holds. A suitable metric on $\tau(X)$ is the following, which is a Strassen-type metric:
	\begin{equation}
		\label{eq:curves-metric}
		d(\eta, \eta'):= \inf\{\epsilon\in(0,1]: \dsup\left(\restrict{\eta}{[0,1/\epsilon]}, \restrict{\eta'}{[0,1/\epsilon]} \right)\leq \epsilon  \}.
	\end{equation}
	It is left to the reader that $\tau$ satisfies the assumptions of this section except Hausdorff-continuity. Similarly to the previous example, it can be seen that the corresponding set $\mathcal D$ is a complete separable metric space and $\mathcal D'$ is a $G_{\delta}$ subset of $\mathcal D$ (the cases where the curve does not explode until time $N$ is open in $\mathcal D$), and hence $\mathcal D'$ is Polish.
	\unwritten{Also, the extension $\varphi(X)$ for non-compact $X$ is the set of continuous curves in $X\cup\{\Delta\}$ that are either defined on the whole $[0,\infty)$ or defined on an interval $[0,T')$ and exit any compact set before time $T'$.}

\end{example}

\unwritten{Example: for all $t\geq 0$, let $\eta(t):=t$ and $\eta_i(t):=t\wedge (i-t)$. Then, $(\mathbb R,0;\eta_i)$ converges to $(\mathbb R,0;\eta)$ in both notions.}

Another method to define a metric on $\mathcal D'$ is by regarding curves as marked closed subsets similarly to Subsection~\ref{subsec:curves}.
Similarly to the proofs of Propositions~\ref{prop:curves} and~\ref{prop:simplemarks}, one can prove the following.

\begin{proposition}
	\label{prop:curves2}
	Let $I=[0,T]$ (resp. $I=[0,\infty)$). By regarding curves as marked closed subsets, $\mathcal D'$ is an open (resp. $G_{\delta}$) subset of the Polish space in Example~\ref{ex:markedcompact} (resp. Example~\ref{ex:measures3}), and hence, $\mathcal D'$ is Polish. 
	In addition, the topology on $\mathcal D'$ coincides with the induced topology.
\end{proposition}

%

\subsubsection{C\`adl\`ag Curves}
\label{subsec:cadlag2}
Let $\mathcal D'$ be the set of tuples $\mathcal X=(X,\eta)$, where $X$ is a \bcm{} and $\eta:I\to X$ is a c\`adl\`ag curve, where $I:=[0,T]$ is a compact interval\del{ (the cases $I=\mathbb R^{\geq 0}$ and $I=\mathbb R$ are treated at the end of the example)}. We show that 
$\mathcal D'$ is not complete but it is still a Polish space. 
\ali{This is an example where cones are not compact, but the weaker assumptions in Subsection~\ref{subsec:weaker} are satisfied.}

\begin{example}
	\label{ex:cadlag}
	For\unwritten{For the subset lemma, a variant of the Skorokhod metric is needed for noncompact intervals, discussed below.} compact $X$,
	let $\tau(X)$ be the set of c\`adl\`ag curves $\eta:I\to X\cup\{\Delta\}$ (where $\Delta$ is a grave as in Example~\ref{ex:double}), such that $\eta^{-1}(\Delta)$ is either the empty set or an interval of the form $[T_{\eta},T]$. For $\eta,\eta'\in \tau(X)$, define $\eta\leq \eta'$ if $\forall t<T_{\eta}:\eta(t)=\eta'(t)$; i.e., $\eta$ is obtained by \textit{killing} $\eta'$ at time $T_{\eta}$. For all isometric embeddings $f:X\to Y$ and $\eta\in\tau(X)$, let $\tau_f(\eta):=f\circ \eta \in \tau(Y)$. Also, for $\eta'\in \tau(Y)$, let the truncation $\tau^t_f(\eta')$ be $f^{-1}\circ\eta'$ killed at the first exit time of $\eta'$ from $f(X)$ (note that if $\eta'(0)\not\in f(X)$, then $\tau^t_f(\eta')(\cdot)=\Delta$). 
	The set $\tau(X)$ is a closed subset of the set of c\`adl\`ag curves in $X\cup\{\Delta\}$ endowed with the Skorokhod metric (which is an extended metric here). So $\tau(X)$ is a complete  separable extended metric space.
	It is left to the reader to show that $\tau$ and $\tau^t$ satisfy all of the assumptions of this section except Assumption~\ref{assump:compactcones}; e.g., if $\eta$ is not continuous on $[0,T_{\eta})$, then the lower cone of $\eta$ is not compact (\new{for verifying Assumptions~\ref{assump:subsetLemma1} and~\ref{assump:subsetLemma2}, consider the construction in Lemma~\ref{lem:basic-hausdorff} for compact subsets and kill the curves suitably}). However, since $\tau$ is both pointwise-continuous and Hausdorff-continuous (see Subsection~\ref{subsec:cadlag}), the results of this section imply that the corresponding set $\mathcal D$ is a complete separable metric space.
	%
	The set $\mathcal D$ is larger than $\mathcal D'$ as described below. 
	When $X$ is not compact, an additional structure on $X$ (Definition~\ref{def:phi}) is a function $\eta:I\to X\cup\{\Delta\}$ such that either (1) $\eta$ is c\`adl\`ag and is killed at the first hitting to $\Delta$ or (2) there exists a time $T_{\eta}\in [0,T]$ such that $\restrict{\eta}{[0,T_{\eta})}$ is a c\`adl\`ag curve in $X$ that exits any compact subset of $X$ before time $T_{\eta}$ (it might have no left limit at $t=T_{\eta}$) and $\restrict{\eta}{[T_{\eta},T]}\equiv \Delta$. 
	Similarly to Example~\ref{ex:curves2}, $\mathcal D'$ is an open subset of $\mathcal D$, and hence, is Polish.
	%
\end{example}

\begin{remark}
	\label{rem:cadlag}
	The above example does not satisfy Assumption~\ref{assump:compactcones} and some of the results based on this assumption. Lemma~\ref{lem:continuityRadius} does not hold here. For instance, let $X:=\mathbb R$, $o:=0$, $\eta(t):=t$ for $0\leq t<1$ and $\eta(t)=t+1$ for $1\leq t\leq T$. If $h(r)$ is the truncation of $\eta$ to $\cball{r}{o}$, then $h$ has no right limit at $r=1$. However, in the general case, it can be seen that $h$ is continuous except at countably many points ($r$ is a continuity point if $\partial\cball{r}{o}$ does not contain any jump points or local maximum of distance from $o$). It can be seen that the weaker assumptions in Subsection~\ref{subsec:weaker} hold, and hence, the convergence result (Theorem~\ref{thm:convergence}) is valid.
	\\
	In addition, the pre-compactness result (Theorem~\ref{thm:functor-precompact-noncompact}) fails in this example. For instance, for $n\geq 1$, let $\eta_n(t):=t$ when $0\leq t<1$ and $\eta_n(t):=t+1-1/n$ when $1\leq t\leq T$. Then, $\mathcal X_n:=(\mathbb R,0,\eta_n)$ is convergent, but $\pcball{\mathcal X}{2}_n$ is not pre-compact.
\end{remark}

\begin{example}
	For\unwritten{\mar{Later: It is Strassen-type}} the case $I:=[0,\infty)$, similar arguments can be applied, but the metric on $\tau(X)$ should be carefully chosen\unwritten{I guess that the topology of~\eqref{eq:variantMetric} is the same anyway.} to ensure that Assumption~\ref{assump:subsetLemma2} holds. A suitable metric on $\tau(X)$ is the following, which has the same idea as~\eqref{eq:ghf-noncompact} \ali{and generates the same topology as the Skorokhod metric (Section~16 of~\cite{bookBi99})}: If $\mathrm{kill}_{t_0}(\eta)$ denotes $\eta$ killed at time $t_0$, let 
	\[
	a_{\epsilon}(\eta,\eta'):=\inf \{ d_S\left(\mathrm{kill}_{1/\epsilon}(\eta), \mathrm{kill}_{t}(\eta')\right): t\geq 0, \norm{t-1/\epsilon}\leq\epsilon\},
	\]
	where $d_S$ denotes the Skorokhod metric defined by the same equation~(12.16) of~\cite{bookBi99}\unwritten{ (\ali{although the formula is given for compact intervals})}. Then, define the distance between $\eta$ and $\eta'$ by a formula similar to~\eqref{eq:ghf-noncompact}. It can be seen that this is a metric on $\tau(X)$ and satisfies the assumptions of this section and the weaker assumptions of Subsection~\ref{subsec:weaker}, except compactness of cones. So the claims of Example~\ref{ex:cadlag} hold also in the case $I=[0,\infty)$.\unwritten{ Here, $\mathcal D'$ is a $F_{\sigma\delta}$ subset.}  
	Similar results can also be obtained for the case $I=\mathbb R$.  
\end{example}

\section{\ali{Further Examples} and Connections to Other Notions}
\label{sec:special}

In this section, some Gromov-Hausdorff-type metrics in the literature are discussed and it is shown how they can be considered as special cases of the general framework of this paper. In addition, some new examples of the framework are provided; e.g., in Subsections~\ref{subsec:closedProcess} and~\ref{subsec:ends}. \new{The goal is to present some examples of the framework and it is not intended to include the full proofs}.

\subsection{Random Objects in Stochastic Geometry}
\label{subsec:stochasticGeometry0}

According to Subsection~\ref{subsec:fixed}, the notions of random measures and random closed sets on a fixed space are special cases of the framework of Section~\ref{sec:noncompact}, when the underlying space $S$ is boundedly-compact (in stochastic geometry, one can also assume that $S$ is a locally-compact second-countable Hausdorff topological space). By Example~\ref{ex:finite}, the same holds for \textit{point processes}, which are random discrete (multi-) subsets of $S$. The framework covers some other notions in stochastic geometry as follows.  For instance, a \textit{marked random measure on $S$} is defined in the literature as a random measure on $S\times \Xi$. Sometimes, it is assumed that the \textit{ground measure}; i.e., the projection of the measure onto $S$, is a boundedly-finite. So, this notion is the case $k=1$ of $k$-fold marked measures in Subsection~\ref{subsec:markedmeasures}. Also, the notion of \textit{marked point processes} is a special case of marked random measures defined similarly (one can use the framework directly as well, similarly to Example~\ref{ex:finite}). The notion of \textit{particle processes} will be discussed in Subsection~\ref{subsec:particle}. Additionally, Examples~\ref{ex:markedcompact} and~\ref{ex:measures3} allow one to define \textit{marked random closed subsets of $S$}. 

{As mentioned in Subsection~\ref{subsec:fixed}, all of these notions in stochastic geometry can be extended to the case where the base space $S$ is a random pointed \bcm, at the cost of considering them up to automorphisms of the base space.}

\subsection{Graphs, Networks and Discrete Spaces}
\label{subsec:networks}

The edges of a graph can be represented as a marking of the pairs of vertices. Therefore, with the graph-distance metric, graphs can be regarded as a metric space equipped with a 2-fold marked closed subset. It can be seen that the corresponding topology (Example~\ref{ex:markedcompact}) extends the Benjamini-Schramm convergence of rooted graphs~\cite{BeSc01}. The same holds for local weak convergence of networks~\cite{processes}, where a network is a graph with an additional marking of vertices and edges. For local weak convergence of doubly-rooted graphs and networks, one can use product of functors (see Examples~\ref{ex:markedcompact}, \ref{ex:double} and~\ref{ex:compose}) and the same claim holds.

In~\cite{I}, a complete metric is defined on the space of marked discrete metric spaces. When there are no marks, by equipping every discrete space with the corresponding counting measure, this is a special case of the metrization of the GHP metric (\cite{AbDeHo13}, \cite{Kh19ghp} and Example~\ref{ex:basicnoncompact-phi} above). The marked case is also a special case of the case where the additional structure is a measure and a marked closed subset, which can be obtained by a product of functors.



\subsection{The Gromov-Hausdorff-Prokhorov-Uniform Metric}
\label{subsec:ghpu}
The \defstyle{GHPu metric} is defined in~\cite{GwMi17} on the space $\mathcal M^u$ of compact metric spaces $X$ equipped with a finite measure $\mu$ on $X$ and  a continuous curve $\eta:\mathbb R\to X$ that is convergent as $t\to\infty$ and $t\to-\infty$. It is also proved that $\mathcal M^u$ is complete and separable. This can be expressed in the framework of this paper. Indeed, this is just the product of the functors of measures (Example~\ref{ex:basic-compact-result}) and curves (a slight generalization of Subsection~\ref{subsec:curves}). 

In addition, using a specific formula to truncate curves, \cite{GwMi17} studies the space $\mathcal M^u_{\infty}$ of complete locally-compact length spaces $X$ together with a locally finite measure $\mu$ on $X$ and a continuous curve $\eta:\mathbb R\to X$ pointed at the distinguished point $\eta(0)$. However, despite the claim of~\cite{GwMi17}, $\mathcal M^u_{\infty}$ is not complete. This issue will be resolved below.

To express $\mathcal M^u_{\infty}$ in the framework of this paper, one can use the product of the functors of measures (Example~\ref{ex:basicnoncompact-conclusion}) and curves (Subsection~\ref{subsec:curvesnoncompact}). The truncation in Subsection~\ref{subsec:curvesnoncompact} is slightly different from that of~\cite{GwMi17} (the latter depends on the radius of the ball), but the resulting topology on $\mathcal M^u_{\infty}$ coincides with that of~\cite{GwMi17}. 	
Note that Subsection~\ref{subsec:curvesnoncompact} adds the curves in $X$ that are defined on an open interval or half-line containing $0$ and \textit{blow up} at the end points (this is necessary for completeness). Similarly to Proposition~\ref{prop:curves2}, $\mathcal M^u_{\infty}$ is a $G_{\delta}$ subset of the corresponding space $\mathcal D$, which is a Polish space. Hence, $\mathcal M^u_{\infty}$ is Polish (it is complete with another metric).

The above discussion generalizes the approach of~\cite{GwMi17} in several aspects. First, the spaces are not limited to length spaces and the curves do not necessarily start at the root. Also, curves can be generalized to continuous functions from a given space to $X$. In addition, a simpler method to define a metric is discussed in Subsections~\ref{subsec:curves} and~\ref{subsec:curvesnoncompact} by regarding (the graphs of) curves as marked closed subsets.

\subsection{Spatial Trees}
\label{subsec:spatialTree}

%

Let $\Xi$ be a complete separable metric space and $\mathcal T$ be the set of pairs $(X,\varphi)$, where $X$ is a compact metric spaces and $\varphi\in C(X,\Xi)$ is a continuous function from $X$ to $\Xi$.
Consider the following distance function on $\mathcal T$:
\begin{equation}
	\label{eq:spatialTree:0}
	d((X,\varphi),(Y,\psi)):=\inf\bigg\{\frac 12 \mathrm{dis}(R) \vee \sup_{(x,y)\in R}\{d(\varphi(x),\psi(y))\} \bigg\},
\end{equation}
where the infimum is over all correspondences $R$ of $X$ and $Y$ and $\mathrm{dis}(R)$ is the {distortion} of $R$ (see Subsection~\ref{subsec:strassen}).
This distance function is defined in~\cite{DuLe05} for the case of \textit{spatial trees}; i.e., when $X$ and $Y$ are  \textit{real trees} (except that the $\vee$ in the formula is a $+$ in~\cite{DuLe05} and the coefficients are slightly different, which are unimportant changes). It is claimed in~\cite{DuLe05} that `it is easy to verify that $\mathcal T$ is a Polish space'. However, as observed in~\cite{CrHaKu12} and~\cite{BaCrKu17}, $\mathcal T$ is not complete. {The results of~\cite{BaCrKu17} imply that $\mathcal T$ is separable.} 
Here, we prove the following.

\begin{proposition}
	The space $\mathcal T$ is a $G_{\delta}$ subspace of some complete separable metric space, and hence, $\mathcal T$ is a Polish space (i.e., it is complete under another metric that generates the same topology).
\end{proposition}

\begin{proof}
	Note that $C(X,\Xi)$ is not a functor since continuous functions cannot be pushed forward under isometric embeddings naturally. 
	\ali{However, one can identify an element $\varphi\in C(X,\Xi)$ with its graph $\mathrm{gr}_\varphi\subseteq X\times \Xi$ and regard it as a \textit{simple marked closed subset} defined in Example~\ref{ex:marking-function}. Also, the metrics are identical; indeed, the Strassen-type results mentioned in Subsection~\ref{subsec:strassen} imply that~\eqref{eq:spatialTree:0} can be rewritten as
		\begin{equation}
			\label{eq:spatialTree:1}
			d((X,\varphi),(Y,\psi))=\inf\big\{\hausdorff\big(f(X),g(Y)\big)\vee \hausdorff\big(f'(\mathrm{gr}_\varphi),g'(\mathrm{gr}_\psi)\big)  \big\},
		\end{equation}
		where the infimum is over all metric spaces $Z$ and isometric embeddings $f:X\to Z$ and $g:Y\to Z$, and where $f':X\times \Xi\to Z\times \Xi$ is defined by $f'(x,\xi):=(f(x),\xi)$. Hence, the metric~\eqref{eq:spatialTree:0} is identical to the metric for marked closed subsets and $\mathcal T$ is a subspace of the space $\mathcal C_{\tau}$ corresponding to simple marked compact subsets (see Proposition~\ref{prop:simplemarks}). 
		Since $C_{\tau}$ is Polish (it is a $G_{\delta}$ subspace of some Polish space by Proposition~\ref{prop:simplemarks}) and $\mathcal T$ is closed in $\mathcal C_{\tau}$, one obtains that $\mathcal T$ is also Polish and the claim is proved.}
\end{proof}

To avoid the issue of non-completeness of $\mathcal T$, \cite{CaHa21brownianweb} uses another approach by restricting attention to the subspace $\mathcal T'$ of spatial trees $(X,\varphi)$ in which $\varphi$ is \textit{little $\alpha$-Holder}. Then, a metric is defined on $\mathcal T'$ by adding to~\eqref{eq:spatialTree:0} a term addressing the Holder property\footnote{In fact, \cite{CaHa21brownianweb} assumes that $\varphi$ is a proper function as well (this matters only when $X$ is not compact) and adds another term to~\eqref{eq:spatialTree:0} addressing this property. We omitted this because it is not needed for having Polishness and some modification is needed to represent it in the framework of this paper.} (Equation~(2.5) of~\cite{CaHa21brownianweb}). It is proved that $\mathcal T'$ is complete and separable. Similarly to the above proposition, it can be shown that this metric is also a special case of the framework of Section~\ref{sec:compact} written in a Strassen-type form, and also Polishness follows from the results of Section~\ref{sec:compact}.

\ali{The paper~\cite{BaCrKu17} studies the set $\mathcal T''$ of \textit{measured rooted spatial trees} 
	\unwritten{The metric of~\cite{BaCrKu17} is more involved and is defined using both isometric embeddings and correspondences. }
	and shows that $\mathcal T''$ is a separable (non-complete) metric space. The issue of non-completeness can be resolved as follows. First, one can rewrite the metric of~\cite{BaCrKu17} in the form of~\eqref{eq:ghfunctor}, where the additional structure is a tuple of a point, a measure and a 1-fold marked compact subset. Similarly to the above proposition, $\mathcal T''$ is a $G_{\delta}$ subspace of the resulting Polish space, and hence, $\mathcal T''$ is also Polish.
}

Measured rooted spatial trees are also extended in~\cite{BaCrKu17} to the case of locally-compact length spaces. This extension is by the same method as~\cite{AbDeHo13} with the difference that the resulting metric space is not complete. This issue can be resolved by proving Polishness similarly to the above discussion using the results of Section~\ref{sec:noncompact} (see Example~\ref{ex:markedcompact}).

\subsection{The Spectral Gromov-Hausdorff Metric}
\label{subsec:spectralGH}
Let $I\subset\mathbb R$ be a fixed compact interval.
The paper~\cite{CrHaKu12} considers the set $\widetilde{\mathcal T}$ of tuples $(X,\pi,q)$, where $X$ is a compact metric spaces, $\pi$ is a Borel probability measure on $X$ and $q\in C(X\times X\times I,\mathbb R)$. A metric on this space is defined in~\cite{CrHaKu12} using both isometric embeddings and correspondences and it is shown that a separable (non-complete) metric space is obtained. This metric is called \defstyle{the spectral Gromov-Hausdorff metric} in~\cite{CrHaKu12}. 
\ali{It is shown below that this is a special case of the framework of Section~\ref{sec:compact} and, in addition, the issue of non-completeness is resolved. }

\ali{Note that $C(X\times X\times I,\mathbb R)$ is not a functor similarly to the previous subsection. However, the elements of $C(X\times X\times I,\mathbb R)$ can be regarded as continuous functions from $X\times X$ to $\Xi$ and vice versa, where $\Xi:=C(I,\mathbb R)$ is equipped with the sup metric. The graph of such a function is a simple 2-fold marked compact subset (Example~\ref{ex:marking-function}). 
So, $\widetilde{\mathcal T}$ is identified with a subset of $\mathcal C_{\tau}$, where $\tau$ is the product of the functor of 2-fold marked compact subsets and the functor of measures. 
In addition, by the Strassen-type results of Subsection~\ref{subsec:strassen}, one can deduce that the metrics on $\widetilde{\mathcal T}$ and $\mathcal C_{\tau}$ are identical. Now, similarly to the previous subsection, it can be seen that $\widetilde{\mathcal T}$ is a $G_{\delta}$ subspace of the Polish space $\mathcal C_{\tau}$, and hence, $\widetilde{\mathcal T}$ is Polish (it is complete with another metric). This resolves the issue of non-completeness of the metric in~\cite{CrHaKu12}.}

\subsection{Particle Processes}
\label{subsec:particle}

Roughly speaking, a \textit{particle process} in a \bcm{} $S$ is a random \textit{discrete} collection of compact subsets of $S$. More precisely, let $\mathcal K(S)$ be the set of nonempty compact subsets of $S$ equipped with the Hausdorff metric. 
Let $\varphi_1(S)$ be the set of discrete subsets $A$ of $\mathcal K(S)$; equivalently, every compact subset of $S$ should contain only finitely many elements of $A$. It is usually required that every compact subset of $S$ intersects only finitely many elements of $A$\del{; equivalently, $A$ is also discrete in $\mathcal F(S)\setminus\{\emptyset\}$}. Let $\varphi_0(S)$ be the set of such subsets of $\mathcal K(S)$. Then, a \defstyle{particle process} in $S$ is a a random element of $\varphi_0(S)$. See e.g., \cite{bookScWe08} for more details\del{\footnote{The usual definition of a particle process is a point process in $\mathcal F(S)\setminus\{\emptyset\}$ that is supported on $\varphi_0(S)$. {This approach is treated in Subsection~\ref{subsec:closedProcess} below.} In this approach, the topology of $\varphi_0(S)$ is different, but the corresponding Borel sigma-field is not changed. This definition is used in the literature to define local finiteness of measures on $\varphi_0(X)$, but continuity properties are treated by the first definition.}}. The discussion below allows one to define a \defstyle{particle process in a random environment} as well.

To apply the framework,
let $\mathcal D_0$ (resp. $\mathcal D_1$) be the space of pointed \bcm s $(S,o)$ equipped with some $a\in \varphi_0(S)$ (resp. $a\in\varphi_1(S)$).
For all compact metric spaces $X$, 
let $\tau(X)$ be the set of finite subsets of $\mathcal K(X)$ equipped with the inclusion partial order and the Prokhorov metric. 
Then, a truncation functor $\tau^t$ is defined as follows: For every isometric embedding $f:Y\to X$ and $a\in\tau(X)$, let $\tau^t_f(a):=\{f^{-1}(K): K\in a, K\subseteq f(Y) \}$.	
Note that this definition of the pair $\tau$ and $\tau^t$ is exactly the composition of the two functors of compact subsets and finite subsets (Example~\ref{ex:compose}). Therefore, all of the assumptions of Section~\ref{sec:noncompact} are satisfied except that $\tau(X)$ is not complete (verification of Assumptions~\ref{assump:subsetLemma1} and~\ref{assump:subsetLemma2} should be done separately and is left to the reader).
In addition, the extension defined in Definition~\ref{def:phi} coincides with $\varphi_1$. So, \eqref{eq:ghf-noncompact} defines a metric on $\mathcal D = \mathcal D_1$ which makes it separable.

The completion of $\mathcal D_1$ is the space $\mathcal D_2$ of pointed \bcm s $(S,o)$ equipped with a discrete multi-set in $\mathcal K(S)$. One can proceed similarly to show that $\mathcal D_2$ is a Polish space. Also,\unwritten{To show that it is a Borel subset, consider the condition that at most $M$ subsets intersect $\oball{r}{o}$.} it can be seen that $\mathcal D_2$ contains $\mathcal D_1$ and $\mathcal D_0$ as  Borel subsets. This allows one to define random elements in $\mathcal D_0$ or $\mathcal D_1$, as desired.

One could also equip $\tau(X)$ with the Hausdorff metric, but the topology would become coarser. 
Here, the completion of $\tau(X)$ is $\mathcal K(\mathcal K(X))\cup\{\emptyset\}$, which is a composition of two functors as in Example~\ref{ex:compose} and satisfies all of the assumptions. Also, the completion of $\mathcal D_1$ is the space of pointed \bcm s $(X,o)$ equipped with a closed subset of $\mathcal K(X)$.

\subsection{Processes of Closed Subsets and Measures}
\label{subsec:closedProcess}
%
%

For a metric space $S$, let $\mathcal F(S)$ (resp. $\mathcal K(S)$) denote the set of closed (resp. compact and nonempty) subsets of $S$.

\subsubsection{Point Processes of Closed Subsets} 
\label{subsec:PPofClosed}	
%

Given a \bcm{} $S$, a point process in $\mathcal F(S)\setminus\{\emptyset\}$ can be called a \textit{point process of closed subsets of $S$} \new{(note that every compact subset of $S$ intersects at most finitely many elements of the point process)}. Examples of such processes are \textit{line processes} and \textit{hyperplane processes} in $\mathbb R^d$ (see e.g., \cite{bookScWe08}). Here, it is shown that this gives an instance of the framework of Section~\ref{sec:noncompact}. As before, this allows one to let $(S,o)$ be random as well. 

\ali{Let $\mathcal D_0$ be the space of pointed \bcm s $S$ equipped with a discrete subset of $\mathcal F(S)\setminus\{\emptyset\}$. To use the framework,}  for compact $X$, let $\tau(X)$ be the set of finite multi-sets in $\mathcal K(X)$ equipped with the {Prokhorov} metric.  \ali{This is just the functor $\tau_2$ in Subsection~\ref{subsec:particle}, but the following truncation functor and partial order make the story different}. For every isometric embedding $f:Y\to X$ and $a\in\tau(X)$, let 
$
\tau^t_f(a):=\{f^{-1}(K): K\in a, K\cap f(Y)\neq\emptyset \}.
$
Also, for $a,a'\in\tau(X)$, define $a'\leq a$ if there exists an injective function $h:a'\to a$ such that $\forall K\in a': K\subseteq h(K)$. 
It can be seen that these definitions satisfy the assumptions of Section~\ref{sec:noncompact}. 
\ali{The difficult part is to prove Assumptions~\ref{assump:subsetLemma1} and~\ref{assump:subsetLemma2}, but the proofs are omitted since they are identical to the proof of Lemma~\ref{lem:RMonClosed} below, except that all measures in the proof should be integer-valued (see Corollary~\Iref{cor:strassen-integer} of~\cite{Kh19ghp} for the integer-valued version of the generalized Strassen's theorem). Therefore, the results of Section~\ref{sec:noncompact} imply that the corresponding space $\mathcal D$ is complete and separable.

Here, $\mathcal D$ is the space of pointed \bcm s $S$ equipped with a discrete multi-set in $\mathcal F(S)\setminus\{\emptyset\}$. It can be seen that $\mathcal D$ contains $\mathcal D_0$ as a Borel subset. This enables one to define a (simple or non-simple) \textit{point process of closed subsets in a random environment}.}

\begin{remark}
The use of multi-sets is necessary due to the nature of the truncation functor. Also, one cannot equip $\tau(X)$ with the Hausdorff metric since Assumption~\ref{assump:subsetLemma1} would not hold. 
\end{remark}

\subsubsection{Closed Subsets of Closed Subsets} 
\label{subsec:RCofClosed}
For compact $X$, let $\tau_0(X):=\mathcal K(\mathcal K(X))$. Note that $\tau_0$ is a composition of functors as in Subsection~\ref{subsec:compos}. So the corresponding metric space $\mathcal C_{\tau_0}$  \unwritten{1. Example: BW in BCRT\\ 2. Question: What is the dual of a BW in a real tree? Is it again a BW?} is complete and separable. This allows one to define a random compact metric space $X$ equipped with a random element in $\mathcal K(\mathcal K(X))$. The same hold for the functor $\mathcal K(\tau^{(s)}(X))$. 

%

For \bcm s, composition of functors (Example~\ref{ex:compose}) gives $\varphi(X)=\mathcal F(\mathcal K(X))$ as mentioned in Subsection~\ref{subsec:particle}.
Here, we would like to use a different truncation to obtain the space $\mathcal D_0$ of pointed \bcm s $(X,o)$ equipped with a closed subset of $\mathcal F(X)\setminus\{\emptyset\}$. 
Unfortunately, it seems that $\mathcal D_0$ cannot be obtained by the framework of Section~\ref{sec:noncompact}\footnote{It seems that there is no useful partial order on $\tau_0(X):=\mathcal F(\mathcal K(X))$ here.}. 
In the following, we use the framework for a specific subset of $\mathcal D_0$.

For compact $X$, let $\tau(X)$ be the set of elements $a\in\mathcal F(\mathcal K(X))$ which are \textit{lower sets}; i.e., for every $K\in a$, every closed subset of $K$ belongs to $a$. Equip $\tau(X)$ with the Hausdorff extended metric and the inclusion partial order. For every isometric embedding $f:Y\to X$, define the truncation by $\tau^t_f(a):=\{f^{-1}(K): K\in a\}\setminus\{\emptyset\}$. 
It is easy to see that all of the assumptions of Section~\ref{sec:noncompact} are satisfied. For brevity, we only prove the following.

\begin{lemma}
The above definitions satisfy Assumptions~\ref{assump:subsetLemma1} and~\ref{assump:subsetLemma2}.
\end{lemma}
{
\begin{proof}
	Assume $f:X\to Z$, $g:Y\to Z$, $a_X\in \tau(X)$ and $a_Y\in \tau(Y)$ are such that $\hausdorff(f(X),g(Y))\leq\epsilon$ and $\hausdorff(\tau_f(a_X),\tau_g(a_Y))\leq\epsilon$. Assume $X'\subseteq X$ and $a'_X\in \tau(X')$ is such that $a'_X\subseteq a_X$ (if one regards $a'_X$ as an element of $\tau(X)$ by an abuse of notation). Let $Y':=N_{\epsilon}(X)\cap Y$ and $a'_Y:= \{K\in a_Y: \exists K'\in a'_X: \hausdorff(f(K'),g(K))\leq\epsilon \}$. It is easy to see that $a'_Y$ is in $\tau(Y')$ and satisfies the assumptions.
	%
\end{proof}
}

Here, $\mathcal D$ is the space of pointed \bcm s $(X,o)$ equipped with a lower set $a\in \mathcal F(\mathcal F(X)\setminus\{\emptyset\})$. So the results of Section~\ref{sec:noncompact} define a metric on $\mathcal D$ which make it a complete separable metric space.

\unwritten{
\begin{remark}
	\label{rem:F(F(X))}
	The problem with $\mathcal D_0$ is how to define a partial order on $\tau_0(X):=\mathcal F(\mathcal K(X))$. A naive candidate is $a_1\leq a_2$ when $\forall K_1\in a_1: \exists K_2\in a_2: K_1\subseteq K_2$,
	but it is not anti-symmetric. Note that under the equivalence relation $a_1\leq a_2$ and $a_2\leq a_1$, every element of $\tau_0(X)$ is equivalent to a closed lower set. This suggests to restrict attention to lower sets as above. Note however that the truncation can be defined for every $a\in \tau_0(X)$ to be the closure of $\{f^{-1}(K): K\in a\}\setminus\{\emptyset\}$. The same can be defined when $X$ is boundedly-compact and $a\in\mathcal F(\mathcal F(X)\setminus\{\emptyset\})$. One can ask whether the integral formula~\eqref{eq:integral} defines a useful metric on $\mathcal D_0$ using this truncation or not.
\end{remark}
}

\subsubsection{Measures on Closed Subsets}
\label{subsec:RMonClosed}

Here, we define a complete metric on the space $\mathcal D_f$ (resp. $\mathcal D_{lf}$) of pointed \bcm s $X$ equipped with a finite measure on $\mathcal F(X)$ (resp. a measure on $\mathcal F(X)\setminus\{\emptyset\}$ that is finite on compact subsets of $F(X)\setminus\{\emptyset\}$). For compact $X$, let $\tau(X):=\mathcal M_f(\mathcal K(X))$, where $\mathcal M_f(\cdot)$ denotes the set of finite Borel measures. For $g:Y\to X$ and $a\in\tau(X)$, define the truncation of $a$ as follows. Let $L:=\{K\in \mathcal K(X): K\cap g(Y)\neq\emptyset \}$ and let $\tau^t_g(a)$ be the push-forward of $\restrict{a}{L}$ under the function $K\mapsto g^{-1}(K)$, which is a function from $L$ to $\mathcal K(Y)$. Let the partial order on $\tau(X)$ be $a'\preccurlyeq a$ if there exists a 
Borel measure $\alpha$ on $\mathcal K(X)^2$ supported on $\{(K_1,K_2):K_1\subseteq K_2\}$ such that $\pi_{1*}\alpha=a'$ and $\pi_{2*}\alpha\leq a$, where $\pi_i$ denotes the projection onto the $i$th component. These definitions are different from the composition of functors in Example~\ref{ex:compose}. The following lemma proves that they satisfy Assumptions~\ref{assump:subsetLemma1} and~\ref{assump:subsetLemma2}. It is easy to verify the rest of the assumptions of Section~\ref{sec:noncompact} and to see that $\mathcal D=\mathcal D_{lf}$. 
Therefore, the results of Section~\ref{sec:noncompact} define a metric on $\mathcal D_{lf}$ that make it a complete separable metric space.

For the space $\mathcal D_f$, it is enough to modify the above definitions by $\tau'(X)=\mathcal F(\mathcal K(X)\cup\{\emptyset\})$ and letting the truncation of $a$ be the push-forward of $a$ (not $\restrict{a}{L}$) under the map $K\mapsto g^{-1}(K)$. Similar arguments can be used to deduce that $\mathcal D_f$ is a complete separable metric space (the proof is in fact simpler and will be generalized in Example~\ref{ex:RMonTau} below). Note that $\mathcal D_f$ induces a finer topology on $\mathcal D_f\cap\mathcal D_{lf}$ since mass cannot escape to infinity.


\begin{lemma}
\label{lem:RMonClosed}
The above definitions satisfy Assumptions~\ref{assump:subsetLemma1} and~\ref{assump:subsetLemma2}.
\end{lemma}

\begin{proof}
Assume $X,Y\subseteq Z$ are compact, $a\in \tau(X)$ and $b\in \tau(Y)$. One can regard $a$ and $b$ as elements of $\mathcal M_f(\mathcal K(Z))=\tau(Z)$. Assume $\hausdorff(X,Y)\leq \epsilon$ and $\prokhorov(a,b)\leq \epsilon$. Assume $X'\subseteq X$ and $a'\in \tau(X')$ such that $a'\preccurlyeq a$ (when regarding $a'$ as an element of $\tau(X)$). By definition, there exists a measure $\alpha$ supported on $\{(K_1,K_2):K_1\subseteq K_2\}$ such that $\pi_{1*}\alpha=a'$ and $\pi_{2*}\alpha\leq a$. Also, since $\prokhorov(a,b)\leq \epsilon$, Theorem~\Iref{thm:strassen} of~\cite{Kh19ghp} implies that there exists a measure $\beta$ on $\mathcal K(Z)^2$ such that $\pi_{1*}\beta\leq a$, $\pi_{2*}\beta\leq b$ and $||a-\pi_{1*}\beta|| + ||b-\pi_{2*}\beta|| + \beta(R^c)\leq\epsilon$, where $R:=\{(K_1,K_2):\hausdorff(K_1,K_2)\leq\epsilon\}$. Let $a_1:=\pi_{2*}\alpha\wedge\pi_{1*}\beta\leq a$. It can be seen that there exists $a'_1\leq a'$ and $\alpha_1\leq \alpha$ such that $\pi_{1*}\alpha_1=a'_1$ and $\pi_{2*}\alpha_1=a_1$ (for this, disintegrate $\alpha$ with respect to the second projection and then integrate it again with respect to $a_1$). Similarly, there exists $\beta_1\leq \beta$ such that $\pi_{1*}\beta_1=a_1$. Now, since $\pi_{2*}\alpha_1=\pi_{1*}\beta_1$, they can be extended to a measure $\gamma$ on $\mathcal K(X)^3$ which is \textit{almost} supported on $\{(K_1,K_2,K_3): K_1\subseteq K_2, \hausdorff(K_2,K_3)\leq\epsilon\}$ (the mass of $\gamma$ out of this set is at most $\beta_1(R^c)$). Let $\delta$ be the push-forward of $\gamma$ under the map $m(K_1,K_2,K_3):=(K_1,N_{\epsilon}(K_1)\cap K_3)$ and let $b':=\pi_{2*}\delta$. One has
\begin{eqnarray*}
	||a'-\pi_{1*}\delta|| + ||b'-\pi_{2*}\delta|| + \delta(R^c)
	\leq  ||a-\pi_{1*}\beta|| + 0 + \beta(R^c)
	\leq \epsilon.
\end{eqnarray*}
Therefore, Theorem~\Iref{thm:strassen} of~\cite{Kh19ghp} implies that $\prokhorov(a',b')\leq\epsilon$.
This implies that Assumption~\ref{assump:subsetLemma1} holds. 

For Assumption~\ref{assump:subsetLemma2}, one should modify the above construction as follows. Let $Y':=N_{\epsilon}(X)\cap Y$. 
Assuming $\restrict{a}{\cball{r}{o}}\preccurlyeq a'\preccurlyeq a$, one gets that $\restrict{a'}{\cball{r}{o}}=\restrict{a}{\cball{r}{o}}$. This implies that $\alpha$ should be supported on $\{(K_1,K_2): K_2\cap\cball{r}{o}\subseteq K_1\subseteq K_2 \}$ (more details are provided in Lemma~\ref{lem:RMonTau} below).
Also, $\gamma$ will be \textit{almost} supported on $R':=\{(K_1,K_2,K_3): K_2\cap\cball{r}{o}\subseteq K_1\subseteq K_2, \hausdorff(K_2,K_3)\leq\epsilon \}$. Then, outside $R'$, change the map $m$ by $m(K_1,K_2,K_3):=(K_1,K_3\cap Y')$. This way, the second coordinate of $m(K_1,K_2,K_3)$ always contains $K_3\cap\cball{r-2\epsilon}{o}$. In the end, let $b':=\pi_{2*}\delta+\restrict{(b-\pi_{2*}\beta_1)}{\cball{r-2\epsilon}{o}}$. 
First, note that the property of $m$ implies that $\restrict{(\pi_{2*}\beta_1)}{\cball{r-2\epsilon}{o}}\leq \pi_{2*}\delta$. As a result, $\restrict{b}{\cball{r-2\epsilon}{o}}\preccurlyeq b'$. So it remains to prove that $\prokhorov(a',b')\leq \epsilon$.

Since $\restrict{a}{\cball{r}{o}}=\restrict{a'}{\cball{r}{o}}$, one gets that $a-\pi_{2*}\alpha$ is supported on $\{K: K\cap \cball{r}{o}=\emptyset \}$. So the definition of $R$ implies that $\pi_{2*}\beta-\pi_{2*}\beta_1$ is almost supported on $\{K:K\cap\cball{r-2\epsilon}{o}=\emptyset\}$ and the mass outside this set is at most $\beta(R^c)-\beta_1(R^c)$. The part supported on the set dies when restricting to $\cball{r-2\epsilon}{o}$, and hence,  $||\restrict{(\pi_{2*}\beta-\pi_{2*}\beta_1)}{\cball{r-2\epsilon}{o}}||\leq \beta(R^c)-\beta_1(R^c)$. So, one obtains that $||\restrict{(b-\pi_{2*}\beta_1)}{\cball{r-2\epsilon}{o}}||\leq ||b-\pi_{2*}\beta|| + \beta(R^c)-\beta_1(R^c)$. Therefore,
\begin{eqnarray*}
	&&||a'-\pi_{1*}\delta|| + ||b'-\pi_{2*}\delta|| + \delta(R^c)\\
	&\leq & ||a-\pi_{1*}\beta|| + ||b-\pi_{2*}\beta|| + \beta(R^c)-\beta_1(R^c) + \delta(R^c)\\
	&\leq& \epsilon.
\end{eqnarray*}
The last inequality holds because $\delta(R^c)\leq\beta_1(R^c)$. So the claim $\prokhorov(a',b')\leq \epsilon$ is implied by Theorem~\Iref{thm:strassen} of~\cite{Kh19ghp} and the proof is completed.
\end{proof}

\del{\begin{remark}
	\label{rem:closedProcess}
	If one equips $\tau(X)$ with the Hausdorff metric in the above definition, then Assumptions~\ref{assump:subsetLemma1} and~\ref{assump:subsetLemma2} do not hold. However, the author has not checked whether~\eqref{eq:variantMetric} can be used to define a metric on $\mathcal D$ and to make it a Polish space or not.
\end{remark}
}

\subsubsection{Two-Level Measures and Measures on $\tau_0(X)$}
\label{subsec:two-level}
In~\cite{Me18}, it is shown that the set of compact metric spaces $X$ equipped with a \textit{two-level measure}; i.e., a finite measure on the set of finite measures on $X$, is a Polish space. 
\ali{The metric considered in~\cite{Me18} (Definition~4.1 of~\cite{Me18}) is a GP-type metric. If one considers a GHP-type metric in a similar manner, then the result will be identical to the composition functor $\tau(X):=\mathcal M_f(\mathcal M_f(X))$
and Polishness is implied by Subsection~\ref{subsec:compos}. Note that $\mathcal M_f(X)$ is not compact, but one can still use Remark~\ref{rem:composition} for the composition.}

One should be careful for extending this to \bcm s. The composition of functors in Example~\ref{ex:compose} does not work here, since Assumption~\ref{assump:subsetLemma1} is not satisfied. The reason is that a small perturbation of a measure might change the support of the measure significantly. If it worked, then $\varphi(X)$ would be the set of measures on the set of compactly-supported measures on $X$ (with a specific locally-finiteness condition). It is not clear whether the integral formula~\eqref{eq:integral} is useful in this case or not.

To extend to \bcm s, one can change the partial order and the truncation maps according to the general construction in Example~\ref{ex:RMonTau} below. Then, when $X$ is boundedly-compact, $\varphi(X)=\mathcal M_f(\mathcal M_{bf}(X))$, where $\mathcal M_{bf}(X)$ is the set of boundedly-finite measures on $X$. In addition, all assumptions are satisfied and the corresponding space $\mathcal D$ is complete and separable.

\begin{example}[Measures on $\tau_0(X)$]
\label{ex:RMonTau}
Let $\tau_0$ and $\tau^t_0$ be functors that satisfy the assumptions of Section~3. Let $\tau(X):=\mathcal M_f(\tau_0(X))$ (or the set of probability measures on $\tau_0(X)$) equipped with the Prokhorov metric. This is useful for comparing two metric spaces equipped with \textit{random} additional structure (e.g., in Subsection~\ref{subsec:cadlagprocess}). Define the partial order on $\tau(X)$ by $a_1\preccurlyeq a_2$ when there is a coupling of $a_1$ and $a_2$ supported on $\{(K_1,K_2)\in\tau_0(X)^2: K_1\leq K_2\}$. For an isometry $g:Y\to X$ and $a\in\tau(X)$, define the truncation $\tau^t_g(a)$ to be the pushforward of $a$ under the map $(\tau^t_0)_g:\tau_0(X)\to\tau_0(Y)$, which makes sense by assuming the following (note that these definitions are different from the composition of functors in Example~\ref{ex:compose}).

\begin{assumption}
	\label{assump:RMonTau}
	Assume the following further conditions: (i) $\tau_0(X)$ is always separable, (ii) the relation $\leq$ is a closed subset of $\tau_0(X)^2$, (iii) the truncation maps $(\tau_0^t)_g$ are Borel measurable and (iv) the choice of $\mathcal Y'$ in Assumptions~\ref{assump:subsetLemma1} and~\ref{assump:subsetLemma2} can be taken to be a Borel measurable function of $\mathcal X'$.
\end{assumption}

\ali{
	\begin{proposition}
		In Example~\ref{ex:RMonTau}, all of the assumptions of Section~\ref{sec:noncompact} are satisfied. For a \bcm{} $X$, the set of additional structures is $\varphi(X)=\mathcal M_f(\varphi_0(X))$. Hence, under the assumptions of Theorem~\ref{thm:functor-polish-noncompact}, the corresponding space $\mathcal D$ is a metric space and is Polish.
	\end{proposition}
	
	This result can be proved similarly to Subsection~\ref{subsec:RMonClosed}. Here, we only prove the following lemma and the rest is skipped.
}

\end{example}

\begin{lemma}
\label{lem:RMonTau}
The above definitions satisfy Assumptions~\ref{assump:subsetLemma1} and~\ref{assump:subsetLemma2}.
\end{lemma}
\begin{proof}
The proof is almost identical to that of Lemma~\ref{lem:RMonClosed}. Here, we only highlight the differences for brevity. 
Here, $\alpha$ is a coupling of $a'$ and $a$. Hence, the situation is simpler:  $\pi_{2*}\alpha=a$, $a_1=\pi_{1*}\beta$ and $\beta_1=\beta$. The measure $\gamma$ is almost supported on the set $R$ of tuples $(K_1,K_2,K_3)\in\tau_0(Z)^3$ such that $K_1\in\tau_0(X')$, $K_1\leq K_2$ and $d(K_2,K_3)\leq\epsilon$. In addition, if $\restrict{a}{\cball{r}{o}}\leq a'$, one might add the condition $\restrict{K_2}{\cball{r}{o}}\leq K_1$ (this is because $\alpha$ induces a coupling of $\restrict{a}{\cball{r}{o}}=\restrict{a'}{\cball{r}{o}}$ and itself that is supported on $\{(K_1,K_2):K_1\leq K_2\}$, and hence, Lemma~\ref{lem:coupling} below implies that the latter is the trivial coupling). To define the map $m$ on $R$, one needs to use Assumption~\ref{assump:RMonTau} to find $K_4:=m(K_1,K_2,K_3)\leq K_3$ such that $K_4\in\tau_0(Y')$ and $d(K_1,K_4)\leq d(K_2,K_3)\leq \epsilon$. In addition, if $\restrict{K_2}{\cball{r}{o}}\leq K_1$, then $\restrict{K_3}{\cball{r-2\epsilon}{o}}\leq K_4$. The rest of the proof is identical to that of Lemma~\ref{lem:RMonClosed}.
%
\end{proof}

\begin{lemma}
\label{lem:coupling}
Let $E$ be a separable metric space and $\leq$ be a partial order on $E$ which is a closed subset of $E^2$. Assume $X$ and $Y$ are random elements of $E$ that have the same distribution and $X\leq Y$ a.s. Then, $X=Y$ a.s.
\end{lemma}
\begin{proof}
If $A$ is a Borel lower set in $E$, then $\myprob{X\in A}\geq \myprob{Y\in A}=\myprob{X\in A}$. So, $\myprob{X\in A, Y\not\in A}=0$. Therefore, it is enough to find countably many closed lower sets in $E$ that separate every pair $e_1<e_2$. Let $x_1,x_2,\ldots$ be a countable dense set in $E$. Let $A_{m,n}$ be the closure of the set of elements below $\cball{1/n}{x_m}$. Let $e_1<e_2$ be an arbitrary pair. For every $n$, find $m=m(n)$ such that $d(e_1,x_m)\leq 1/n$. It is easy to deduce that one of these sets $A_{m(n),n}$ separates $e_1$ and $e_2$ (otherwise, one finds sequences $f_n\leq g_n$ such that $f_n\to e_2$ and $g_n\to e_1$, which contradicts closedness of the relation). This completes the proof.
\end{proof}

\subsection{Isometries}
\label{subsec:isometries}
\ali{In this subsection, we apply the framework to metric spaces equipped with an isometry or a group of isometries.}
\subsubsection{Convergence of Isometries}
\label{subsec:isometry}
Gromov defined convergence of pointed \bcm s $(X,o)$ equipped with an isometry $h:X\to X$ (Section~6 of~\cite{Gr81}). This is used to prove that the Gromov-Hausdorff limit of a sequence of homogeneous spaces (i.e., the isometry group acts transitively) is homogeneous. Note that the set $\varphi(X)$ of such isometries is not a functor, since an isometric embedding $f:X\to Y$ does not induce a natural map from $\varphi(X)$ to $\varphi(Y)$. However, by identifying every isometry $h$ with its graph, which is a closed subset of $X\times X$, isometries are special cases of 2-fold marked closed subsets. 
This way, it can be seen that the notion of convergence is the same as that of Example~\ref{ex:measures4}. In addition,
Example~\ref{ex:measures4} provides a complete metrization of Gromov's notion of convergence.

\subsubsection{Equivariant Hausdorff Convergence}
\label{subsec:equivariantGH}
Fukaya~\cite{Fu86convergence} 
defined a metric on the space $\mathcal T$ of triples $(X,\Gamma,p)$, where $(X,p)$ is a pointed \bcm{} and $\Gamma$ is a closed subgroup of the isometries of $X$. This is called the \textit{equivariant Hausdorff distance} in~\cite{Fu86convergence}. For brevity, we only mention the notion of convergence in $\mathcal T$ in an equivalent form (after some correction\footnote{In fact, the definition in~\cite{Fu86convergence} is flawed and does not satisfy the triangle inequality (see Remark~\ref{rem:flaws}). The issue still exists in the modifications of the metric given in~\cite{FuYa92fundamental} and~\cite{Fu88boundary}. It is corrected here by the same idea as~\eqref{eq:beingclose-khodam}.
}) and show that this convergence is a special case of the framework of this paper. In addition, a completeness and separability result for $\mathcal T$ is obtained.

Let $\mathcal X_n:=(X_n,\Gamma_n,p_n)$ be a sequence in $\mathcal T$. Call $(\mathcal X_n)_n$ convergent to $\mathcal X:=(X,\Gamma,p)$ 
if and only if they can be isometrically embedded in a common \bcm{} $Z$ such that, after the embedding, (i) $p_n$ converges to $p$, (ii) $X_n$ converges to $X$ in the Fell topology, (iii) For every $\epsilon>0$, for large enough $n$ and every $\gamma_n\in\Gamma_n$, the graph of $\gamma_n$ (regarded as a closed subset of $Z^2$) is $\epsilon$-close to the graph of some $\gamma\in\Gamma$ in some metrization of the Fell topology (e.g., \eqref{eq:integral} or Example~\ref{ex:Fell-metrization}) and (iv) the same holds by swapping $X_n$ and $X$. The last two conditions can be summarized as $\hausdorff(\Gamma_n,\Gamma)\to 0$, where $\Gamma_n$ and $\Gamma$ are regarded as closed subsets of $\mathcal F(Z^2)$; i.e., $\Gamma_n$ is close to $\Gamma$ as elements in $\tau_0(Z):=\mathcal K(\mathcal F(Z^2))$. It is good to mention that for compact metric spaces, a stronger notion of convergence is obtained by requiring $Z$ to be compact and by equipping $\mathcal F(Z^2)$ with the Hausdorff metric.

To express this convergence in the framework of this paper, we proceed similarly to Subsection~\ref{subsec:RCofClosed}. As mentioned therein, we need to focus on the set $\tau(Z)$ of closed lower sets in $\mathcal F(Z^2)$, which is smaller than $\tau_0(Z)$. Similarly to Subsection~\ref{subsec:RCofClosed}, one can define the truncation and partial order and show that the corresponding space $\mathcal D$ is complete and separable. Now, for $\mathcal X_n$ and $\mathcal X$ mentioned above, 
define $a:=\bigcup_{\gamma\in\Gamma_n}\mathcal F(\mathrm{gr}(\gamma))\in\tau(X)$, where $\mathrm{gr}(\gamma)\subseteq X^2$ is the graph of $\gamma$. Note that $a$ is a closed lower set in $\mathcal F(X^2)$. Define $a_n\in\tau(X_n)$ similarly. 

\begin{theorem}
$(X_n,\Gamma_n,p_n)$ converges to $(X,\Gamma,p)$ in the equivariant Hausdorff metric if and only if $(X_n,p_n;a_n)$ converges to $(X,p;a)$ in $\mathcal D$. In addition, the space $\mathcal T$ is complete and separable under the metric induced from $\mathcal D$.
\end{theorem}
\begin{proof}[Sketch of the proof.]
Assume $\mathcal X_n\to\mathcal X$ in $\mathcal T$.
The definition of convergence in $\mathcal T$, mentioned above, implies that $\mathcal X'_n:=(X_n,p_n;a_n)$ converges to $\mathcal X':=(X,p;a)$ in $\mathcal D$. Conversely, assume $\mathcal X'_n\to \mathcal X'$. By Lemma~\ref{lem:convergence2}, one can assume that $X_n$ and $X$ are subsets of a common \bcm{} $Z$ such that $p_n\to p$, $X_n\to X$ in the Fell topology and $a_n\to a$ as elements in $\tau(Z)$. So, for large $n$, the graph of every $\gamma_n\in\Gamma_n$ is close to some subset $S$ of the graph of some $\gamma\in\Gamma$. One has $S=\mathrm{gr}(\restrict{\gamma}{D})$ for some $D\subseteq X$. 
Since $\gamma_n$ is defined on the whole $X_n$ and $X_n$ is close to $X$, one obtains that $D$ is close to $X$ in the Fell topology. This implies that $\gamma_n$ is close to $\gamma$ as well. Similarly, every $\gamma\in\Gamma$ is close to some $\gamma_n\in\Gamma_n$. Therefore, $\mathcal X_n\to\mathcal X$ in $\mathcal T$. The second part of the claim is also implied by the fact that $\mathcal T$ is a closed subset of $\mathcal D$, which is straightforward to prove.
\end{proof}

\subsubsection{Convergence of Metric $G$-Spaces}
\label{subsec:G-spaces}
Given a group $G$, let $\mathcal G_1$ be the space of pointed \bcm s $(X,o)$ equipped with an action of $G$ on $X$ by isometries, which are called \textit{metric $G$-spaces}. If $G$ is a $\sigma$-compact topological group, define $\mathcal G_2$ similarly by considering only continuous actions. The papers~\cite{Fu88boundary} and~\cite{MoPy14} define a metric on $\mathcal G_1$ and a topology on $\mathcal G_2$ respectively. In what follows, after some correction,\footnote{In fact, both definitions should be corrected according to Remark~\ref{rem:flaws}. The first does not satisfy the triangle inequality and the second requires convergence of every ball, which is not compatible with Gromov's notion of convergence of pointed metric spaces.} the topologies on $\mathcal G_1$ and $\mathcal G_2$ are expressed in the framework of this paper.

Let $\mathcal X_n:=(X_n,o_n; \pi_n)$, where $\pi_n$ is an action of $G$ on $X_n$ by isometries.
As in Subsection~\ref{subsec:isometry}, we identify every isometry of $X$ with its graph, which is a closed subset of $X^2$. Then, $\pi_n$ can be regarded as a function from $G$ to $\mathcal F(X_n^2)$. It can be seen that $\mathcal X_n$ converges to $\mathcal X:=(X,o;\pi)$ in the sense of~\cite{Fu88boundary} (resp. \cite{MoPy14}) if and only if they can be embedded in a common \bcm{} $Z$ such that, after the embedding, $o_n\to o$, $X_n\to X$ in the Fell topology and $\pi_n\to\pi$ uniformly (resp. uniformly on compact subsets of $G$). In the latter, $\pi_n$ and $\pi$ are regarded as functions from $G$ to $\mathcal F(Z^2)$ and the space $\mathcal F(Z^2)$ is equipped with the metrization of the Fell topology in Example~\ref{ex:Fell-metrization} (note that the choice of any root in Example~\ref{ex:Fell-metrization} leads to the same notion of convergence here). \unwritten{The latter is strictly weaker. For example, $\mathbb R$ acts an $\mathbb R^2$ by rotations by angle $\lambda_n x$, where $\lambda_n\to 0$. This helps for separability.}

%
%

To use the framework of this paper, when $X$ is compact, let $\tau(X)$ be the set of functions from $G$ to $\mathcal F(X^2)$ equipped with the sup metric (equip $\mathcal F(X^2)$ with the Hausdorff extended metric). 
If $a:G\to\mathcal F(X^2)$ and $f:Y\to X$, define the truncation of $a$ by truncating the sets $a(g)$ separately for all $g\in G$ (note that the truncation of an isometry of $X$ is not necessarily an isometry of $Y$). Also, define the partial order on $\tau(X)$ by $a_1\leq a_2$ if and only if $\forall g\in G: a_1(g)\subseteq a_2(g)$.
It can be seen that the assumptions of Section~\ref{sec:noncompact} are satisfied except separability of $\tau(X)$ and compactness of the cones. Also, the extension $\varphi(X)$ to \bcm s $X$ is the set of functions from $G$ to $\mathcal F(X^2)$ and the topology on $\varphi(X)$ is that of uniform convergence with respect to the metrization of the Fell topology in Example~\ref{ex:Fell-metrization}. Therefore, Lemma~\ref{lem:topology} implies that the topology of $\mathcal G_1$ is just the restriction of the topology of the corresponding metric space $\mathcal D$. 

For $\mathcal G_2$, one can proceed similarly by equipping $\tau(X)$ with a suitable metrization of uniform convergence on compact subsets of $G$. For instance, let $G_1\subseteq G_2\subseteq \cdots$ be an exhaustion of $G$ by compact subsets. For (not necessarily continuous) functions $\pi_i:G\to \mathcal F(X^2)$ ($i=1,2$), let $d(\pi_1,\pi_2):=\inf\{\epsilon\leq 1: \dsup(\restrict{\pi_1}{G_{\lfloor 1/\epsilon\rfloor}}, \restrict{\pi_2}{G_{\lfloor 1/\epsilon\rfloor}})\leq\epsilon \}$ similarly to~\eqref{eq:curves-metric}. Similar claims hold and the corresponding space $\mathcal D$ extends $\mathcal G_2$ and its topology. However, it is not clear whether $\mathcal G_1$ and $\mathcal G_2$ are separable (note that we cannot restrict attention to continuous functions $\pi$ in the above approach, since the truncation of a continuous function is not necessarily continuous).

\subsection{C\`adl\`ag Curves and Processes}
\label{subsec:cadlagprocess}
Let $X_n$ be a metric space and $\eta_n$ be a c\`adl\`ag curve in $X_n$. In Definition~1.3 of~\cite{AtLoWi17}, it is said that $(X_n,\eta_n)$ converges to $(X,\eta)$ if they can be embedded in a common metric space such that the images of the curves $\eta_n$ converges to the image of $\eta$ in the Skorokhod metric. If $\eta_n$ is a random curve in $X_n$, a similar notion is defined by requiring that the images of the random curves converge weakly. This is used in~\cite{AtLoWi17} for studying limits of random walks on trees. However, \cite{AtLoWi17} does not define a metrization of this notion of convergence.

Here, we concentrate on \bcm s \ali{and we add the natural condition that the image of $X_n$ converges to the image of $X$ after embedding in the common metric space. Also, we consider $\eta(0)$ as the root}. For deterministic curves, this notion of convergence is the same as that of Subsection~\ref{subsec:cadlag2} (see Lemma~\ref{lem:convergence2}). 
The case of random c\`adl\`ag curves is also obtained by composition with the functor of measures, which is discussed in Example~\ref{ex:RMonTau} (verification of Assumption~\ref{assump:RMonTau} is skipped). Therefore, the framework of Section~\ref{sec:noncompact} provides a metrization of these notions of convergence. Note that these metrizations are not complete (the completion can be obtained by adding curves that blow up at a finite time) but the space is still Polish, as explained in Subsection~\ref{subsec:cadlag2}. If one restricts attention to compact metric spaces only, then one can use the simpler setting of Section~\ref{sec:compact} (discussed in Subsection~\ref{subsec:cadlag}) and the metrization is complete.

\subsection{Ends}
\label{subsec:ends}
Ends are defined in~\cite{Fr31} for all topological spaces and graphs and, heuristically, are the \textit{points at infinity}. 
For simplicity, we only consider the class $\mathfrak L$ of \bcm s $X$ such that $X$ is either a simple connected graph (equipped with the graph distance metric) or it is \textit{locally-connected}\del{; i.e., every  neighborhood of every point in $X$ contains another neighborhood which is connected}. In this case, given a point $o$ of $X$, an end $\xi$ of $X$ can be uniquely described by a sequence of closed sets $\xi_1\supseteq \xi_2\supseteq\cdots$, where $\xi_n$ is a connected component of $X\setminus\oball{n}{o}$ for each $n$, where $\oball{n}{o}$ is the \textit{open ball} of radius $n$ centered at $o$.\del{ Ends of graphs are defined similarly by using the notion of connectedness in graphs.} 
So the set $\mathcal L$ of (equivalence classes of) tuples $(X,o,\xi)$, where $X\in\mathfrak L$, $o\in X$ and $\xi$ is an end of $X$, can be regarded as a subset of the space $\mathcal D$ defined in Section~\ref{sec:noncompact}. 
Here, the additional structure is a sequence of closed subsets (see Example~\ref{ex:multiple-noncompact}).
It can be seen that $\mathcal L$ is a closed subset of $\mathcal D$. Therefore, the metric on $\mathcal D$ can be used to make $\mathcal L$ a Polish space. This allows one to define random metric spaces equipped with an end.
A similar idea can be used to define random metric spaces equipped with a closed subset of ends, which is skipped here.

{\section*{Acknowledgements}
The major part of this paper was written when the author was affiliated with Tarbiat Modares University\footnote{Department of Applied Mathematics, Faculty of Mathematical Sciences, Tarbiat Modares University, P.O. Box 4115-134, Tehran, Iran.} and IPM.\footnote{Institute for Research in Fundamental Sciences (IPM), Tehran, Iran.} This research was in part supported by a grant from IPM (No. 98490118). The revision of the paper was completed when the author was affiliated with Inria. This work was supported by the ERC NEMO grant, under the European Union's Horizon 2020 research and innovation programme, grant agreement number 788851 to INRIA.
}

\bibliography{bib} 
\bibliographystyle{plain}

\end{document}